\documentclass[11pt,a4paper]{amsart}
\DeclareFontFamily{U}{mathx}{\hyphenchar\font45}
\DeclareFontShape{U}{mathx}{m}{n}{
      <5> <6> <7> <8> <9> <10>
      <10.95> <12> <14.4> <17.28> <20.74> <24.88>
      mathx10
      }{}
\DeclareSymbolFont{mathx}{U}{mathx}{m}{n}
\DeclareFontSubstitution{U}{mathx}{m}{n}
\DeclareMathAccent{\widecheck}{0}{mathx}{"71}
\DeclareMathAccent{\wideparen}{0}{mathx}{"75}		

\usepackage[applemac]{inputenc}					

\usepackage{lmodern}						
\usepackage[english]{babel}					
\usepackage[weather]{ifsym}	
\usepackage{mathrsfs}				

\usepackage{hyperref,enumitem,soul}

\hypersetup{
	colorlinks,
	linkcolor={red!50!black},
	citecolor={blue!50!black},
	urlcolor={blue!80!black}
}

\usepackage{tikz}
\usetikzlibrary{calc, svg.path, patterns}
\usepackage{todonotes}

\usepackage{amsmath,amsfonts,amssymb,amsthm}

\usepackage[left=3cm,right=3cm,top=3cm,bottom=3cm]{geometry}

\usepackage{mathtools}						

\def\refer#1{~\ref{#1}}
\def\refeq#1{~(\ref{#1})}
\def\ccite#1{~\cite{#1}}

\def\inte#1{
\displaystyle\mathop{#1\kern0pt}^\circ }

\def\minetage#1#2{\min_{\substack{{#1}\\{#2}}}}

\let\b=\beta

\let\lam=\lambda

\let\s=\sigma

\let\wt=\widetilde

\def\cA{{\mathcal A}}
\def\cB{{\mathcal B}}

\def\cG{{\mathcal G}}
\def\cH{{\mathcal H}}

\def\cJ{{\mathcal J}}

\def\cL{{\mathcal L}}
\def\cM{{\mathcal M}}

\def\cQ{{\mathcal Q}}
\def\cR{{\mathcal R}}
\def\cS{{\mathcal S}}

\def\cU{{\mathcal U}}

\def\cW{{\mathcal W}}

\def\virgp{\raise 2pt\hbox{,}}
\def\cdotpv{\raise 2pt\hbox{;}}

\def\im {\mathop{\rm Im}\nolimits}

\def\C{\mathop{\mathbb C\kern 0pt}\nolimits}
\def\DD{\mathop{\mathbb D\kern 0pt}\nolimits}
\def\EE{\mathop{{\mathbb E \kern 0pt}}\nolimits}
\def\K{\mathop{\mathbb K\kern 0pt}\nolimits}
\def\N{\mathop{\mathbb N\kern 0pt}\nolimits}
\def\Q{\mathop{\mathbb Q\kern 0pt}\nolimits}
\def\R{{\mathop{\mathbb R\kern 0pt}\nolimits}}
\def\SS{\mathop{\mathbb S\kern 0pt}\nolimits}
\def\ZZ{\mathop{\mathbb Z\kern 0pt}\nolimits}
\def\TT{\mathop{\mathbb T\kern 0pt}\nolimits}
\def\P{\mathop{\mathbb P\kern 0pt}\nolimits}
\def \H{{\mathop {\mathbb H\kern 0pt}\nolimits}}
\newcommand{\ds}{\displaystyle}

\newcommand{\Z}{{\ZZ}}
\def\BT{{\bf T}}
\def\BA{{\bf A}}
\def\BD{{\bf D}}

\DeclareMathOperator{\tr}{Tr}
\DeclareMathOperator{\re}{Re}

\newcommand{\beq}{\begin{equation}}
\newcommand{\eeq}{\end{equation}}
\newcommand{\ben}{\begin{eqnarray}}
\newcommand{\een}{\end{eqnarray}}
\newcommand{\beno}{\begin{eqnarray*}}
\newcommand{\eeno}{\end{eqnarray*}}
\newcommand{\bqs}{\begin{equation*}}
\newcommand{\eqs}{\end{equation*}}
\newcommand{\andf}{\quad\hbox{and}\quad}
\newcommand{\with}{\quad\hbox{with}\quad}

\def\equivH#1 {\buildrel\hbox{\tiny {$#1$}}\over \equiv}
\def\simH#1 {\buildrel\hbox{\footnotesize {$#1$}}\over \sim}

\newtheorem{definition}{Definition}[section]
\newtheorem{theorem}{Theorem} 
\newtheorem{lemma}{Lemma}[section]
\newtheorem{remark}{Remark}[section]
\newtheorem{cor}{Corollary}[section]
\newtheorem{proposition}{Proposition}[section]
\numberwithin{equation}{section}

\begin{document}
\title[The derivative nonlinear Schr\"odinger equation on the circle]
{Global well-posedness for the derivative nonlinear Schr\"odinger equation on the circle}

\author[H. Bahouri]{Hajer Bahouri}
\address[H. Bahouri]
{CNRS  \&  Sorbonne Universit\'e  \\
 Laboratoire Jacques-Louis Lions (LJLL) UMR  7598 \\
4, Place Jussieu\\
75005 Paris, France.}
\email{hajer.bahouri@sorbonne-universite.fr}
\author[G. Perelman]{Galina Perelman}
\address[G. Perelman]%
{Laboratoire d'Analyse et de Math{\'e}matiques Appliqu{\'e}es UMR 8050 \\
Universit\'e Paris-Est  Cr{\'e}teil\\
61, avenue du G{\'e}n{\'e}ral de Gaulle\\
94010 Cr{\'e}teil Cedex, France}
\email{galina.perelman@u-pec.fr}

\date{\today}

\begin{abstract}
{\sl We consider  the derivative nonlinear Schr\"odinger  equation on the circle, and prove that it is globally well-posed in the Sobolev space $H^1(\TT)$.}

\end{abstract}

\maketitle

\tableofcontents


\vskip 0.2cm


\section{Introduction}\label {intro}
\setcounter{equation}{0}
In this paper,  we consider  the derivative nonlinear Schr\"odinger equation (DNLS) on the circle: 
\begin{equation}\label {DNLS-T} 
iu_t +u_{xx} = -  i \partial_x(|u|^2 u), \quad x\in \TT= \R/ \Z,
\end{equation}
with initial data
\begin{equation}\label{id}u|_{t=0} = u_0 \in H^{s}(\TT). 
\end{equation}
The DNLS equation is a canonical dispersive equation arising  in a variety of physical contexts (see for instance\ccite{HKillipNtekoumeVisan2021, MOMT, mjohus} and the references therein). It is $L^2$ critical, being invariant under the scaling 
\begin{equation}
\label{scaling}
   u(t,x)\longrightarrow u_\mu(t,x) = \frac 1 {\sqrt{\mu}}  u(\frac t {\mu^2}, \frac x{\mu}), \quad \mu >0 \, ,
\end{equation}
 and it is known to be 
  completely integrable. 
  
  \medskip
  The well-posedness questions for the DNLS equation
have been extensively studied, both on the real line and on the circle.  
  Local well-posedness   in $H^s$, $s\geq \frac12$, was proved in the case of the real line
by Takaoka\ccite{HT},  and by Herr\ccite{Herr} in the periodic setting.
This range is optimal\footnote{The gap between the $s=\frac12$ threshold and the  critical scaling regularity can  be (almost) closed by leaving the $H^s$-scale and working in more general 
Fourier-Lebesgue spaces\ccite{DNY, AG, GH}.}  if one requires the solutions to be locally uniformly continuous with respect to initial data, see for instance\ccite{Biagioni, mosincat1, HT2}.

\medskip  The global  well-posedness issue is more challenging. As a completely integrable equation, DNLS 
 admits an infinite family of  polynomial conservation laws, including the conservation of  the mass, momentum and energy:\begin{equation}
\label{mass} M(u)=\int |u|^2  dx \, ,
\end{equation}
\begin{equation}
\label{momentum}  P(u)=  {\rm Im} \int \overline {u} u_x  dx+ \frac 1 2 \int |u|^4 dx \virgp
\end{equation}
\begin{equation}
\label{energy} E(u)= \int \Big(|u_x |^2 -\frac 3 2 {\rm Im} (|u|^2 u \overline {u}_x) + \frac 1 2  |u|^6\Big)dx \cdotp \end{equation}
To avoid any confusion,  throughout this paper, we will denote by $M_\R, P_\R, E_\R$  the mass, momentum and  energy on the real line, and by~$M_{\rm \mu \TT}, P_{\rm \mu \TT}, E_{\rm \mu \TT}$ their counterparts on the circle
 $\R/\mu \Z$. 
These three conservation laws are known to control the $H^1$-norm of the solutions provided  the mass is strictly less than $4\pi$. 
More precisely, on the real line,  Wu\ccite{W} and Guo-Wu\ccite{Guo} 
established the following inequality (see Lemma 2.2 in\ccite{Guo}) 
 \begin{equation}
\label{controlhomH1} \|u\|^2_{\dot H^{1}(\R)} \lesssim \frac   {P^2_\R(u)+ |E_\R(u)| } {\big(1- \frac 1 {2 \sqrt \pi} \|u\|_{L^{2}(\R)}\big)^2},\quad \forall \, u\in H^1(\R), \, \, \|u\|^2_{L^2(\R)}< 4 \pi, \end{equation}  which immediately  gives   global well-posedness of  DNLS in $ H^1(\R)$ under the restriction~$M_\R(u)<4\pi$.  This global well-posedness result (only based on variational and PDE arguments) was subsequently extended  to  $H^{\frac 1 2}(\R)$ and $H^{\frac 1 2}(\TT)$ by  Guo-Wu\ccite{Guo} and  Oh-Mosincat\ccite{mosincat2},   Mosincat\ccite{mosincat1} respectively, and more recently to   $H^{s}$, for~$1/6 \leq s<1/2$, both on the real line and on the circle by Killip, Ntekoume and Visan in\ccite{KillipNtekoumeVisan2021}.  

\medskip
The barrier $4\pi$ has been surpassed only recently, and  all progress in this direction  relies on the complete integrability  of DNLS.
The principal difficulty is the failure of coercivity of the DNLS conservation laws in the regime $\|u\|_{L^2}^2\geq 4\pi$ that can be readily seen by considering
the algebraic solitons:
\begin{equation}
\label{algsoliton}   u_{c}(t, x)=   2 \sqrt{c} \, e^{ -i \frac  {c^2} 4t+ i \frac c 2 x} \frac{cx-c^2t+ i}{{(cx-c^2t- i)^2}}, \quad c>0,\end{equation}
for which  $M_\R(u_{c})=4\pi$, while all other polynomial conservation laws vanish.

\medskip
On the real line, the first global well-posedness results without any restriction on the mass were obtained  via  inverse scattering techniques,
and therefore were  limited to initial data with strong spacial decay\ccite{Sulem1, Sulem, Sulem0, LPS,  PSS,  PSS2}. 

\smallskip
The most definite result is due to 
Jenkins, Liu, Perry and Sulem who have proved in\ccite{Sulem1} that the Cauchy problem for the DNLS equation is globally well-posed  in the weighted Sobolev space~$H^{2, 2}(\R)=\big\{f  \in H^2(\R)  : x^2 f  \in L^2(\R)\big\}$. 

Thereafter in\ccite{BaPe},  the authors proved  global well-posedness of DNLS in $H^{\frac 1 2}(\R)$, combining  the DNLS  integrability  structure with the profile decompositions techniques.  Finally, Harrop-Griffiths, Killip, Ntekoume and Visan succeeded in  extending global well-posedness  to $L^2(\R)$, which is the scaling-critical space for DNLS, by using
 the commuting flows method.  

\medskip
The periodic case is less understood. The first result going beyond the $4\pi$ restriction was established in\ccite{BaPe2} where the authors proved global well-posedness of DNLS in $H^1(\TT)$
 provided  the mass of initial data is strictly less than $8\pi$. The aim of the present paper is to completely remove  the restriction on the mass.
Our main result      is the following: 
 \begin{theorem}
\label{Mainth}
{\sl 
 The Cauchy problem \eqref{DNLS-T}, \eqref{id}  is globally well-posed in $H^1(\TT)$. Furthermore, for any $R>0$,  there exists~$C(R)>0$ so that
\begin{equation}\label{b55}
\sup\limits_{t\in \R}\|u(t)\|_{H^{1}(\TT)}\leq C(R),\end{equation}
for any $H^1(\TT)$-solution to  \eqref{DNLS-T}  with   $\|u(0)\|_{ H^1(\TT)}\leq R$.}
\end{theorem}

\medbreak

The structure of the paper is as follows.   Section~\ref{0preliminarystatement20}  collects   the properties of the Kaup-Newell spectral problem arising in the Lax pair formulation of the DNLS equation,
 which are required  for  the proof of Theorem \ref{Mainth}.   The subsequent sections  are devoted to  the proof of  Theorem\refer{Mainth}, which combines profile decomposition techniques  with
 tools from
 complete integrability. 
 The paper also contains four appendices. In Appendix A, we recall some properties of regularized determinants. Appendix B provides some background on the Zakharov-Shabat spectral problem, while Appendices C and D contain the proofs of several technical results.

\medskip
Throughout this article, we  use the following convention
for the Fourier transform on the   real  line
\begin{equation*}
\hat f(\xi)= \frac 1 {\sqrt {2 \pi}} \int_\R e^ {-i x \xi}f(x) dx.
\end{equation*}
We denote by  $\|\cdot\|_2$ the Hilbert-Schmidt norm, and by ~$\|\cdot\|$ the operator norm on~$L^2(\R)$.  The 
standard hermitian norm  on~$\C^2$ and the induced operator norm  on the space of complex~$2\times 2$ matrices will be denoted by~$|\cdot|$.
 The letter $C$ denotes  universal constants
which may vary from line to line. If the implied constant depends on parameters, we indicate this by subscripts.
 We  also use the notation $A\lesssim B$  (respectively $A\gtrsim B$) to
denote   bounds of the form $A\leq C B$ (respectively $A \geq C B$),   and $A \lesssim_\alpha B$ 
for~$A\leq C_\alpha B$, where~$C_\alpha$ depends only  on $\alpha$.
For simplicity, we continue to denote by~$(u_n)$ any
subsequence of~$(u_n)$. 
Finally, throughout the paper, we adopt the convention that sums over empty sets are zero and products over empty sets are equal to $ 1$. 

\smallskip


\section{Kaup-Newell spectral problem}\label{0preliminarystatement20}
\subsection{Lax pair formulation of the DNLS equation} 
As it  was discovered 
 by Kaup and Newell in\ccite{KN}, the DNLS equation can be viewed  as a compatibility condition of the following linear system
 \begin{equation}
\label{system}\begin{array}{c}
\partial_x \psi= \cU(\lam)  \psi,\\
\partial_t \psi= \Upsilon(\lam)  \psi\,,
\end{array}
 \end{equation} 
with
\begin{eqnarray*}\cU(\lam) &= &-i \sigma_3(\lam^2 + i\lam U), \quad U=\left(
\begin{array}{ccccccccc}
0 &u \\
\overline {u} &0
\end{array}
\right), \\ \Upsilon(\lam) &= &-i (2\lam^4- \lam^2\, |u|^2) \sigma_3 + \left(
\begin{array}{ccccccccc}
0 & 2 \lam^3u - \lam |u|^2 u + i \lam u_x\\ 

-2 \lam^3\overline {u}+  \lam |u|^2 \overline {u} + i \lam \overline {u_x}&0
\end{array}
\right) ,\end{eqnarray*}
where  $\lam\in\C$, $\psi$  is a $\C^2$-valued function of~$(t, x, \lam)$, and~$\sigma_3$  is the Pauli matrix given by~$\sigma_3= 
\left(
\begin{array}{ccccccccc}
1 &0 \\
0 &-1
\end{array}
\right)$.  
In other words, $u$ satisfies the DNLS equation if and only if 
$$\frac{\partial \cU}{\partial t}-\frac{\partial \Upsilon}{\partial x}+ [\cU, \Upsilon]=0.$$
The first equation of  \eqref{system},  viewed as a spectral problem, will play a central role in our analysis. We will write it in the form 
\begin{equation}\label{sp}
L_u(\lambda)\psi=0,
\end{equation}
  where $L_u(\lambda)=\cL_0-\lambda^2-i\lambda U$, with $\cL_0= i \sigma_3 \partial_x$.  
\smallskip
In the next  three subsections, 
 we collect the properties of the Kaup-Newell spectral problem \eqref{sp} as well as some connected results that will be  used in the proof of Theorem\refer{Mainth}, referring  to\ccite{BaPe, BaPe2, BaPeL, bealscoifman, GKap, KillipVisan2021, Sulem0, KN, Lee,  LPS,  PSS, PSS2} and to Appendices for proofs and
  further details.

\subsection{Kaup-Newell spectral problem in the periodic case}\label{periodic}
Given $u\in L^2_{loc}(\R)$,
we denote $E_u(x, \lam)$  the canonical fundamental solution  of\refeq{sp}:
\begin{equation}\label{fundmat} \left\{
\begin{array} {ccl}
L_u(\lambda)    E_u(x, \lam) &= & 0\\
E_u(0, \lam) &=& {\rm Id}\, .
\end{array}
\right.\end{equation}   
For each $x\in \R$, the fundamental matrix $E_u(x,\lam)$ is an  analytic function   of $(\lambda, u, \bar u)$ with the following properties:
\begin{equation}\label{wronskian} \det E_u(x,\lam)=1,  \end{equation} 
\begin{equation}\label{scaling1}  E_{u_\mu}(x, \lambda)=E_u(\mu^{-1}x, \sqrt\mu \lambda),
\end{equation}
\begin{equation}\label{symMKPbis} \sigma_3 \,E_u(x, - \lam)\, \sigma_3 =  E_u(x,  \lam) , \, \,  \sigma_1 \overline {E_u(x, -\overline {\lam})}\,   \sigma_1= E_u(x,  \lam), 
\end{equation}
where $\sigma_1= 
\left(
\begin{array}{ccccccccc}
0 &1 \\
1 &0
\end{array}
\right)$. 
 
\smallskip 

In the case where  $u$ is periodic with period $T>0$: $u(x+T)=u(x)$,  we denote  by~$M_u(\lam)$ the corresponding monodromy matrix:
\begin{equation}\label{monodromy}
M_u(\lam)=\left(
\begin{array}{ccccccccc}
M_{11} (\lam,u)&M_{12} (\lam,u)\\
M_{21} (\lam,u) &M_{22} (\lam,u)
\end{array}\right)=E_u(T, \lambda),
\end{equation}
 which satisfies $\det M_u(\lambda)=1$, 
and by
 $\Delta_u(\lam)$ the trace of~$M_u(\lambda)$: \begin{equation}\label{trace} \Delta_u(\lam) =  M_{11}(\lam, u)+M_{22}(\lam, u).\end{equation} 
  It follows from \eqref{system} that $\Delta_u(\lambda)$ is time independent if $u$ is a $T$-periodic solution of the~DNLS equation.

\smallskip
We also need to introduce the functions 
\begin{eqnarray}
\label{defD}  A^{D}_u (\lam) &= & \frac{i}2\big( M_{11} (\lam,u)+M_{12} (\lam,u) -M_{21} (\lam,u)-M_{22} (\lam,u) \big), \\  \label{defN}A^{N}_u (\lam) &= & \frac{i}2 \big(M_{11} (\lam,u)-M_{12} (\lam,u) +M_{21} (\lam,u) - M_{22} (\lam,u)\big). \end{eqnarray} 
Zeros of the function $A^{D}_u $ (resp. $A^{N}_u$) are  the eigenvalues of \eqref{sp} considered on $[0,T]$ with the Dirichlet (resp. Neumann) boundary condition~$(\psi_1-\psi_2) (T)= (\psi_1-\psi_2) (0) = 0$    (resp.~$(\psi_1+\psi_2) (T)= (\psi_1+\psi_2) (0) = 0$), where we denote  $\psi= \left(
\begin{array}{ccccccccc}
\psi_1 \\
\psi_2
\end{array}
\right)$.

 \smallskip Obviously, in the case of $u=0$,    the fundamental solution of \eqref{sp} is given by  \begin{equation}\label{case0}   E_0(x,\lam)= \exp\big(- ix \lam^2 \,\sigma_3\big)= \left(
\begin{array}{ccccccccc}
e^{-i \lam^2  x} &0 \\
0 & e^{i \lam^2 x}
\end{array}
\right)\,\end{equation}  
and so
 \begin{equation}\label{trace0} M_0(\lam) =e^{-iT\lam^2 \,\sigma_3}, \quad \Delta_0(\lam)=2 \cos(T\lam^2) \, ,\end{equation} 
 \begin{equation}\label{DN0}A^{D}_0 (\lam) = A^{N}_0 (\lam) =\sin (T\lam^2) \, .\end{equation}

Thanks to the scale invariance property \eqref{scaling1} and the symmetry relations \eqref{symMKPbis},  one has
\begin{equation}  \label{symMKmono} \sigma_3 M_u(- \lam) \sigma_3 =  M_u(\lam), \quad \sigma_1 \overline {M_u(-\overline {\lam})}  \sigma_1= M_u(\lam), \quad 
M_{u_\mu}(\lambda)=M_u(\sqrt\mu\lambda),
\end{equation}
and therefore,
\begin{eqnarray}    \label{symMKdelta}\Delta_u(-\lam) &=& \Delta_u(\lam), \quad   \overline {\Delta_u(\overline {\lam})}= \Delta_u(\lam), \quad \Delta_{u_\mu}(\lambda)=\Delta_u(\sqrt\mu\lambda)\\  \label{symMKdi}\overline {A^{D}_u (- \overline \lam)}&=& {A^{D}_u (\lam)},  \quad {A^{D}_u (-\lam)}= {A^{N}_u (\lam)}.
\end{eqnarray}

 \smallskip
To study the behavior of the  matrix $E_u(x,\lam)$ for large $\lam$,  it is more convenient to transform the Kaup-Newell spectral system \eqref{sp} into a more "familiar" Zakharov-Shabat spectral problem 
which is linear with respect to the spectral parameter.
 Setting
\begin{equation}
\label{gaugetorus} \wt\psi(x) =
 \exp\Big(\frac {i  \sigma_3}  {2}  \int^x_0 dy |u(y)|^2\Big) 
 \left(
\begin{array}{ccccccccc}
1 &0 \\
- \overline {u}(x) &  2 i\lam
\end{array}
\right)\psi (x) ,\end{equation}   one can easily check   that   $\psi$ is a solution of \eqref{sp} if and only if $\tilde \psi$ solves    the following system
\begin{equation}\label{sp1}
 (i\sigma_3\partial_x -\zeta-\cQ)\tilde \psi=0,
\end{equation}
where $\zeta=\lam^2$ and
$ \cQ=   \left(
\begin{array}{ccccccccc}
0 &q_1  \\
q_2 &0
\end{array}
\right)
$
with 
\begin{eqnarray}\label{sp1-1}
 q_1(x) &=& \frac12 u(x)\exp\Big(  i  \int^x_0 dy |u(y)|^2\Big)  \\   \label{sp1-2} q_2(x) &= & (i  \overline {u}_x+\frac12\overline {u}|u|^2) (x)\exp\Big( -i  \int^x_0 dy |u(y)|^2 \Big).\end{eqnarray}

  The representation \eqref{gaugetorus}  leads immediately to the following bounds, see\ccite{BaPe2} for the proof.
 
 \begin{proposition} [\cite{BaPe2}]\label{pr1}
{\sl As $|\lam|\rightarrow \infty$, one has
\begin{equation}\label{p1-1}\begin{split}
 \Delta_u(\lam)= 2\cos\Big(\lambda^2T+\frac{\|u\|^2_{L^2([0,T])}}2\Big)+  O_T\Big(\frac {e^{|\im \lam^2|T}}{|\lam|^2 }\Big),\\
A^D_u(\lam), A^N_u(\lam)=\sin \Big(\lambda^2T+\frac{\|u\|^2_{L^2([0,T])}}2\Big)+O_T\Big(\frac {e^{|\im \lam^2|T}}{|\lam| }\Big),
\end{split}
\end{equation}
uniformly with respect to $u$ in bounded sets of $H^1([0,T])$. }
\end{proposition}
\medskip    

As a direct consequence  of the asymptotics \eqref{p1-1} and  the symmetry properties \eqref{symMKdi},  one obtains the following result 
concerning  the behavior of large imaginary Dirichlet and Neumann eigenvalues under continuous deformations of the potential preserving the Floquet discriminant.

\begin{lemma}[\cite{BaPe2}]\
\label{DN}
{\sl Let $R\geq 0$ and $u \in C([0, 1], H^1(\TT))$ with $\|u(0)\|_{H^1(\TT)}\leq R$ such that, for all~{$\tau \in [0, 1]$}, $\Delta_{u(\tau)}=\Delta_{u(0)}$. Then for any  $0<\gamma\leq \frac\pi 4$, there exists an integer~$N_0=N_0(R, \gamma) \gg1$ such that, for all~$n \leq - N_0 $, the interval\footnote{Since $u$ depends continuously on $\tau$, the conservation of $\Delta_{u}$ implies the conservation of the mass of $u$.}
$$I_n=[\lam_n^+, \lam_n^-]\subset i\R_+, \quad \lam_n^\pm=\sqrt{n\pi-\frac12\|u\|^2_{L^2}\pm\gamma}, $$
contains at least one Dirichlet eigenvalue and one Neumann eigenvalue of $L_{u(\tau)}(\lam)$.}
\end{lemma}

This lemma was one of the key ingredients in the analysis in \cite{BaPe2} and will also play an important role in  the proof of Theorem\refer{Mainth}.

\medbreak


 \subsection{Kaup-Newell spectral problem on the real line}\label{0preliminarystatement2} 
 \subsubsection{Basic properties} In order to establish   Theorem~\ref{Mainth}, we also need to recall some  results about the  Kaup-Newell spectral problem on the real line. 
 Given~$u\in \mathcal{S}(\R)$,   for any~$\lambda \in \C$ with~$\im\lam^2\geq0$,
there are unique solutions\footnote{If $\im \lam^2>0$, the solutions $\psi_1^-$ and  $\psi_2^+$ remain well defined for $u\in L^2(\R)$.}
$\psi_1^-(x, \lambda; u)$ and~$\psi_2^+(x, \lambda; u)$ to~\eqref{sp}, the so-called Jost solutions,  with  the following behavior at $\pm \infty$:
\begin{equation}\label{Jost1}\begin{aligned}
\psi_1^-(x, \lambda;u)&= e^{ -i \lambda^2 x} \left[ \left(
\begin{array}{ccccccccc}
1  \\
0 
\end{array}
\right) + o(1)\right], \quad \mbox{as} \quad x \to - \infty, \\
\psi_2^+(x, \lambda; u) &=   e^{ i \lambda^2 x} \, \, \, \left[ \left(
\begin{array}{ccccccccc}
0  \\
1 
\end{array}
\right) + o (1)\right], \quad \mbox{as} \quad x \to + \infty.
\end{aligned}\end{equation}
The Jost solutions $\psi_1^-$, $\psi_2^+$ are holomorphic functions of   $\lambda$ on $\Omega_+=\{\lambda\in \C: \,\ \im \lambda^2>0\}$,~$C^\infty$ up to the boundary.

 \smallskip Similarly, for $\lambda \in \C$ with $\im \lambda^2\leq0$, there are 
unique solutions  $\psi_1^+(x, \lambda;u)$, $\psi_2^-(x, \lambda;u)$ to \eqref{sp}  such that
{\begin{equation}\label{Jost2}\begin{aligned}
\psi_1^+(x, \lambda;u) &=  e^{- i \lambda^2 x} \left[ \left(
\begin{array}{ccccccccc}
1  \\
0
\end{array}
\right) + o(1)\right], \quad \mbox{as} \quad x \to + \infty,\\
\psi_2^-(x, \lambda;u)&= e^{ i \lambda^2 x} \, \, \,\left[ \left(
\begin{array}{ccccccccc}
0  \\
1
\end{array}
\right) + o(1)\right], \quad \mbox{as} \quad x \to - \infty.
\end{aligned}\end{equation}} 
For $\lambda\in \R\cup i\R$,  this gives two pairs of linearly independent solutions: $\psi_1^-, \psi_2^-$ and $\psi_1^+, \psi_2^+$.
We denote the corresponding transfer matrix by $t_u(\lambda)=\begin{pmatrix} a_u(\lambda) & c_u(\lambda)\\b_u(\lambda) & d_u(\lambda)\end{pmatrix} $:
\begin{equation}\label{transfert}\begin{pmatrix} \psi_1^-(x,\lambda;u) &\psi_2^-(x,\lambda;u)\end{pmatrix}=\begin{pmatrix} \psi_1^+(x,\lambda;u) &\psi_2^+(x,\lambda;u)\end{pmatrix}t_u(\lambda).\end{equation} 
Obviously, 
$$ t_u(0)= {\rm Id}.$$

\medskip Thanks to the symmetry relations
\begin{equation}\label{symR1}
\psi_1^-(x,\lambda;u)=  \sigma_3\psi_1^-(x,-\lambda;u), \,  \,  \,  \psi_2^+(x, \lambda;u)=-\sigma_3\psi_2^+(x,-\lambda;u), 
\end{equation}
\begin{equation}\label{symR2}
\psi_1^-(x, \lambda;u)=-\sigma_1\sigma_3\overline{\psi_2^-(x, \bar \lambda;u)}, \quad \psi_2^+(x,\lambda;u)=\sigma_1\sigma_3\overline{\psi_1^+(x, \bar \lambda;u)},
\end{equation} 
one has, for all $\lambda\in \R\cup i \R$, 
\begin{eqnarray*} a_u(\lambda) &= & a_u(-\lambda), \,\,\, a_u(\lambda)=\overline{d_u(\bar\lambda)},  \\ b_u(-\lambda) &= &-b_u(\lambda), \,\,\,  c_u(\lambda)=-\overline{b_u(\bar \lambda)}.\end{eqnarray*} 
Since $\det t_u(\lambda)=1$, the above relations imply that
\begin{eqnarray} \label{identab}|a_u(\lambda)|^2+|b_u(\lambda)|^2 &=&1, \quad \forall\,\, \lambda\in \R, \\
\label{identabim} |a_u(\lambda)|^2-|b_u(\lambda)|^2 &=&1, \quad \forall\,\, \lambda\in i\R.\end{eqnarray} 

\medskip Using the second equation in \eqref{system}, it can be shown that $a_u(\lambda)$ is time-independent if $u$ is a solution of the DNLS equation, while the evolution of $b_u$ is given by
$$\partial_t b_{u(t)}(\lam)=-4i\lam^4 b_{u(t)}(\lam).$$

\medskip

The function  $a_u$ extends analytically to $\Omega_+$ since it can be expressed through the Wronskian of $\psi_1^-$ and $\psi_2^+$:
\begin{equation}\label{defa}
a_u(\lambda)=\det(\psi_1^-(x,\lambda; u), \psi_2^+(x,\lambda; u)).
\end{equation}
The relation \eqref{defa} shows that  the zeros of $a_u$ in $\Omega_+$ are the eigenvalues of the spectral problem~\eqref{sp} (considered in $L^2(\R)$). Let us recall the following result proved in\ccite{BaPe}.
\begin{lemma}[\cite{BaPe}]
\label{lemmcont}
{\sl  For any 
$R\geq 0$ there exists a positive constant $C_R$ such that $a_u(\lam)\neq 0$, for all $\lam\in \Omega_+$ with $\re \lam^2\leq -C_R$, provided that
$\|u\|_{H^\frac12(\R)}\leq R$.
}
\end{lemma}

Since $a_u$ is an even function of $\lam$, and $b_u$ is odd,  it is convenient to introduce $\tilde a_u(\zeta)=a_u(\sqrt{\zeta})$ and $\tilde b_u(\zeta)=\frac{b_u(\sqrt{\zeta})}{\sqrt{\zeta}}$.
In terms of $\tilde a_u, \tilde b_u$ the relations  \eqref{identab} and \eqref{identabim} take the   following   form:
\begin{equation} \label{identab1}|\tilde a_u(\zeta)|^2+\zeta|\tilde b_u(\zeta)|^2 =1, \quad \forall\,\, \zeta\in \R.
\end{equation}
In particular,
\begin{equation}
\label{size}|\tilde a_u(\zeta)|\geq 1 \, \, \, \mbox{for} \, \, \zeta <0, \, \, \,|\tilde a_u(\zeta)|\leq 1 \, \, \, \mbox{for} \, \, \zeta >0 \andf   \tilde a_u(0)=1. \end{equation}

From the equivalence between the systems \eqref{sp} and  \eqref{sp1}, \eqref{sp1-1}, \eqref{sp1-2}, one can readily deduce that~$\tilde b_u\in \mathcal{S}(\R)$, while 
$\tilde a_u$ satisfies
\begin{equation}\label{lim}
\lim\limits_{|\zeta|\rightarrow\infty, \, \zeta\in \overline{\C}_+}\tilde a_u(\zeta)=e^{-\frac{i}2\|u\|^2_{L^2(\R)}}.
\end{equation}
Furthermore, one has the following asymptotic expansion as~$|\zeta|\rightarrow +\infty$, $\im \zeta\geq 0$:
 \begin{equation}\label{as}
\ln \tilde a_u(\zeta)=\sum\limits_{k\geq 0} E_k(u)\zeta^{-k}.\end{equation}
The coefficients\footnote{We choose the branch of the logarithm so that $\lim\limits_{|\zeta|\to \infty, \zeta\in \overline{\C}_+} \ln \tilde a_u(\zeta) = -\frac{i}{2}\|u\|_{L^2(\R)}^2.$}
 $E_k(u)$
 are polynomial with respect to    $u, \overline {u}$ and their derivatives,
 and homogeneous with respect to the scaling\refeq{scaling},
since
\begin{equation}
\label{sca}
\tilde a_{u_\mu}(\zeta)=\tilde a_u\big({{\mu}}\zeta\big)\cdotp
\end{equation}
 As~$\tilde a_u(\zeta)$ is time-independent, the quantities~$E_k(u)$ are conservation laws. 
The first three of  them coincide, up to constants,  with the  mass, momentum and energy:
\begin{equation}
\label{relcons}
E_0(u) = -\frac{i}{2}\|u\|_{L^2(\R)}^2, 
\qquad 
E_1(u) = \frac{i}{4} P_\R(u), 
\qquad E_2(u) = -\frac{i}{8}E_\R(u).
\end{equation}

\smallskip
The function  $\tilde a_u$ is analytic in $\C_+=\{z\in \C:\,\,  \im z>0\}$ and $C^\infty$ up to the boundary,  which together with \eqref{size} and \eqref{lim} leads to the following representation
 \begin{equation}\label{devolg} 
 \tilde a_u(\zeta) = e^{-\frac{i}2\|u\|^2_{L^2(\R)}} \prod\limits_{j}  \Big( \frac {\zeta -\zeta_j}  {\zeta - \overline \zeta_j} \Big) \exp\Big(\frac 1 {i\pi} \int^\infty_{-\infty} \frac {\xi} {\zeta-\xi} d\mu_u(\xi) \Big)\, ,  \, \, \forall \zeta \in \C_+ ,\end{equation} 
where $\zeta_j$ are the zeros\footnote{By \eqref{size}, \eqref{lim}, there exists $R>0$ such that
$R^{-1}\leq|\zeta_j|\leq R, \quad 0<\arg \zeta_j\leq\pi-R^{-1}$ for all $j$, which ensures that $\sum\limits_j \im \zeta_j<+\infty$ and
$\sum\limits_j \arg \zeta_j<+\infty$,   see\ccite{KillipVisan2021} for details. }
of $\tilde a_u$ in $\C_+$ counted with their multiplicity, and $\mu_u$ is a  positive measure on $\R$ with finite moments of all order, see  Proposition~4.2 in\ccite{KillipVisan2021}. In the case where $\tilde a_u$ does not vanish\footnote{It follows from \eqref{size} and \eqref{lim} that in that case, $\tilde a_u$ has 
only a finite number of zeros in $\C_+$.} on $\R_+$, the measure $\mu_u$ is given by 
$$d\mu_u(\xi)= -\frac 1{2\xi} \ln |\tilde a_u(\xi)|^2.$$
From \eqref{devolg}, 
we infer that 
\begin{equation}
\label{asympek}E_k(u)=- \frac{2i}k\sum\limits_{j} \im\zeta_j^k-\frac{i}{\pi}\int^\infty_{-\infty}   \, \xi^{k} \, d\mu_u (\xi), \quad \forall k\in \N^* \,. \end{equation} 
Furthermore, the continuity\footnote{ The map $u\mapsto \tilde a_u$ is real analytic from $H^{1, 1}(\R)= \{u\in H^1(\R): xu, x\partial_xu\in L^2(\R)\}$ to $L^\infty(\overline \C_+)$.}
of $\tilde a_u$ with respect to $u$, together with \eqref{size} and \eqref{lim} ensures that for all $u\in \mathcal{S}(\R)$,
\begin{equation}\label{mass-00}
\int\limits_{-\infty}^0  \frac {\tilde a'_{u}(s)}  {\tilde  a_{u}(s)} ds=\frac{i}2\|u\|^2_{L^2(\R)}.
\end{equation}
 As a straightforward consequence of this
 formula and of the analyticity of $\tilde a_u$ in $\C_+$, one obtains the following result.
\begin{lemma} 
\label{intrel}
{\sl Let~$u\in \cS(\R)$.
\begin{itemize}
\item[i)]
Let  $\theta \in ]0, \pi[$ such that $\tilde a_u(\zeta) \neq 0$, for all $\zeta$ in $\C_+$ with~$\arg \zeta = \theta$ and  let $n(u; \theta)$ be the number of zeros of $\tilde a_u(\zeta)$ in the angle~$\{ \zeta\in \C_+:   \, \,  \theta < \arg \zeta< \pi \big\}$ counted with their multiplicity.
Then  
\begin{equation}
\label{controlpert} 
n(u;\theta)=
\frac{1}{2i\pi }\int\limits^{+ \infty \, e^{i\theta}}_0 \frac {\tilde a'_{u}(s)}  {\tilde  a_{u}(s)} ds+\frac1{4\pi}{\|u\|^ 2_{L^2(\R)}} \virgp\end{equation}
 where, all along this paper,    $ \ds \int\limits^{+ \infty \, e^{i\theta}}_0 ds $ denotes the integral along the path 
$\gamma = \big\{z= \rho \, e^{i \theta},  \, \,  \rho \in \R_+ \big\}.$
\item[ii)] Let~ $\zeta_j$ and $\mu_u$ be as in \eqref{devolg}. Then
\begin{equation}\label{mass1-0}
M_\R(u)  =  4 \sum\limits_{j} \arg \zeta_j  +\frac 2 { \pi} \int_\R d\mu_u(\xi).
\end{equation}

\end{itemize}}
\end{lemma}

\begin{remark}\label{rem-number}
{\sl Observe that due to\refeq{mass1-0}, we have:
\begin{equation}
\label{controbound} n(u; \theta) \leq \frac{\|u\|^2_{L^2(\R)}}{4\theta},   \end{equation}
for all $u\in\cS(\R)$ and all $\theta \in ]0, \pi[$.}
\end{remark}
\medskip


\subsubsection{Regularized determinant realization of $a_u$}
A convenient way to study the function~$a_u$ is to use its representation as a regularized perturbation determinant\footnote{See Appendix\refer{basicdeterminants} for the definition of the regularized determinants ${\rm det}_n$ and their basic properties.}:
\begin{equation}
\label{awithdet}  a_u(\lam)= {\rm det}_2 ({\rm I}-T_u(\lam)), \quad \forall \,\lambda\in \Omega_+, \,\, u\in L^2(\R),
\end{equation}
where
 \begin{equation}\label{T}T_u(\lam)= i \lam (\cL_0-   \lam^2)^{-1} U,  \quad U=\left(
\begin{array}{ccccccccc}
0 &u \\
\overline {u} &0
\end{array}
\right).\end{equation}
The operator~$T_u(\lambda)$ is Hilbert-Schmidt, with 
\begin{equation}\label{T-HS}
\|T_u(\lambda)\|_2^2 = \frac{|\lambda|^2}{\im(\lambda^2)} \|u\|_{L^2(\R)}^2.
\end{equation}

\medskip The representation \eqref{awithdet} together with the corresponding properties of the regularized determinants leads immediately to the following bounds, the proof of which can be found, for instance, in \cite{BaPe}.

\begin{proposition} [\cite{BaPe}]
\label{stability0}
{\sl There exists  a positive  constant~$C$ such that the following estimates hold 
\begin{equation}  \begin{aligned}
\label{eqstab} |a_{u_1} (\lam)-a_{u_2} (\lam)| & \leq C  e^{ C  \frac {|\lam|^2 } {{ \rm Im} (\lam^2)} \big(\|u_1\|^2_{L^2(\R)}+\|u_2\|^2_{L^2(\R)}\big)} \frac { |\lam| } {\sqrt { { \rm Im} (\lam^2) }  }  \|u_1- u_2\|_{L^2(\R)}, &\\ &\qquad \qquad \qquad \qquad \qquad \qquad \quad \forall \,\lambda\, \in \Omega_+, \, \forall u_1, \, u_2\in L^2(\R). \end{aligned}\end{equation}
 \begin{equation}  \begin{aligned} \label{eq6}
 |a_u(\lam)e^{\frac{i}2\|u\|^2_{L^2(\R)}}- 1| &\leq Ce^{C\frac { |\lam|^4 }
  {({ \rm Im}  (\lam^2) )^2}  \|u\|^4_{L^2(\R)}}    \frac {|\lam|^2} {({ \rm Im}  (\lam^2) )^2}
  \|u\|^2_{\dot H^{ \frac 1 2}(\R)}  , \\ &\qquad \qquad \qquad \qquad \qquad \qquad \quad \forall \lam\in \Omega_+, \, \forall u\in H^{\frac12}(\R). \end{aligned}\end{equation} }
 \end{proposition}
  \medbreak 
 Combining \eqref{eqstab} with   \eqref{size}, \eqref{lim} and  Remark \ref{rem-number}, and taking into account the density of $\cS(\R)$ in $L^2(\R)$, we obtain:
  \begin{cor} 
    \label{coruse}
    {\sl Let $u$ be a function in $L^{2}(\R)$. 
    \begin{itemize}
    \item[i)]
    For all $0<\delta<\frac\pi 2$,  there holds
  \begin{equation}
\label{studyaubeh}    
\lim\limits_{\zeta\rightarrow 0, \,\zeta \in \Gamma_\delta}\tilde a_{u} (\zeta)= 1, \quad
\lim\limits_{|\zeta|\rightarrow \infty, \, \zeta \in \Gamma_\delta} \tilde a_{u} (\zeta)=e^{-\frac{i}2\|u\|^2_{L^2(\R)}} \, ,
\end{equation} 
where  $\Gamma_\delta=\big\{ \zeta \in \C_{+}  :  \delta <  \arg \zeta  < \pi -\delta \big\}$.
\item[ii)] For all $\theta \in ]0, \pi[$,
\begin{equation}\label{controbound-1}
n(u; \theta) \leq \frac{\|u\|^2_{L^2(\R)}}{4\theta}, \end{equation}
where, as before, 
$n(u; \theta)$ denotes the number of zeros of $\tilde a_u(\zeta)$ in the angle $\{ \zeta\in \C_+:   \, \,  \theta < \arg \zeta< \pi \big\}$, counted with their multiplicity.
\smallskip 
\item[iii)] Let $\theta\in ]0,\pi[$ such that   $\tilde a_{u}(\zeta) \neq 0$, for all $\zeta$ belonging to the ray $e^{i\theta}\R_+$. Then,
\begin{equation}
\label{calcullimitnew} \frac{1}{2i\pi} \int\limits^{+ \infty \, e^{i \theta}}_{0}ds \, \frac { \tilde  a'_{u}(s)}  {\tilde  a_{u}(s)}+   \frac 1 {4 \pi} \|u\|^ 2_{L^2(\R)} \in \N.\end{equation}
\end{itemize}} 
    \end{cor}
 \medbreak    
 We next recall the following  lemma from \cite{BaPe2} that gives a lower bound for $|\tilde a_u|$ in the sectors where $\tilde a_u$ does not vanish.
\begin{lemma}[\cite{BaPe2}] \label{lem-lowb}
 {\sl Let  $M\geq 0$, $0<\theta_1<\theta_2<\pi$, and $0<t\leq\frac12$. Then there exists a positive constant $C=C(\theta_1, \theta_2, t, M)$  so that 
 \begin{equation}\label{lowb}
\left|\frac1{\tilde a_u(\zeta)}\right|\leq C,
\end{equation}
 for all $\zeta\in \C_+$, with $(1-t)\theta_1+t\theta_2\leq\arg \zeta\leq t\theta_1+(1-t)\theta_2$ and all $u\in L^2(\R)$ with~$M_\R(u)\leq M$, and  such that $\tilde a_u$ has no zeros in the angle 
 $\{ \zeta\in \C_+:   \, \,  \theta_1 \leq \arg \zeta\leq \theta_2\big\}$.}
 \end{lemma}
\medbreak

\subsubsection{Algebraic multi-soliton states}  Here we collect some properties of the non-zero~$H^1$ potentials $u$ whose spectral coefficient $\tilde a_u$ is identically equal to 1.  These potentials will play a central role in the proof of Theorem \ref{Mainth}, arising as the only possible candidates for  blow up profiles, see Theorem \ref{result1newth} for the precise statement. We start by recalling the following result from \cite{BaPe2}:

\begin{lemma} [\cite{BaPe2}]
\label{studyfirstpartimaginary}
{\sl Let $u\in H^{1 }(\R)$ be such that the corresponding function $\tilde a_u$ has no zeros in~$\C_+$. Then
 $ E_\R(u)\geq 0$.
 Furthermore,
 $E_\R(u) =0$  if and only if 
 \begin{equation} \label{impartcoef1}  \tilde a_u(\zeta) = 1\, , \quad  \forall \zeta \,\in \C_+ . \end{equation}
 In that case,  one also has
    $P_\R(u)=0$ and  $M_\R(u) \in 4\pi \N$.}
\end{lemma}

The set  of non-zero potentials satisfying $\tilde a_u\equiv 1$ is non empty:
  it contains the algebraic solitons\refeq{algsoliton}, see for instance \cite{BaPe, HKillipNtekoumeVisan2021}. Furthermore, up to translations in phase and space,  they are the only potentials 
 with~$\tilde a_u\equiv 1$ and $M_\R(u) = 4\pi$, see Lemma 2.6 in \cite{BaPe2}. The next result, whose proof is   given in   Appendix\refer{gen-multi},
 provides a  
 characterization of these potentials that  will be needed for the proof of Theorem \ref{Mainth} and that may be of independent interest.

 \begin{proposition}\label{car-sabis}
 {\sl   Let $u\in H^1(\R)$ such that  $\tilde a_u\equiv 1$ and let $\ds N= \frac {M_\R(u)}{4 \pi}\geq 1$. Then,   the following properties hold:
 \begin{enumerate} \item[(i)] $u$ belongs to  $\ds H^\infty(\R)$;
 \item[(ii)] $u$ solves  a differential equation  of the form
  \begin{equation} \label{eq:ODE} u^{(2N)}  + P(u) =0
  \end{equation}
  where $P$ is a polynomial in $u, \bar u$ and their derivatives of order less or equal to~$2N-1$,  vanishing at zero.
  \end{enumerate}}
  \end{proposition}
   \begin{remark}
 {\sl   In Appendix\refer{gen-multi}, we will obtain a much more precise statement showing that~$u$ is a critical point of the functional 
  $${E}_{2N}+ \sum\limits_{j=0}^{2N-1}\gamma_j E_j,$$ for some suitable choice of $\gamma_0, \dots \gamma_{2N-1}\in \C$. }
\end{remark}

  \medbreak
 \subsubsection{A Guo-Wu type inequality} \label{GW} 
  We  conclude this subsection with the following proposition that 
extends   Guo-Wu estimate to Schwartz potentials $u$  with $M_\R(u)\geq 4\pi$ subjected to some spectral  assumption.   
\begin{proposition} \label{keygen}
 {\sl For all  $M>0$, there exist $\varepsilon= \varepsilon (M)>0$ and $C(M)>0$ such that, for any potential~$u \in  \cS(\R)$ with $M_\R(u)  \leq M$  and such that  $\tilde a_u$ has no zeros  in the angle~$\big\{ \zeta\in \C_+:  \frac \pi 2 + \varepsilon<  \arg  \zeta< \pi \big\}$, there holds  \begin{equation}
\label{controlhomH1spect} \|u\|^2_{\dot H^{1}(\R)} \leq C(M) (P^2_\R(u)+ |E_\R(u)|) \, .  \end{equation} 
 }
 \end{proposition}
 \smallbreak
 
 The proof of Proposition \ref{keygen} is given in   Appendix\refer{app-gw}. 
 
\medbreak


 \subsection{B\"acklund transformation} \label{defBacklund transformation}  In this paragraph,   we introduce the B\"acklund transform for the Kaup-Newell spectral problem
 following   \cite{PSS} (see also
\cite{KillipVisan2021, SHW} and the references therein).

For $\lam\in \C$, $z\in \C_{++}=\{z\in \C:\,\, \re z>0, \, \im z>0\}$, and $ \eta=\begin{pmatrix}\eta_1\\\eta_2\end{pmatrix}\in \C^2\setminus\{0\}$,  define\footnote{Observe  that  for any constant~$C\neq 0$, $G_z (\eta)=G_z (C\eta)$, $\cS_z(\eta)=\cS_z(C\eta)$, and $ A( \lam; \eta, z)=A( \lam; C\eta, z)$.} 
 \begin{equation}  \label{defd}d_z (\eta)= z |\eta_1|^2+ \overline z |\eta_2|^2, \end{equation}
 \begin{equation}  \label{defS} G_{z} (\eta) = \frac {d_{\overline z} (\eta)}  {d_z (\eta)}, \quad \cS_{z} (\eta) = 2i (z^2- \overline z^2) \frac {\eta_1 \overline \eta_2} {d_z (\eta)}\,  , \end{equation}
and
\begin{equation}  \label{defMatrix} A( \lam; \eta, z)=  \begin{pmatrix}
\lam^2 G_{z} (\eta) - |z|^2&\frac {i\lam}{2} \cS_{z} (\eta) \\
-\frac {i\lam}{2} \overline {\cS_{z} (\eta) }  & -\lam^2 \overline {G_{z} (\eta)} +|z|^2
\end{pmatrix}.  \end{equation} 
Note that 
\begin{equation}\label{SLinfty}
|G_z(\eta)| = 1, \quad G_z(\eta)\neq -1, \quad {\rm and}\quad
 |\cS_{z} (\eta)|  \leq 4 \im z.\end{equation}
 It is  also easy to check that 
 \begin{equation}\label{symAB}
 A(- \lam; \eta, z)=\sigma_3A( \lam; \eta, z)\sigma_3, 
 \quad \overline{A( \bar\lam; \eta, z)}=-\sigma_1\sigma_3A( \lam; \eta, z)\sigma_3\sigma_1,\end{equation}
 and
 \begin{equation}\label{AB-0}
 A(z; \eta, z)=z(z^2-\bar z^2)\begin{pmatrix} \frac {\overline \eta_2}{d_z (\eta)} &0 \\ 0& \frac {\overline \eta_1}{{d_{\bar z} (\eta)}}\end{pmatrix} \begin{pmatrix} \eta_2&- \eta_1\\ \eta_2&- \eta_1\end{pmatrix},\end{equation}
 so that
 \begin{equation}\label{AB-1}
 A(z; \eta, z)\eta=0.\end{equation}
 Taking into account the form of $A(\lam; \eta, z)$, one readily deduces from \eqref{symAB}  and  \eqref{AB-1} that
 \begin{equation}\label{AB-2}
  \det A(\lam; \eta, z)=-(\lam^2-z^2)(\lam^2-\bar z^2).\end{equation}
 
\medskip

Let~$u\in \mathcal{S}(\R)$,~$z\in \C_{++}$ and let~$ \eta=\begin{pmatrix}\eta_1\\\eta_2\end{pmatrix}$
be  a non-zero smooth solution of   $L_u(z)\eta=0$.  Then the transformation $\psi\mapsto A(\lam;z,\eta)\psi$ maps solutions of  
the KN  spectral problem~$L_u(\lambda)\psi=0$ to solutions of the KN  spectral problem~$L_{u^{(1)}}(\lambda)\psi=0$ 
where $u^{(1)}$ is  the  B\"acklund transform of $u$ defined by
$$u^{(1)}=\cB_{z} (\eta) u,$$
with
  \begin{equation}
\label{backtransf} \cB_{z} (\eta) u = G_{z} (\eta) \big[-G_{z} (\eta) \, u+ \cS_{z} (\eta)  \big] .\end{equation}
This is a standard consequence of the properties \eqref{symAB}, \eqref{AB-0}, \eqref{AB-2} and of the following identity that can be easily checked by  direct computations (see for instance\ccite{BaPe}) :
\begin{equation}\label{derG1}\begin{aligned}\frac d {dx}G_{\lam} (\psi) &= - \frac{i}2G_\lam(\psi)\left(\left|\cS_\lam(\psi)\right|^2-uG_\lam(\psi)\overline{\cS_\lam(\psi)}
- \overline{uG_\lam(\psi)}\cS_\lam(\psi)\right)\\&= - \frac{i}2G_\lam(\psi)\left(|\cB_{\lam} (\psi) u|^2 - |u|^2\right).  \end{aligned} \end{equation}
Note that $\cB_{z} (\eta) u$ is in $\mathcal{S}(\R)$ and thanks to \eqref{SLinfty}, it admits the following $L^\infty$ bound:
 \begin{equation}\label{BLinfty}
\|\cB_{z} (\eta) u\|_{L^\infty (\R)}  \leq \|u\|_{L^\infty (\R)} +4 \im z.
\end{equation}
By straightforward computations, one can also show that
setting 
  \begin{equation}  \label{inversecomp} \eta^{(1)}=\begin{pmatrix}  \eta_1^{(1)}\\ \eta_2^{(1)}\end{pmatrix}, \quad 
  \eta_1^{(1)}=  \frac {\overline \eta_2} {d_{z} (\eta)}, \quad \eta_2^{(1)}=  \frac {\overline \eta_1} {d_{\overline z} (\eta)}, \end{equation}
gives a solution of the KN system
$L_{u^{(1)}}(z)\psi^{}=0$ and that
the
 B\"acklund transformation~$\cB_{z} (\eta^{(1)})$ is a left inverse of $\cB_{z} (\eta)$:
\begin{equation}  \label{inverse}  u= \cB_{z} (\eta^{(1)})  u^{(1)}. \end{equation}

\smallskip The key property of the B\"acklund transformation \eqref{backtransf}  is that it allows to add or to remove eigenvalues of the 
Kaup-Newell spectral problem. Namely, assume that~$\tilde a_u(\zeta)$ has a  zero $\zeta_1\in \C_{+}$ and let~$\eta\in L^2(\R, \C^2)\setminus\{0\}$ be the corresponding eigenfunction:~$L_u(\lam_1)\eta=0$, $\lambda_1=\sqrt{\zeta_1}\in \C_{++}$.
Then the Jost solutions of the spectral problem $L_{u^{(1)}}(\lam)\psi=0$ with~$u^{(1)}=\cB_{\lam_1} (\psi) u$, are related to the Jost solutions associated to the potential $u$ by:
\begin{equation}\label{JB-1} \begin{split}
\psi^ -_1(x,\lam; u^{(1)})&=  \frac {\lam_1} {\overline \lam_1}  \frac 1 {\lam^2 -  \lam^2_1}A( \lam; \eta(x),  \lam_1)\psi^ -_1(x,\lam; u),\\
\psi^+_2(x,\lam; u^{(1)})&=    -  \frac {\lam_1} {\overline \lam_1}  \frac 1 {\lam^2 -  \lam^2_1}A(\lam; \eta(x),   \lam_1)\psi^+_2(x,\lam; u),
\end{split}\end{equation}
and  
 \begin{equation}\label{JB-2} \begin{split}
\psi^ -_2(x,\lam; u^{(1)})&=  -\frac {\overline\lam_1} {\lam_1}  \frac 1 {\lam^2 -  \overline\lam^2_1}A(\lam; \eta(x),  \lam_1)\psi^ -_2(x,\lam; u),\\
\psi^+_1(x,\lam; u^{(1)})&=     \frac {\overline\lam_1} { \lam_1}  \frac 1 {\lam^2 - \overline \lam^2_1}A( \lam; \eta(x),   \lam_1)\psi^+_1(x,\lam; u).
\end{split}\end{equation} 

Combining \eqref{JB-1} with \eqref{AB-2} and \eqref{defa},  one obtains that
the function $\tilde a_{u^{(1)}}$ 
  associated to the potential~$u^{(1)}=\cB_{\lam_1} (\psi) u$ is given by \begin{equation}
\label{backtransfrela} 
\tilde a_{u^{(1)}}(\zeta)= \tilde a_{u} (\zeta)  \frac {\zeta_1(\zeta  -  \overline \zeta_1)}  {\overline\zeta_1(\zeta -  \zeta_1)}.\end{equation}
In view of \eqref{devolg} 
 and \eqref{mass1-0}, the relation \eqref{backtransfrela} implies  that
\begin{equation}
\label{masszero} M_\R(u^{(1)}) = M_\R(u) - 4 \arg \zeta_1, \end{equation}
and 
\begin{equation}
\label{energieszero}E_k(u^{(1)})=E_k(u)+\frac{2i}k \im\zeta_1^k, \quad \forall \, k\geq 1.
\end{equation}

We conclude this  subsection with the following lemma which  is a direct consequence of\refeq{derG1}:
  \begin{lemma}
\label{gen-propagation}
{\sl  For  $u\in \cS(\R)$ and $\lambda \in \C_{++}$, let $u^{(1)}= \cB_{\lam} (\psi^-_1) u$, where 
$\psi^-_1$ is the Jost solution of $L_u(\lam)\psi=0$ given by\refeq{Jost1}.
Then, for all $x \in \R$, one has,  with $\arg z \in  ]- \pi,   \pi[$, 
\begin{equation} \label{1-propagation}  
\int^x_{-\infty} |u^{(1)}(y)| ^2dy -2 \arg \overline{G_{\lam} (\psi^-_1) (x)}=\int^x_{-\infty} |u(y)| ^2dy -2\arg \lam^2.
\end{equation}
 Similarly, setting ${u}^{(2)}= \cB_{\lam} (\psi^+_2) u$, with~$\psi^+_2$ the Jost solution defined by\refeq{Jost1}, one has
 \begin{equation} \label{2-propagation}  
 \int_x^{+\infty} |u^{(2)}(y)| ^2dy -2\arg G_{\lam} (\psi^+_2) (x)
 =\int_x^{+\infty} |u(y)| ^2dy -2\arg \lam^2.
\end{equation}}
\end{lemma}


 \section{General structure of  the proof of Theorem\refer{Mainth}}\label {proofmainthH1} 
 \subsection{Strategy of proof}\label{genstrategy} 
 To prove Theorem\refer{Mainth}, we  shall  generalize the approach introduced in \cite{BaPe, BaPe2}. 
  By virtue of   the local   well-posedness result of   Herr\ccite{Herr}, it suffices to show that
 for any~$R>0$, there exists~$C=C(R)>0$ such that
\begin{equation}\label{b1H1}
\sup\limits_{0\leq t\leq T}\|u(t)\|_{H^{1}(\TT)}\leq C,
\end{equation}
for any solution  {$u\in C([0, T], H^1(\TT))$} of the DNLS equation with $u(0)\in C^\infty(\TT)$ satisfying~$\|u(0)\|_{H^1(\TT)}\leq R$.
To establish the bound  \eqref{b1H1},  we  shall argue by contradiction assuming    that there exists a sequence of initial data $(u_0^{(n)})$ in $C^\infty(\TT)$, bounded in~$H^1(\TT)$, and a sequence of times $(t_n) \subset \R_+$ such that, denoting $u_n$ 
the solution of the DNLS equation with initial data $u_0^{(n)}$ we have:~$u_n\in C([0, t_n], C^\infty(\TT)) $ and~$\|u_n(t_n)\|_{H^{1}(\TT)}\to +\infty$, as~$n$ goes to infinity. 
After eventually passing to a subsequence, we can assume that~$\|u_0^{(n)}\|_{L^2(\TT)}\to m$, with\footnote{The case of $m<8\pi$ is precluded by the global well-posedness result of \cite{BaPe2}.}
 $m\geq 8\pi$. 
 Setting 
\begin{equation}
\label{defsca} U^{(0)}_n(y)= \frac 1 {\sqrt{\mu_n}} u_n(t_n, \frac y { \mu_n}) \quad {\rm with}\quad \mu_n= \|u_n(t_n)\|_{\dot H^{1}(\TT)},\end{equation} 
we obtain a sequence of  $C^\infty$  functions satisfying, for all $n\in \N$, 
\begin{equation}
\label{estimatesboundary100}
U^{(0)}_n(y+\mu_n)=U_n^{(0)}(y), \, \,  \|U^{(0)}_n\|_{L^{2}(\mu_n\TT)}=\|u_0^{(n)}\|_{L^{2}(\TT)} \,\,\mbox{and} \,\,\|\partial_yU^{(0)}_n\|_{L^{2}(\mu_n\TT)}=1 \, .
 \end{equation}
We next fix a sequence~$(N_n)_{n\in \N }\subset \N^*$   such that
\begin{equation}
\label{constcondN} N_n \stackrel{n\to\infty}\longrightarrow \infty \andf \frac {\mu_n} {N_n} \stackrel{n\to\infty}\longrightarrow \infty \, , \end{equation}
and 
for each $n$, 
 split the interval $[0, \mu_n]$ into $N_n$ intervals  of the same length. Then we choose $k\in \{0, \dots, N_n-1\}$ so that
\begin{equation}\label{estimatesboundary000}
\|U^{(0)}_n\|^2_{H^{1}([\frac {k \mu_n} {N_n}, \frac {(k+1) \mu_n} {N_n}]
)} \leq \frac {\|u_0^{(n)}\|^2_{L^{2}(\TT)}+1} {N_n}, 
\end{equation}
and set 
\begin{equation}
\label{defext}V^{(1)}_n(y)=\chi \big(\frac {y} {R_n}\big)U^{(1)}_n(y) \chi \big(\frac {\mu_n-y} {R_n}\big) \, , \end{equation}
where  $U^{(1)}_n(y)= U^{(0)}_n\big(y+(k+\frac12)R_n\big)$, $\ds R_n= \frac { \mu_n} {N_n}$ and  $\chi$ is a $C^\infty$  function valued in the interval $[0, 1]$ and such that 
\begin{equation}
\label{defchi} \chi \equiv 1 \, \, \mbox{on} \, \,  \R_+  \andf  \chi \equiv 0 \, \, \mbox{on} \, \,  (-\infty, - \frac 1 2 ] \, .\end{equation}
 The constructed function $V^{(1)}_n$ is supported in $[- \frac   {R_n} 2\,\virgp   \,\mu_n+ \frac   {R_n} 2]$,  equal to  $U^{(1)}_n$
on $[0,   \mu_n]$,     and according to\refeq{estimatesboundary000}  satisfies 
\begin{equation}
\label{estimatesext}\|V^{(1)}_n\|^2_{H^{1}((-\infty, \frac {R_n} {2}] \cup [\mu_n-\frac {R_n} {2}, +\infty))} \stackrel{n\to\infty}\longrightarrow 0 \, . \end{equation}
Therefore,  
$$\lim\limits_{n\rightarrow \infty}\|V^{(1)}_n\|^2_{L^{2}(\R )}=m , \quad \lim\limits_{n\rightarrow \infty}\|V^{(1)}_n\|^2_{\dot H^1(\R )}=1.$$

The sequence $(V^{(1)}_n)_{n \in \N}$ has been already studied in \cite{BaPe2} where we proved 
that after eventually passing to a subsequence, it admits a profile decomposition of the following form (see Section  \ref{profileH1} for the precise statement):
\begin{equation}
\label{dec0}V^{(1)}_n(y)= \sum_{\ell=1}^{L_0}  \varphi_\ell(y-y^{(\ell)}_n)  + {\rm
r}^{(1)}_n(y), \quad \|{\rm
r}_n^{(1)}\|_{\dot H^{1}(\R)} \stackrel{n\to\infty}\longrightarrow 0,\end{equation} 
where $\ds 1 \leq L_0 \leq   \frac m {4 \pi}$,  and where  the profiles~$\varphi_\ell$ are in $H^1(\R)\setminus\{0\}$ and  satisfy $\tilde a_{\varphi_\ell} \equiv 1$ on~$\C_+$ (which in view of Lemma \ref{studyfirstpartimaginary} ensures that $M_\R(\varphi_\ell) \in 4 \pi \N$),  
 and where $|y_n^{(\ell)}-y_n^{(\ell')}|\stackrel{n\to\infty}\longrightarrow +\infty$, 
 $y^{(\ell)}_n\stackrel{n\to\infty}\longrightarrow +\infty$,~$\mu_n -y^{(\ell)}_n\stackrel{n\to\infty}\longrightarrow +\infty$,  for all~$\ell, \ell'$ with $\ell\neq \ell'$.
The goal will be to show that such a  profile  decomposition is in contradiction with the properties of the Kaup-Newell spectral problem\refeq{sp} described in the previous section.
To this end, we first  refine the above construction by introducing a family of $H^1$-functions $V_n^K$ depending on an additional parameter $K\in \N$,
$M^2\leq K\leq M^{-2}\mu_n$, for some large  fixed constant~$M$, so that the following properties hold: for all $n$ and $K$, the function $V_n^K$ is supported in
the interval $[-M, \mu_n+M]$, on $[0, \mu_n]$ it coincides up to a translation with~$U_n^{(0)}$, and it is of order $\frac1{\sqrt{\mu_n}}$ in a $M$-vicinity of $0$ and $\mu_n$ and of order~$\frac1{\sqrt{K}}$ in a~$\frac {\mu_n} K$ vicinity of these points (see\refeq{estimatesext1}). 
The decomposition \refeq{dec0}
ensures that there exists a family~$(y^{(\ell)}_n(K))\subset [0, \mu_n]$  with $\min\limits_{{\ell \neq \ell'} \atop{M^2\leq K\leq \frac{\mu_n}{M^2}}} |y^{(\ell)}_n (K)-  y^{(\ell')}_n(K)| \stackrel{n\to\infty}\longrightarrow \infty$, so that
\begin{equation}
\label{basic} V^{K}_n(y)=  \sum_{\ell=1}^{L_0}  \varphi_\ell(y-y^{(\ell)}_n (K)) + {\rm
r}^{(K)}_n (y),\end{equation}
with
$ \max\limits_{M^2\leq K\leq \frac{\mu_n}{M^2}}\|{\rm r}^{(K)}_n\|_{\dot H^1(\R)}\stackrel{n \to  \infty} \longrightarrow 0$.

\smallskip 
We next analyze the zeros
of the spectral coefficient  $\tilde a_{V^{K}_n}$. Arguing as in\ccite{BaPe, BaPe2}, we  show that, for any  $\frac \pi 2 < \theta< \pi$,  the function   $\tilde a_{V^{K}_n}$ admits at least $I_0=\sum_{\ell=1}^{L_0}\frac {M_\R(\varphi_\ell)} {4 \pi}$ zeros in the angle~$\{ \zeta\in \C_+:   \, \,  \theta < \arg \zeta< \pi \big\}$, 
provided  $n$   is sufficiently large and~$K\in  \big\{M^2, \dots, [\frac {\mu_n} {M^2}]\big\}$ (see Lemmas\refer{lemzerosH1I0}  below). 
Denoting  $(\zeta^{j}_n(K))_{1 \leq  j  \leq I_n^{(K)}}$  the zeros of $\tilde a_{V^{K}_n}$ in the angle~$\{ \zeta\in \C_+: \theta^* < \arg \zeta< \pi\}$ 
where $\theta^*$ is chosen accordingly to Proposition\refer{keygen},  and setting 
$$\alpha(K,n)=\max_{1 \leq  j  \leq I_n^{(K)}}\,  |\zeta^{j}_n(K)| \im \zeta^{j}_n(K),$$
we distinguish two regimes:
 \begin{enumerate}
\item $ \ds \mu_n \min_{K}  \alpha(K,n) \stackrel{n\to\infty}\longrightarrow \infty$,
\smallskip 
\item  up to a subsequence, $\ds \mu_n \min_{K}  \alpha(K,n) \stackrel{n\to\infty}\longrightarrow \beta \in \R_+$.
\end{enumerate}
In the first case, 
 we fix  $K=K_n$ in such a way that
$K_n\stackrel{n\to\infty}\longrightarrow \infty$ and  $ \frac {K_n}{\ds \mu_n \min_{K}  \alpha(K, n) }\stackrel{n\to\infty}\longrightarrow 0$, and in the second case
we choose  $K_n$ so that $  \ds  \alpha(K_n,n)= \min_{K}  \alpha(K,n)$.

\smallskip
 To get a contradiction, we   perform a detailed study of the monodromy matrix associated to~$V_n^{K_n}\big|_{[0, \mu_n]}$.  For that purpose, we start by  
 removing the zeros~$\zeta^{j}_n(K_n)$  by applying  to~$V_n^{K_n}$ recursively  the    B\"acklund transformation introduced in Paragraph\refer{defBacklund transformation}. 
Since the function~$\tilde a_{\tilde  {\rm  r}_n}$ corresponding to the resulting potential $\tilde  {\rm  r}_n$ does not vanish in the angle~$\{\zeta\in \C_+:   \,  \theta^* < \arg \zeta< \pi\}$,   we can control $\tilde {\rm r}_n$   by  the generalization of  Guo-Wu estimate\refeq{controlhomH1spect}, which together with \eqref{asympek}   leads to the bound  
$$\|\tilde  {\rm r}_n\|^2_{\dot H^{1}(\R)} \lesssim    \alpha(n)+  \frac {1}{\mu_n},$$ 
where  $\alpha(n)= \alpha(K_n, n)\stackrel{n \to  \infty} \longrightarrow 0$ (see Lemma \ref{rem-as}).
  The smallness of $\tilde  {\rm
r}_n$ in~$\dot H^1(\R)$ allows us to control  the Jost solutions corresponding to the potential $\tilde{\rm
r}_n$  by  a suitable WKB ansatz.   Returning to the spectral problem $L_{V_n^{K_n}}(\lam)\psi=0$ by means of the connection formulae \eqref{JB-1} and  \eqref{JB-2}, we get an   explicit  approximation of the monodromy matrix associated to~$V_n^{K_n}\big|_{[0, \mu_n]}$ and then  show that this approximation
is in contradiction with the conservation of the Floquet discriminant in the first regime and with Lemma\refer{DN} in the second regime.
Compared to the case of $m<8\pi$, where one has only one simple zero~$\zeta_n^1(K_n)$,  the central difficulty here is the eventual presence of multiple or close zeros.  
One of the key   arguments   
allowing us to treat this issue is the observation that, accordingly to Lemma \ref{gen-propagation}, the B\"acklund transformations corresponding to the zeros whose arguments are close to $\pi$, 
preserve the $L^2$ smallness of the potential in  a $\ds \frac {\mu_n} {K_n}$ vicinity of~$0$ and~$\mu_n$.

\subsection{Profile decomposition}\label{profileH1}
In  this section, we recall  
the  profile decomposition of  the sequence~$(V^{(1)}_n)_{n \in \N}$   established in \cite{BaPe2}. 
\begin{theorem}[\cite{BaPe2}] 
\label{result1newth}
{\sl Under the notations of Section \ref{genstrategy},
there exist  an integer~$\ds 1 \leq L_0 \leq   \frac m {4 \pi}$,  a family of  profiles~$(\varphi_\ell)_{1 \leq \ell \leq L_0}$ in $H^1(\R)\setminus\{0\}$ with~$\tilde a_{\varphi_\ell}\equiv 1$,
for all $\ell\in \{1, \dots, L_0\}$, and a family of orthogonal cores\footnote{Following the vocabulary of P. G\'erard in\ccite{pgerard1}, we designate by a core $\underline{y}^{(\ell)}$ any   real sequence~$(y^{(\ell)}_n)_{n \in \N} $.}~$(\underline{y}^{(\ell)})_{1 \leq \ell \leq L_0}$,   
 in the sense that for all~$\ell \neq \ell'$, we have~$|y^{(\ell)}_n-  y^{(\ell')}_n| \stackrel{n\to\infty}\longrightarrow \infty$, 
 such that,
up to a subsequence,   
\begin{equation}\label{decH1} V^{(1)}_n(y)= \sum_{\ell=1}^{L_0}  \varphi_\ell(y-y^{(\ell)}_n)  + {\rm
r}^{(1)}_n(y) ,\end{equation}
where $\|{\rm
r}^{(1)}_n\|_{\dot H^{1}(\R)}\stackrel{n\to\infty}\longrightarrow  0$,
 and where in addition,  
 for all $1 \leq \ell \leq L_0$, 
\begin{equation}\label{addition}  y^{(\ell)}_n\stackrel{n\to+\infty}\longrightarrow \infty \andf \mu_n- y^{(\ell)}_n \stackrel{n\to\infty}\longrightarrow +\infty.\end{equation} 
}
\end{theorem}
 \begin{remark}\label{rem-masse}
{\sl Note   that,   as $n\rightarrow \infty$, we have
 \begin{eqnarray}\label{ortogonalth2H1} 
 \| V^{(1)}_n\|_{L^2(\R)}^2&=&\sum_{\ell=1}^{L_0} \|  \varphi_\ell\|_{L^2(\R)}^2+\| 
 {\rm
r}^{(1)}_n\|_{L^{2}(\R)}^2+o(1),\\
\label{ortogonalth2H2} \|V^{(1)}_n\|_{\dot H^{1}(\R)}^2&=&\sum_{\ell=1}^{L_0} \|\varphi_\ell\|_{\dot H^{1}(\R)}^2+o(1) .
\end{eqnarray}
}
\end{remark}
\medbreak


 \subsection{Refined profile decomposition}
As mentioned in Section \ref {genstrategy},  to proceed further  we need to introduce a refined version of the construction of Subsection \ref{profileH1}.
We fix once forever a large  integer $M\gg 1$. Then for each\footnote{We assume that $n$ is sufficiently large so that  $[\frac {\mu_n} {M^2}]\geq M^2$.}~$K\in  \Big\{M^2, \dots, [\frac {\mu_n} {M^2}]\Big\}$,  we split the interval~$[0, \mu_n]$ into~$K$ subintervals of the same length and  
 choose $k_0=k_0(K)\in \{0, \dots, K-1\}$ such that
\begin{equation}\label{estimatesboundary0}
\|U^{(1)}_n\|^2_{H^{1}([\frac {k _0\mu_n} {K}, \frac {(k_0+1) \mu_n} {K}]
)} \leq \frac {\|u_0^{(n)}\|^2_{L^{2}(\TT)}+1} {K} \, .
\end{equation}
Next we take the interval $I^0_n(K)=\Big[\frac {(k_0+\frac14)\mu_n} {K}, \frac {(k_0+\frac34) \mu_n} {K}\Big]$ and divide it into $K_1= [\frac {\mu_n} {8MK}]$ subintervals of the same size and choose again $k_1\in \{0, \dots, K_1-1\}$
such that
\begin{equation}\label{estimatesboundary1}
\|U^{(1)}_n\|^2_{H^{1}(I_n^1(K))} \leq \frac{\|U^{(1)}_n\|^2_{H^{1}(I^0_n(K))}} {K_1} \lesssim \frac {\|u_0^{(n)}\|^2_{L^{2}(\TT)}+1} {\mu_n},
 \end{equation}
where $I_n^1(K)=\Big[\frac{\mu_n}K\left(k_0+\frac14 +\frac{k_1}{2K_1}\right),
\frac{\mu_n}K\left(k_0+\frac14 +\frac{k_1+1}{2K_1}\right)
\Big]$. 

\medskip Finally,  we set 
\begin{equation}
\label{DefextK} 
U^{K}_n(y)=U^{(1)}_n(y+\tilde y_n(K)),\quad 
V^{K}_n(y)= \chi \big(\frac {y} {M}\big)U^{K}_n(y) \chi \big(\frac {\mu_n-y} {M}\big)\, ,\end{equation}
where $\tilde y_n(K)$ is the center of the interval $I_n^1(K)$ and $\chi$ is the function  defined by\refeq{defchi}.
The function $V^{K}_n$ is supported in $[-\frac M2, \mu_n+\frac M 2]$,  coincides with $U^{K}_n$ on $[0, \mu_n] $ and satisfies
\begin{eqnarray}
\label{estimatesext1} 
\|V^{K}_n\|^2_{H^{1}((-\infty, M] \cup [\mu_n-M, +\infty))} &\lesssim& \frac {\|u_0^{(n)}\|^2_{L^{2}(\TT)}+1} {\mu_n},\\
\label{estimatesext2}\|V^{K}_n\|^2_{H^{1}((-\infty, \frac{\mu_n}{4K}] \cup [\mu_n-\frac{\mu_n}{4K}, +\infty))} &\lesssim & \frac {\|u_0^{(n)}\|^2_{L^{2}(\TT)}+1} {K},
\end{eqnarray}
for all $n$ sufficiently large and all $K\in \big\{M^2, \cdots, [\frac {\mu_n} {M^2}]\big\}$. As a consequence of \eqref{estimatesext1}, we get
\begin{eqnarray} \label{equi1} |M_\R(V^{K}_n)-M_{\TT}(u_0^{(n)})| &\lesssim & \frac 1 {\mu_n},\\ \label{equi2}|P_\R(V^{K}_n)| \lesssim    
\frac1{\mu_n} |P_{\TT}(u_0^{(n)})|+ \frac1{\mu_n}
&\lesssim &\frac 1{\mu_n},\\ 
\label{equi3}|E_\R(V^K_n)|\lesssim
\frac1{\mu^2_n} |E_{\TT}(u_0^{(n)})|+ \frac1{\mu_n}
&\lesssim &\frac 1{\mu_n}. \end{eqnarray}
According to\refeq{decH1}, we write
\begin{equation}
\label{DecKH1} V_n^K(y)= \sum_{\ell=1}^{L_0}  \varphi_\ell(y-y^{(\ell)}_n (K)) +
{\rm
r}^{(K)}_n (y),\end{equation}
where $y^{(\ell)}_n(K)\in (0, \mu_n)$ is given by    \begin{equation}
\label{DefcoreK}
 y^{(\ell)}_n(K) = \left\{
\begin{array} {ccl}
 y^{(\ell)}_n-\tilde y_n(K) \quad \mbox{if} \quad  y^{(\ell)}_n \geq \tilde y_n(K)\\
y^{(\ell)}_n+\mu_n-\tilde y_n(K)\quad \mbox{if} \quad  y^{(\ell)}_n < \tilde y_n(K)   .  
\end{array}
\right.
\end{equation}
By virtue of Theorem\refer{result1newth}, this implies that \begin{equation}
  \label{orthKl} \minetage {\ell \neq \ell'} {M^2\leq K\leq \frac{\mu_n}{M^2}} |y^{(\ell)}_n (K)-  y^{(\ell')}_n(K)| \stackrel{n\to\infty}\longrightarrow \infty.\end{equation}
 Furthermore, Proposition \ref{car-sabis} ensures that 
\begin{equation}
\label{estimpK}  \min\limits_{M^2\leq K\leq \frac{\mu_n}{M^2}\atop 1 \leq \ell \leq L_0} |\tilde y_n(K)- y^{(\ell)}_n| \stackrel{n\to\infty}\longrightarrow \infty\, .\end{equation}
Indeed,   assuming that \eqref{estimpK}  does not hold, we infer that there exist $1 \leq \ell^\star\leq L_0$ and subsequences $(n_j)_{j\in  \N}$, $(K_j)_{j\in  \N}$ such  that 
\begin{equation} 
\label{contorth1}  \tilde y_{n_j}(K_j)- y^{(\ell^\star)}_{n_j} \stackrel{j\to\infty}\longrightarrow y^\star\in \R.\end{equation} 
Since $\|{\rm r}^{(1)}_n\|_{\dot H^{1}(\R)}\stackrel{n\to\infty}\longrightarrow  0$ and $ |y^{(\ell)}_n-  y^{(\ell')}_n| \stackrel{n\to\infty}\longrightarrow \infty$,  for all~$\ell \neq \ell'$, it follows  from\refeq{defext}, \refeq{decH1}, \eqref{estimatesboundary1} and \eqref{contorth1} that  
 \begin{equation} 
\label{contorth} \|\varphi_{\ell^\star} (\cdot- y^{(\ell^\star)}_{n_j})\|_{L^{2}(I_{n_j}^1(K_j))} \stackrel{j\to\infty}\longrightarrow  0.
\end{equation} 
In view of \eqref{contorth1}, this means that $\varphi_{\ell^\star}$ vanishes on an interval of a size of order $1$, which contradicts Proposition \ref{car-sabis},  and therefore, completes the proof of the claim\refeq{estimpK}.

\medskip Combining\refeq{estimpK} with  Theorem\refer{result1newth},  we obtain 
\begin{equation}\label{DecKH11rem}
  \max\limits_{M^2\leq K\leq \frac{\mu_n}{M^2}}\|{\rm r}^{(K)}_n\|_{\dot H^1(\R)}\stackrel{n \to  \infty} \longrightarrow 0.\end{equation}
  By virtue of  \eqref{orthKl}, this implies  
\begin{equation}\label{massrnK0}
\max\limits_{M^2\leq K\leq \frac{\mu_n}{M^2}}\Big|\|V_n^K\|^ 2_{L^2(\R)}-\|{\rm r}^{(K)}_n\|^ 2_{L^2(\R)}-\sum_{\ell=1}^{L_0} \|\varphi_\ell\|_{L^2(\R)}^2\Big|\stackrel{n \to +\infty}\longrightarrow 0, \end{equation}
which in view of \eqref{equi1} gives  
\begin{equation}\label{massrnK1}
\max\limits_{M^2\leq K\leq \frac{\mu_n}{M^2}}\Big|M_\R({\rm r}^{(K)}_n)-m +\sum_{\ell=1}^{L_0} M(\varphi_\ell)\Big|\stackrel{n \to  \infty} \longrightarrow 0.
\end{equation}


 \section{Preliminary results}\label{Preliminary}
 \subsection{Study of the zeros of  $\tilde a_{V_n^K}$} 
Consider  the functions $\tilde a_{V_n^K}$. From the decomposition\refeq{DecKH1}, \eqref{orthKl}   and  \eqref{DecKH11rem},  we deduce the following result.
\begin{lemma}
\label{lemzerosH1I0}
{\sl    For all   $\frac \pi 2 < \theta< \pi$, there exists $N(\theta)$  such that  denoting as before, $n(V_n^K; \theta)$ the number of zeros of 
$\tilde a_{V_n^K}(\zeta)$  in the angle $\{ \zeta\in \C_+:   \, \,  \theta < \arg \zeta< \pi \big\}$, counted with their multiplicity,
we have\footnote{Observe that, in light of \refeq{controbound-1}, $n(V_n^K; \theta)$ is uniformly bounded with respect to $n$ and $K$.}
\begin{equation}
\label{controboundbis} n(V_n^K; \theta) \geq I_0=\sum_{\ell=1}^{L_0}\frac {M_\R (\varphi_\ell)} {4 \pi}  \, ,\end{equation} for all $n \geq N(\theta)$  and $K\in  \big\{M^2, \dots, [\frac {\mu_n} {M^2}]\big\}$.}\end{lemma}  

\medbreak
 \begin{proof}
We  prove  Lemma\refer{lemzerosH1I0} by proceeding as  in\ccite{BaPe, BaPe2}.  
  To this end,  we first recall the following proposition  from\ccite{BaPe}, which  guarantees   the closeness of the functions $\tilde  a_{V_n^K}$ and~$\tilde  a_{{\rm r}^{(K)}_n}$. 
   \begin{proposition} [\cite{BaPe}] 
 \label{preturbationspectrumgen}
{\sl Let   $(V^{(\ell)})_{1 \leq \ell \leq L}$ be a finite family of functions in $H^1(\R)$ satisfying~$\tilde  a_{V^{(\ell)}}\equiv 1$ on $\C_+$. For $\underline{c}=(c_\ell)$,~$\underline{y}=(y_\ell)$  in~$\R^L$,  we denote $\ds u_{\underline{c}, \underline{y}}(y)= \sum_{\ell=1}^{L} e^{i c_\ell} V^{(\ell)}(y-y_\ell)$.
 Then,   for all~$\ds 0<\delta<\frac\pi2$ and  all~$ \alpha>0$, we have, 
  \begin{equation} 
\label{detyest3impbis} \tilde a_{u_{\underline{c},\underline{y}}+{ r}} (\zeta) - \tilde a_{{ r}} (\zeta) \longrightarrow 0\, ,
\end{equation} as $\ds \min_{\ell \neq \ell'}| y_\ell - y_{\ell'}| \rightarrow \infty$    and  $\|{ r}\|_{L^4}\rightarrow 0$,
 ${ r}\in L^2(\R) \cap L^4(\R)$ with $\|{ r}\|_{L^2(\R)} \leq \alpha$, uniformly with respect to  $\zeta\in\Gamma_\delta$ and $\underline{c}\in \R^L$.} 
\end{proposition}

 \medbreak  

In order to establish Lemma \ref{lemzerosH1I0},   we argue by contradiction assuming that there  exist~$\frac \pi 2 < \theta_0 < \pi$ and a subsequence 
$(V_{n_j}^{K_j})_{j \in \N}$ such that~$n(V_{n_j}^{K_j}, \theta_0)< I_0$, for all $j\in \N$. Without loss of generality, we can assume that there exists $\cJ$ such that, for all~$j \geq \cJ$, $n(V_{n_j}^{K_j};\theta_0)=I'$, with $0\leq I'<I_0$.
 Furthermore, if $I'>0$, denoting $\zeta_{j}^\ell$, $\ell=1, \dots, I'$, the zeros of $\tilde a_{V_{n_j}^{K_j}}(\zeta)$  in the angle 
$\{ \zeta\in \C_+:   \, \,  \theta_0 < \arg \zeta< \pi \big\}$,
we can assume that 
 for all~$1 \leq \ell \leq I'$, there exists $\theta_0 \leq \alpha_\ell \leq \pi $ so that 
 $$( \arg \zeta^{\ell}_j)_{1 \leq \ell \leq I'} \stackrel{j \to +\infty}\longrightarrow (\alpha_1, \cdots, \alpha_{I'}).$$
We set $\theta^*=\theta_0$ if $I'=0$, and 
$$\theta^\star=\left\{
\begin{array} {ccl}
\max \{\alpha_\ell, \alpha_\ell < \pi\} & & \mbox{if} \, \, \exists \, \alpha_\ell < \pi\\\
\theta_0 & & \mbox{if not}\, ,
\end{array}
\right.$$
in the case where  $I'>0$.
Proposition \ref{stability0} and Lemma \ref{lem-lowb} together with \eqref{equi1} ensure  that for all $\theta^\star<   \theta < \pi$,   there exists a positive constant~$C_\theta$    such that for all $j$ sufficiently large we have:
 \begin{equation}
\label{estnew}   \frac1 {C_\theta} \leq \Big|\frac  1 {\tilde  a_{V_{n_j}^{K_j}}(\zeta)}  \Big| \leq  {C_\theta} , \quad \forall \, \zeta\in e^{i\theta}\R_+.
\end{equation}

  \medskip   Combining \eqref{estnew} with   Proposition\refer{preturbationspectrumgen},   \eqref{orthKl} and \eqref{DecKH11rem},   we deduce  that, for all $\theta_0<   \theta < \pi$,  
\begin{equation}
\label{justifnew}    
 \sup\limits_{{\zeta\in \C_+\atop \arg \zeta=\theta}}
 \Big|1- \frac   {\tilde   a_{{\rm r}^{(K_j)}_{n_j}}(\zeta)} {\tilde  a_{V_{n_j}^{K_j}}(\zeta)}  \Big| \stackrel{j \to +\infty}\longrightarrow 0 \,. \end{equation}
In particular,  $\tilde   a_{{\rm r}^{(K_j)}_{n_j}}$ does not vanish on the ray $e^{i\theta} \R_+$ provided $j$ is sufficiently large, which according to 
Lemma\refer{intrel}  and Corollary\refer{coruse} iii), implies   that 
$$n(V_{n_j}^{K_j};\theta)\geq \frac{1}{2i\pi }\int\limits^{+ \infty \, e^{i\theta}}_0  \left(\frac {\tilde a'_{V_{n_j}^{K_j}}(s)}  {\tilde  a_{V_{n_j}^{K_j}}(s)} -
\frac {\tilde a'_{{\rm r}^{(K_j)}_{n_j}}(s)}  {\tilde  a_{{\rm r}^{(K_j)}_{n_j}}(s)}\right) ds
+\frac1{4\pi}\left({\|V_{n_j}^{K_j}\|^ 2_{L^2(\R)}}-\|{\rm r}^{(K_j)}_{n_j}\|^ 2_{L^2(\R)}\right)\,.
$$
By virtue  of  \eqref{justifnew}, we have
$$\left|\int^{+ \infty \, e^{i\theta}}_0 \left(\frac {\tilde a'_{V_{n_j}^{K_j}}(s)}  {\tilde  a_{V_{n_j}^{K_j}}(s)} -
\frac {\tilde a'_{{\rm r}^{(K_j)}_{n_j}}(s)}  {\tilde  a_{{\rm r}^{(K_j)}_{n_j}}(s)}\right) ds\right| \stackrel{j \to +\infty}\longrightarrow 0,$$
which together with \eqref{massrnK0} ensures that,
 for all $\theta\in [\theta_0, \pi)$,  there exists $\cJ(\theta)$ such that
 $$n(V_{n_j}^{K_j};\theta)\geq I_0, \quad \forall \, j\geq \cJ(\theta).$$
 This contradicts
 our assumption $ n(V_{n_j}^{K_j};\theta_0)<I_0 $, thereby completing
the proof of Lemma\refer{lemzerosH1I0}.
 \end{proof}  

\bigbreak

We next fix   $m_0>m-\sum_{\ell=1}^{L_0} M_\R(\varphi_\ell)$ and, according to Proposition\refer{keygen}, set~$\varepsilon_0= \varepsilon(m_0)$. For $\ds \theta^*= \frac \pi 2 + \varepsilon_0$, we   denote by 
$\zeta^{j}_n(K)$, $1 \leq  j  \leq I_n^{(K)}$,  the zeros of the function~$\tilde  a_{V_n^K}$ in the angle~$\{ \zeta\in \C_+:   \, \,  \theta^* < \arg \zeta< \pi \big\}$ counted with their multiplicity, and   
 define \begin{equation}
\label{defreg} \alpha(K,n)=\max_{1 \leq  j  \leq I_n^{(K)}}\,  |\zeta^{j}_n(K)| \im \zeta^{j}_n(K) \, .\end{equation} 
Note that since $\theta^*>\frac \pi 2$, Lemma \ref{lemmcont} guarantees the boundedness of  the family $(\zeta_n^j(K))_{j,n, K}$.

\medskip
 \begin{lemma}\label{rem-as}
{\sl Under the above notations, one has \begin{equation}
\label{limbis}   \max_{M^2\leq K\leq \frac{\mu_n}{M^2}} \alpha(K,n)\stackrel{n\to\infty}\longrightarrow 0  \,. \end{equation} }
\end{lemma}
 \begin{proof} If \eqref{limbis} does not hold, then  there exists  a subsequence~$(\zeta_{\ell})_{\ell \in \N}=(\zeta^{j_\ell}_{n_\ell}(K_\ell))_{\ell \in \N}$ such that~$\theta^*=\frac \pi 2 + \varepsilon_0\leq  \arg \zeta_{\ell} < \pi$ and,  for all $\ell \in \N$, we have   
\begin{equation}\label{intkey1}|\zeta_{\ell}| \im \zeta_{\ell} \geq  \alpha_0 > 0 \with \tilde a_{V_{n_\ell}^{K_\ell}}(\zeta_\ell)=0.\end{equation} 
After eventually passing to   a subsequence, we can assume that $$\zeta_{\ell} \stackrel{\ell\to\infty}\rightarrow \zeta^*   \with |\zeta^*| \im \zeta^* \geq  \alpha_0 >0\, .$$ Consequently, $\zeta^* \neq 0$ and $\theta^*\leq \arg \zeta^* < \pi$.  
It then follows from Proposition\refer{preturbationspectrumgen},  \eqref{orthKl} and\refeq{DecKH11rem}  that 
\begin{equation*}\label{intkey}  \tilde a_{{\rm
r}^{(K_\ell)}_{n_\ell}}(\zeta_\ell) \stackrel{\ell \to +\infty}\longrightarrow 0,
\end{equation*}  which leads to contradiction 
since in view of\refeq{eq6} and  \eqref{DecKH11rem}, we have
$$ | \tilde a_{{\rm r}^{(K_\ell)}_{n_\ell}}(\zeta_\ell)| \stackrel{\ell\to +\infty}\longrightarrow 1.$$
 \end{proof}
\medbreak

In what follows, we shall    distinguish
two  cases depending on whether :  
  \begin{equation}
\label{case1}  \mu_n \min_{K}  \alpha(K,n) \stackrel{n\to\infty}\longrightarrow \infty\, ,\end{equation}
or   up to a subsequence, \begin{equation}
\label{case2}  \mu_n \min_{K}  \alpha(K,n) \stackrel{n\to\infty}\longrightarrow \beta \in \R_+\, .\end{equation}
In the first case, we fix  $K=K_n$ in such a way that
\begin{equation}
\label{case1bis}  K_n\stackrel{n\to\infty}\rightarrow \infty \andf   \frac {K_n}{\ds \mu_n \min_{K}  \alpha(K, n) }\stackrel{n\to\infty}\rightarrow 0  \, .\end{equation}
In the second case, we take $K=K_n$ where the minimum is achieved:
\begin{equation}
\label{case2bis}   \alpha(K_n,n) = \min_{K}  \alpha(K,n)   \,.\end{equation}
In both  cases, for the sake of simplicity, we  denote  $ \alpha(K_n, n)$ by $  \alpha(n) $, $V^{K_n}_n$ by $V_n$,~$U^{K_n}_n$ by~$U_n$,  $y^{(\ell)}_n (K_n)$
by $\wt y^{(\ell)}_n$,  $\zeta_{n}^{j} (K_n) $ by $\zeta_{n}^{j}$, and  $I_n^{(K_n)}$ by $I_n$.

\bigbreak  

Without loss of generality,  after an eventual subsequence extraction,
we can assume  that there exists an integer $N$ so that, for all~$n\geq N$, $I_n=I\geq I_0$, and that, for all~$1\leq j \leq I$,
\begin{equation}\label{eq:red}
\zeta_{n}^{j} \stackrel{n\to\infty}\rightarrow \zeta^j \with |\zeta^1|\leq |\zeta^2|\leq \cdots\leq |\zeta^{I}|,\quad      \arg \zeta_{n}^{j} \stackrel{n\to\infty}\rightarrow \varphi^{j} \with  \theta^* \leq  \varphi^{j} \leq  \pi \, .\end{equation}
Furthermore, we can assume that there exists $I_1\in \{1, \dots , I\}$ so that 
\begin{equation}\label{def-I_1}
|\zeta^{I_1}_n| \im \zeta^{I_1}_n=\max\limits_{1\leq j\leq I}\,  |\zeta^{j}_n| \im \zeta^{j}_n, \quad \forall \,\, n\geq N,
\end{equation}
and setting  $z^j_n= \frac {\zeta_n^j-\zeta_n^{I_1}}{\im \zeta_n^{I_1}}$ in the case of \eqref{case1} and 
$z^j_n  =\frac {\zeta_n^j-\zeta_n^{I}}{\im \zeta_n^{I}}$ in the case of \eqref{case2},
we have for all $j=1, \dots, I$, either
\begin{equation}\label{zj-0}
|z^j_n|\stackrel{n\to\infty}\longrightarrow \infty,
\end{equation}
or 
\begin{equation}\label{zj-0-1}
z^j_n\stackrel{n\to\infty}\longrightarrow z^j\in \C.
\end{equation}
Note that by Lemma \ref{rem-as},
$\zeta_j\in \R_-$ for all $1\leq j\leq I$.
We claim that the number of the non-zero $\zeta^j$'s is greater or equal to $I_0$ defined by \eqref{controboundbis}.

\begin{lemma}
\label{lemzerosH1}
{\sl    One has:
\begin{equation}\label{zerosboundbis}
\zeta^j\neq0, \quad \forall  \, \,I-I_0+1\leq j\leq I.
\end{equation}}\end{lemma}   \begin{proof}  
We proceed by contradiction, arguing as in \cite{BaPe}, Lemma 3.4. Assuming that $ \zeta^j = 0, \, \, \forall \, 1\leq j \leq J$ with $I-I_0<J\leq I$, we remove the zeros $\zeta_n^j$, $j=1,  \dots, J$, one by one by applying the corresponding B\"acklund transformations. Denoting $V_{n, j}$ the potential obtained after removing $j$ zeros, we have:
\begin{equation}\label{***1}
 \begin{aligned} V_{n, 0} &=  V_{n}\, , \\  V_{n, j}&= \cB_{\lam_{n}^{j}} (\psi_{n}^{j} ) V_{n, j-1}  \, , \, \, j=1,\cdots, J,\end{aligned} 
 \end{equation}
 or explicitly,
  \begin{equation}\label{defseq}  V_{n, j}  = G_{\lam_{n}^{j}} (\psi_{n}^{j})\big[-G_{\lam_{n}^{j}} ( \psi_{n}^{j})V_{n, j-1}+ \cS_{\lam_{n}^{j}} (\psi_{n}^{j} )  \big], 
  \end{equation}
where 
 $\lam_n^j=\sqrt{\zeta_n^j}\in \C_{++}$, and $\psi_{n}^j\in H^1(\R, \C^2)\setminus\{0\}$, 
 $L_{V_{n, j-1}}(\lam_{n}^{j})\psi_{n}^{j}=0$.

\medskip  By virtue of\refeq{masszero},   one has  
 \begin{equation}\label{mass2}
 \| V_{n, j}\|_{L^2(\R)}^2= \| V_{n}\|_{L^2(\R)}^2- 4 \sum^{j}_{i=1}  { \arg} (\zeta_{n}^{i}), \quad j=1, \dots, J.
  \end{equation}   Since $\big|G_{\lam_{n}^{j}} (\psi_{n}^{j})\big|=1$,  the above identity implies   that
 there exists a positive constant~$C$ such that, for all    $n,j$,
 \begin{equation}\label{lpS}
 \|S_{\lam_{n}^{j}} (\psi_{n}^{j})\|_{L^2(\R)}\leq C.\end{equation} 
 From \eqref{***1} and \eqref{defseq}, we readily deduce that
  $$\begin{aligned}
V_{n,   J}  & =  (-1)^{J}G^2_{\lam_{n}^{ J} } (\psi_{n}^{ J}) \cdots G^2_{\lam_{n}^{1}} (\psi_{n}^{1})V_{n}
\\ &   \quad    + \sum^{ J} _{j=1} (-1)^{  J -j} G^2_{\lam_{n}^{J}} (\psi_{n}^{ J})
\cdots G^2_{\lam_{n}^{j+1}} (\psi_{n}^{j+1}) G_{\lam_{n}^{j}} (\psi_{n}^{j})\cS_{\lam_{n}^{j}} (\psi_{n}^{j}) . \end{aligned}
$$
Recalling \eqref{DecKH1} and setting
$$\beta_{n, \ell}= \underbrace{\big((-1)^{ J}G^2_{\lam_{n}^{J} (K)} (\psi_{n}^{ J}) \cdots G^2_{\lam_{n}^{1}} (\psi_{n}^{1})\big)}_{\cG_{n}}  (\wt y^{(\ell)}_n)\, ,$$
we get, as in \cite{BaPe}, that
\begin{equation}
\label{newdec0}   V_{n, J} = \sum_{\ell=1}^{L_0}   \beta_{n, \ell} \,  \varphi_\ell(\cdot-\wt y^{(\ell)}_n ) +\cR_{n},\end{equation}
 where 
 \begin{equation}
 \label{newdec1}
   \|\cR_{n}\|_{L^p(\R)}\stackrel{n \to  + \infty}\longrightarrow0 , \quad \forall \,\,2<p<\infty.
   \end{equation}
    Indeed, we have  
$$\cR_{n}=\cR_{n, 1}+\cR_{n, 2}+\cR_{n, 3} , $$
with
\begin{eqnarray*} \cR_{n, 1}(y) &=& \cG_{n}(y) \, {\rm
r}_n(y)\, , \\ \cR_{n, 2}(y) &=& \sum^{J} _{j=1} (-1)^{J -j} G^2_{\lam_{n}^{ J}} (\psi_{n}^{ J})
\cdots G^2_{\lam_{n}^{j+1}} (\psi_{n}^{j+1}) G_{\lam_{n}^{j}} (\psi_{n}^{j})\cS_{\lam_{n}^{j}} (\psi_{n}^{j})
\\ \cR_{n, 3}(y) &=& \sum^{L_0} _{\ell=1} (\cG_{n}(y)-\cG_{n}(\wt y^{(\ell)}_n ))\varphi_\ell(y-\wt y^{(\ell)}_n)\,   . \end{eqnarray*}
  Since $|\cG_{n}|=1$, \eqref{DecKH11rem} and \eqref{massrnK1} ensure that,
for all $p>2$,  $$\ds \|\cR_{n, 1}\|_{L^{p}(\R)}\stackrel{n \to  + \infty} \longrightarrow  0. $$
On the other hand, in view of \eqref{SLinfty} and \eqref{lpS},   we have
$$\|\cR_{n, 2}\|_{L^p(\R)}\leq \sum^{J} _{j=1} \|\cS_{\lam_{n}^{j} } (\psi_{n}^{j})\|_{L^p(\R)}\lesssim  \max_{1 \leq j \leq J} \big(\im \lam^{j}_n\big)^{1-\frac 2 p}
\stackrel{n \to  + \infty} \longrightarrow  0. $$
Finally, to estimate  the remaining  term $\cR_{n, 3}$, we  make use of\refeq{derG1} which gives :  
\begin{equation} \label{pointwiseest}
\big |\frac d {dy}\cG_{n} (y)\big| \lesssim  \Big(\sum^{J} _{j=1} |\lam^{j}_{n}|\Big)^2+ |V_{n} (y)|\sum^{J} _{j=1} |\lam_{n}^{j}|  \, .\end{equation}
This ensures  that, for all $1\leq  \ell \leq L_0$ and $y\in \R$, 
$$ \big|\cG_{n}(y+\wt y_n^{(\ell)})-\cG_{n}(\wt  y^{(\ell)}_n) \big| \lesssim  \Big(\sum^{J} _{j=1} |\lam^{j}_{n}|\Big)^2 \, |y| + \sum^{J} _{j=1} |\lam^{j}_{n}| \, |y|^{\frac 1 2} \|V_{n}\|_{L^2(\R)} \stackrel{n \to  + \infty} \longrightarrow  0. $$
  Consequently,
$$\|\cR_{n, 3}\|_{L^p(\R)}\stackrel{n \to  + \infty} \longrightarrow 0, \quad \forall\,\, 2\leq p<\infty,$$
 which ends the proof of\refeq{newdec1}.  
 
   \medbreak 
 To arrive at a contradiction, we repeat  the arguments we have used above to establish the bound\refeq{controboundbis}. 
 Let $\theta\in ]\theta^*, \pi[ $ so that $\theta >\max \{\varphi^j:  \varphi^j< \pi, \, J+1\leq j\leq I\}$. By Proposition \ref{stability0}, Lemma \ref{lem-lowb},   \eqref{equi1} and \eqref{mass2}, we have
 \begin{equation}
\label{estnew1}       \frac1{C_{\theta}} \leq \Big|\frac  1 {\tilde a_{V_{n, J}}(\zeta)}  \Big| \leq  C_{\theta},\quad \forall \, \zeta\in e^{i\theta}\R_+,
\end{equation}
for all $n$ sufficiently large.
 This bound together with \eqref{orthKl}  and  \eqref{newdec1} allows us to apply Proposition\refer{preturbationspectrumgen}, Lemma\refer{intrel}  and Corollary\refer{coruse} iii), in order to conclude that for all $n$ sufficiently large,
 $$I-J\geq n(V_{n,J}; \theta)\geq  \frac{1}{2i\pi }\int\limits^{+ \infty \, e^{i\theta}}_0 \left(\frac {\tilde a'_{V_{n, J} }(s)}  {\tilde  a_{V_{n, J}}(s)} -
\frac {\tilde a'_{\cR_n}(s)}  {\tilde  a_{\cR_n}(s)}\right)
ds+\frac1{4\pi}\Big({\|V_{n, J}\|^ 2_{L^2(\R)}}-\|\cR_n\|^ 2_{L^2(\R)}\Big)\, ,
$$ 
with
$$ \biggl|\int^{+ \infty \, e^{i\theta}}_0 \left(\frac {\tilde a'_{V_{n, J}}(s)}  {\tilde  a_{V_{n, J}}(s)} -
\frac {\tilde a'_{\cR_n}(s)}  {\tilde  a_{\cR_n}(s)}\right)ds\biggr| \stackrel{n \to +\infty}\longrightarrow 0.$$
This contradicts the assumption $I-J<I_0$ since 
 according to \eqref{newdec0}, \eqref{newdec1}, we have
$$\frac1{4\pi} \big({\|V_{n, J}\|^ 2_{L^2(\R)}}-\|\cR_n\|^ 2_{L^2(\R)}\big)\stackrel{n \to +\infty}\longrightarrow  I_0.$$ 
\end{proof}  

 \medbreak
 We next prove  the following  lemma. 
 \begin{lemma}
\label{lemargument}
{\sl  Assume\refeq{case1}. Then $\varphi^{I_1}=\pi$}\end{lemma}  
 \begin{proof}
We  argue by contradiction assuming that~$\varphi^{I_1} \in [\theta^*,  \pi[$.   
Consider  $\Delta_{U_n}(\lam)$ with~$\lam^2$ in a~$\im\zeta^{I_1}_n$  neighborhood of~$\zeta_n^{I_1}$.
In view of the assumptions\refeq{case1} and \eqref{def-I_1}, writing~$\lam^2=\re \zeta_n^{I_1}+p \im\zeta_n^{I_1}$, and taking into account the conservation of 
the Floquet discriminant together with the scaling property\refeq{symMKdelta} and the asymptotics \eqref{p1-1},  we get
\begin{equation}
\label{behaves}
 e^{i\mu_n\lam^2} \Delta_{U_n} (\lam)\Big|_{\lam^2=\re \zeta_n^{I_1}+p \im\zeta_n^{I_1}}= e^{- \frac i2m}+ o(1), \quad n\rightarrow \infty,
\end{equation}   
uniformly with respect to $p$ in compact subsets of $\C_{+}$. 

\smallskip  On the other hand
  $\Delta_{U_n}(\lam)$ can be expressed in terms of the Jost solutions  
$\psi_1^-(x, \lam, V_n)$ and $\psi_2^+(x, \lam, V_n) $ as follows: 
 \begin{equation}\label{l4.4-1}
  \Delta_{U_n} (\lam)= \tr[ (\psi^-_1 \psi^+_2) (\mu_n, \lam; V_n)(\psi^-_1 \psi^+_2)^{-1} (0, \lam; V_n)].
  \end{equation}
  Lemma\refer{jost} together with   the bounds  \eqref{estimatesext1} and \eqref{equi1}   ensures that,   for every $0<\delta<\frac\pi2$,   there exists $C_\delta>0$ so that, 
  for all $\lam\in \Omega_{+}$ with $\lam^2\in \Gamma_\delta$ and all $n$, we have
  \begin{equation}\label{l4.4-2}\begin{split}
  &\|e^{i\lam^2x}\psi_1^-(x, \lam; V_n)\|_{L^\infty(\R)}+ \|e^{- i\lam^2x}\psi_2^+(x, \lam; V_n)\|_{L^\infty(\R)}\leq C_\delta,\\
&|\psi_{1,1}^-(0,\lam; V_n)-1|+ |e^{-i \lam^2 \mu_n}\psi_{2,2}^+(\mu_n,\lam; V_n)-1|\leq \frac{C_\delta}{\mu_n},\\
  &|\psi_{1,2}^-(0,\lam; V_n)|+ e^{\im \lam^2 \mu_n}|\psi_{2,1}^+(\mu_n,\lam; V_n)|\leq \frac{C_\delta}{\sqrt{\mu_n}},
 \end{split} \end{equation}
  where $\psi_1^-=\begin{pmatrix}\psi_{1,1}^-\\\psi_{1,2}^-\end{pmatrix}$ and  $\psi_2^+=\begin{pmatrix}\psi_{2,1}^+\\\psi_{2,2}^+\end{pmatrix}$. 
  
  \smallskip Recalling \eqref{defa}, we deduce from \eqref{l4.4-1}-\eqref{l4.4-2} that
   \begin{equation*}\label{keydelta-a} e^{ i \lambda^2 \mu_n}  \Delta_{U_n}(\lam)=     a_{V_n}(\lam) + \frac 1 {a_{V_n}(\lam)}\Big({ e^{2 i \lambda^2 \mu_n}- 
   e^{i\lam^2\mu_n}\psi^-_{1, 2} (\mu_n, \lambda;V_n)\psi^+_{2,1} (0, \lambda;V_n)} + O\Big(\frac 1 { \sqrt{\mu_n}}\Big) \Big) \, , \end{equation*} 
   uniformly with respect to   $\lam^2\in \Gamma_\delta$. 
   
  \medskip Writing further $\psi^-_{1, 2} (\mu_n, \lambda;V_n)$, $\psi^+_{2,1} (0, \lambda;V_n)$ as
   \begin{equation*}\begin{split}\psi^-_{1,2}(\mu_n, \lam, V_n)=-\lam\int_{-\frac M 2}^{\mu_n-\frac{\mu_n}{4K_n}}&e^{i\lam^2(\mu_n-y)}\overline{V_n(y)}\psi^-_{1,1}(y, \lam; V_n)dy \\&-
   \lam\int^{\mu_n}_{\mu_n-\frac{\mu_n}{4K_n}}e^{i\lam^2(\mu_n-y)}\overline{V_n(y)}\psi^{-}_{1,1}(y, \lam; V_n)dy,
   \end{split}
   \end{equation*}
 $$\psi^+_{2,1}(0, \lam, V_n)=-\lam\int^{\mu_n+\frac{M}2}_{\frac{\mu_n}{4K_n}}e^{i\lam^2y}V_n(y)\psi^+_{2,2}(y, \lam; V_n)dy-\lam\int_{0}^{\frac{\mu_n}{4K_n}}e^{i\lam^2y}V_n(y)\psi^+_{2,2}(y, \lam; V_n)dy, $$
and taking into account the bounds  \eqref{estimatesext2} and \eqref{l4.4-2},  we get
$$e^{-\im\lam^2\mu_n}|\psi^-_{1, 2} (\mu_n, \lambda;V_n)|+|\psi^+_{2,1} (0, \lambda;V_n)|\lesssim_\delta e^{-  \im \lam^2 \frac {\mu_n}  {2 K_n}}  + \frac1{\sqrt{K_n}}, \quad  \lam^2 \in \Gamma_\delta,$$ 
 which implies that
    \begin{equation}\label{keydelta-a} e^{ i \lambda^2 \mu_n}  \Delta_{U_n}(\lam)=     a_{V_n}(\lam) + \frac 1 {a_{V_n}(\lam)}O\Big(  e^{-  \im \lam^2 \frac {\mu_n}  {K_n}}+\frac1{\sqrt{K_n}}\Big), \end{equation} 
   uniformly with respect to   $\lam^2\in \Gamma_\delta$.

\medskip  We next fix $\frac\pi2<\theta^{**}<\theta^*$ and denote by $\zeta_n^j$, $j=I+1, \dots, J_n$, the zeros of $\tilde a_{V_n}(\zeta)$ in the angle~$\{\zeta\in \C_+:   \,  \theta^{**}<\arg \zeta\leq \theta^*\}$, counted with their multiplicity. As for the $\zeta_n^j$'s with $j=1, \dots, I$, after passing to a subsequence, we can assume that
  $J_n=J\geq I$ for all~$n$ sufficiently large, and that, for all $j=I+1, \dots, J$,  there holds:
  \begin{enumerate}
 \item[(i)]  $ \arg \zeta_{n}^{j} \stackrel{n\to\infty}\rightarrow \varphi^{j} \in [\theta^{**}, \theta^*] ,$
 \item[(ii)]
either $| z^j_n|\stackrel{n\to\infty}\longrightarrow \infty$,
or $ z^j_n\stackrel{n\to\infty}\longrightarrow z^j\in \C$,
where $z_n^j= \frac {\zeta_n^j-\zeta_n^{I_1}}{\im \zeta_n^{I_1}}$.
\end{enumerate}
 With these notations,  taking advantage of Proposition \ref{stability0} and Lemma\refer{lem-lowb}, 
  we infer that   there exists $\theta^{**}<\alpha^{**}<\varphi^{I_1}<\alpha^{*}<\pi$ and a positive constant $C$  such  that, for all $n$ sufficiently large,  we have
   \begin{equation}   \label{eq: factor} \tilde a_{V_n}(\zeta)= \prod_{\varphi^j=\varphi^{I_1}}  \left( \frac {\zeta -\zeta_n^j} {\zeta -\overline\zeta_n^j}\right) A_{V_n}(\zeta)\, \,  \mbox{with}  \,
 \sup_{\alpha**\leq\arg \zeta\leq \alpha*}\left(\Big| {A_{V_n}(\zeta)}\Big|+\Big|\frac1 {A_{V_n}(\zeta)}\Big|\right)\leq C.
\end{equation} 
  Fixing $r_0>0$ sufficiently small,   we deduce from  \eqref{keydelta-a}  and \eqref{eq: factor} 
 that
 \begin{equation}\label{keydelta-a-1} \big |e^{ i \lambda^2 \mu_n}  \Delta_{U_n}(\lam)\big|\,\Big|_{\lam^2=\re \zeta_n^{I_1}+p \im\zeta_n^{I_1}}\lesssim |p-i|+
  \frac 1{ \prod\limits_{z^j=0}|p-i-z^j_n|}o(1), \quad n\rightarrow \infty,
\end{equation}   
uniformly with respect to $|p-i|\leq r_0$. 
Consequently, 
$$\lim\limits_{p\to i}\lim\limits_{n\to \infty}  e^{ i \lambda^2 \mu_n}  \Delta_{U_n}(\lam)\,\Big|_{\lam^2=\re \zeta_n^{I_1}+p \im\zeta_n^{I_1}}=0.$$
This   contradicts \eqref{behaves} and completes the proof of the lemma.   
\end{proof}
     
  \bigbreak

We  next re-order   the  zeros  $(\zeta_{n}^{j})_{1 \leq j \leq I}$ of~$\tilde  a_{V_n}$ in the angle~$\{ \zeta\in \C_+:   \, \theta^* < \arg \zeta< \pi \big\}$ according to the regimes \eqref{case1}-\eqref{case2}.
In the first regime \eqref{case1},  we re-number $(\zeta_{n}^{j})_{1 \leq j \leq I}$ so that $I_1$   defined by \eqref{def-I_1} verifies:
   \begin{enumerate}  
 \item[(a)] $\varphi^j=\pi$ for all $1 \leq j \leq I_1$;
 \item[(b)] $\varphi^j<\pi$ for all $I_1+1 \leq j \leq I$;
 \end{enumerate}
 and furthermore: 
 \begin{enumerate}
 \item[(c)] $z^j=\infty$ for all $1\leq j\leq I_3$;
\item[(d)] $z^j\in \C\setminus\{0\}$ for all $I_3+1\leq j\leq I_2$;
 \item[(e)] $z^j=0$  for all $I_2+1\leq  j\leq I_1$,
 \end{enumerate}
  for some $1\leq I_3<I_2<I_1$, where  $z^j$ are given by \eqref{zj-0}, \eqref{zj-0-1} with 
  $z^j _n= \frac {\zeta_n^j-\zeta_n^{I_1}}{\im \zeta_n^{I_1}}$, and we set  $z^j=\infty$ if $| z^j_n|\stackrel{n\to\infty}\longrightarrow \infty$.

  \medskip
 
 In the second regime \eqref{case2}, we keep $\zeta_n^{I}$, so that $\zeta^{I}<0$, and re-oder $(\zeta_{n}^{j})_{1 \leq j \leq I-1}$ in such a way that
     $z^j=\infty$
    for all $1\leq j \leq {\tilde I}_1$  and $z^j\in \C$
      for all ${\tilde I}_1 + 1 \leq j \leq I$, $1\leq \tilde I_1<I$, where, as in \eqref{zj-0}, \eqref{zj-0-1}, 
      $z^j=\lim\limits_{n\to \infty}z^j_n$, $z^j_n= \frac {\zeta_n^j-\zeta_n^{I}}{\im \zeta_n^{I}}$.

  \smallskip

\begin{remark}   
 {\sl Note that with this ordering, in the first regime we have:
\begin{equation}\label{eq:I1}|\zeta^{j}_n| \im \zeta^{j}_n \leq |\zeta^{I_1}_n| \im \zeta^{I_1}_n \stackrel{n\to\infty}\longrightarrow  0,  \quad
\forall j=1, \dots, I,
 \end{equation} 
  \begin{equation}\label{eq:angl}\frac {\im \zeta^j_{n}} {|\zeta^j_{n}|} \stackrel{n\to\infty}\longrightarrow 0, \quad \forall j=1, \dots, I_1,\end{equation}
    \begin{equation}\label{eq:norm3} \frac{|\zeta^{j}_n| } {|\zeta^{I_1}_n| } \lesssim \sqrt {\frac {\im \zeta^{I_1}_{n}}  {|\zeta^{I_1}_{n}|}}\stackrel{n \to +\infty}\longrightarrow 0, 
   \quad z^j=\infty, \quad
   \quad \forall j=I_1+1, \dots, I.\end{equation}}  
   \end{remark}
 \medbreak
 
   \subsection{Removing  the  zeros $\zeta_n^j$, $j=1, \dots, I$.}  
We   apply the  B\"acklund transformation to
  remove  successively all   the zeros~$(\zeta^{j}_n)_{1 \leq  j  \leq I}$   in the angle~$\{ \zeta\in \C_+:   \, \,  \theta^* < \arg \zeta< \pi \big\}$,     setting:
    \begin{equation} \begin{aligned} \label{Bk} V_{n, 0} &=  V_{n}\, , \\  V_{n, j}&= \cB_{\lam_{n}^{j}} (\psi_{n}^{j} ) V_{n, j-1}  \, , \, \, j=1,\cdots, I\, ,\end{aligned} \end{equation} 
where
 $\lam_n^j=\sqrt{\zeta_n^j}\in \C_{++}$, and $\psi_{n}^j\in H^1(\R, \C^2)\setminus\{0\}$ solves
 $L_{V_{n, j-1}}(\lam_{n}^{j})\psi_{n}^{j}=0$, $j=1, \dots, I$.
 
 \medskip  Note that in view of \eqref{defS} and\refeq{backtransf}, the functions $V_{n, j}$, $0\leq j \leq I$,  are $C^\infty$ functions supported in $[-\frac M2, \mu_n+\frac M 2]$, so that the associated Jost solutions
 are entire functions of $\lam$   satisfying
$$\begin{aligned}
\psi_1^-(x, \lambda;V_{n, j})&= e^{ -i \lambda^2 x}  \left(
\begin{array}{ccccccccc}
1  \\
0 
\end{array}
\right) , \quad \mbox{if} \quad x \leq -\frac M2, \\
\psi_2^+(x, \lambda; V_{n, j}) &=   e^{ i \lambda^2 x}  \left(
\begin{array}{ccccccccc}
0  \\
1 
\end{array}
\right) , \quad \mbox{if} \quad x \geq \mu_n+\frac M 2
\end{aligned}$$
and
$$\begin{aligned}
\psi_1^+(x, \lambda;V_{n, j}) &=  e^{- i \lambda^2 x} \left(
\begin{array}{ccccccccc}
1  \\
0
\end{array}
\right), \quad \mbox{if} \quad x \geq \mu_n+\frac M 2,\\
\psi_2^-(x, \lambda;V_{n, j})&= e^{ i \lambda^2 x}  \left(
\begin{array}{ccccccccc}
0  \\
1
\end{array}
\right), \quad \mbox{if} \quad x \leq -\frac M2.
\end{aligned}$$ 

Note also that thanks to  \eqref{estimatesext}, we have the following bound.
\begin{lemma}\label{firstbound}
{\sl  For all $j=1, \dots, I$, there holds\footnote{In fact one has $\|V_{n, j}\|^2_{H^{1}(]-\infty, M] \cup [\mu_n-M, +\infty[)} \lesssim \frac {1} {\mu_n}$, but \eqref{estimatesext1bis} will be enough for our purpose.} 
\begin{equation}
\label{estimatesext1bis}
\|V_{n,j}\|^2_{L^2(-\infty, M] \cup [\mu_n-M, +\infty[)} \lesssim \frac {1} {\mu_n}.
\end{equation}}
\end{lemma}
\begin{proof}
By virtue of \eqref{estimatesext}, one has \eqref{estimatesext1bis}  for $j=0$. Furthermore,
if \eqref{estimatesext1bis} holds for some~$0\leq j\leq I-1$, then   the corresponding Jost solutions $\psi_1^-(x, \lam;V_{n, j})$, $\psi_2^+(x, \lam;V_{n, j})$
satisfy
 \begin{equation}\label{jost-ends}
 \begin{aligned}
\Big|e^{i\lam^2x}\psi_1^-(x, \lam;V_{n, j})-\begin{pmatrix}1\\0\end{pmatrix}\Big |&\lesssim  \frac 1 {\sqrt{\mu_n}}, \quad \forall \, x\leq M, \\
\Big|e^{-i\lam^2x}\psi_2^+(x, \lam; V_{n, j})-\begin{pmatrix}0\\1\end{pmatrix}\Big |&\lesssim  \frac 1 {\sqrt{\mu_n}}, \quad \forall \, x\geq  \mu_n-M, 
\end{aligned}
\end{equation}
uniformly with respect to $\lam$ in bounded sets of $\C$. 
Since~$\psi^{j+1}_n$ is a non-zero $L^2$ solution of~$L_{V_{n, j}}(\lam^{j+1}_{n})\psi^{j+1}_n=0$ and therefore, $\psi^{j+1}_n=C_{n, j}^-\psi_1^-( \lam^{j+1}_{n}; V_{n, j})$,  $\psi^{j+1}_n=C_{n, j}^+\psi_2^+( \lam^{j+1}_{n}; V_{n, j})$ for some  non-zero constants~$C_{n, j}^-$, 
$C_{n, j}^+$, 
combining  \eqref{defS} and \eqref{jost-ends}
one gets that
$$\|\cS_{\lam^{j+1}_{n}} (\psi^{j+1}_{n})\|_{L^\infty(]-\infty, M] \cup [\mu_n-M, +\infty[)} \lesssim \frac {1} {\sqrt{\mu_n}},$$
which in view of \eqref{backtransf} and \eqref{estimatesext1bis}, implies that
$$
\|V_{n,j+1}\|^2_{L^2(]-\infty, M] \cup [\mu_n-M, +\infty[)} \lesssim \frac {1} {\mu_n}.
$$
\end{proof}
The next result is a direct consequence of 
 Lemma\refer{gen-propagation}   and the bound\refeq{estimatesext2}.
   \begin{lemma}
\label{propagation}
{\sl  Assume  \refeq{case1}. Then  for all $j=1,\cdots, I_1$, we have
\begin{eqnarray}  \label{behbig-0} 
 &\|V_{n, j}\|^2_{L^{2}((-\infty, \frac{\mu_n}{4K_n}]  \cup[\mu_n-\frac{\mu_n}{4K_n}, +\infty))} 
\lesssim  \frac 1 {K_n}  + \sum\limits^j_{\ell =1} (\pi - \arg \zeta_n^\ell)\, \stackrel{n\to\infty}\longrightarrow 0,\\
 \label{behbig} &  \|G_{\lam_n^j}(\psi_n^j)+1\|_{L^\infty((-\infty, \frac{\mu_n}{4K_n}]  \cup[\mu_n-\frac{\mu_n}{4K_n}, +\infty))}
\lesssim  \frac 1 {K_n}  + \sum\limits^j_{\ell =1} (\pi - \arg \zeta_n^\ell)\, \stackrel{n\to\infty}\longrightarrow 0.
\end{eqnarray} }\end{lemma} 

\medbreak

  \begin{remark}
\label{propagation-r}
{\sl   It follows from \eqref{behbig} that, for all $n$ sufficiently large and all $j=1,\cdots, I_1$,
 \begin{equation}\label{delta}
 \begin{aligned}
 &|\psi^{j}_{n, 1}(x)|^2-|\psi^{j}_{n, 2}(x)|^2>0, \quad \forall\,  x \leq \frac{\mu_n}{4K_n}, \\
 &|\psi^{j}_{n, 1}(x)|^2-|\psi^{j}_{n, 2}(x)|^2<0, \quad \forall \, x \geq  \mu_n-\frac{\mu_n}{4K_n}.
 \end{aligned}
\end{equation}
Furthermore, 
 for all $j=1,\cdots, I_1$ and all $x \in (-\infty, \frac{\mu_n}{4K_n}]  \cup[\mu_n-\frac{\mu_n}{4K_n}, +\infty)$,  
we have
  \begin{equation}
 \label{imptsuite}   \Big|\frac {\re \lam^{j}_{n} (|\psi^{j}_{n, 1}(x)|^2+|\psi^{j}_{n, 2}(x)|^2)}  {\lam^{j}_{n}(|\psi^{j}_{n, 1}(x)|^2-|\psi^{j}_{n, 2}(x)|^2)}
   \Big|\lesssim \frac 1 {K_n}+ \sum^j_{\ell =1} (\pi - \arg \zeta_n^\ell).
 \end{equation}
}\end{remark} 
  \bigbreak

Consider    the resulting potential 
\begin{equation}\label{def-rnKbis}
 \tilde{\rm r}_n= V_{n, I}.
 \end{equation}
It    follows from\refeq{masszero} that 
\begin{equation}\label{mass20bis}
\|\tilde {\rm r}_{n}\|_{L^2(\R)}^2= \|V_{n}\|_{L^2(\R)}^2 - 4 \sum^{I}_{i=1}  { \arg} (\zeta_{n}^{i}) \, .
  \end{equation}  
Moreover, taking advantage of Proposition \ref{keygen},  we get the following  estimate.
 \begin{lemma}
\label{lemdotrnH1}
{\sl  One has
 \begin{equation} 
  \label{estbakenbis3} \|\tilde  {\rm r}_n\|^2_{\dot H^{1}(\R)} \lesssim    \alpha(n)+  \frac {1}{\mu_n} \stackrel{n\to\infty}\longrightarrow  0\,. \end{equation}
}\end{lemma}    
  \begin{proof}
 By\refeq{relcons} and \eqref{energieszero}, we have
  \begin{equation} \label{bis3-0} 
P_\R(\tilde  {\rm
r}_n)= P_\R(V_n)+8 \sum^{I} _{j=1} \im \zeta^{j}_n,
\quad  E_\R (\tilde {\rm
r}_n)= E_\R(V_n) - 16  \sum^{I} _{j=1} \im \, \zeta^{j}_n \, \re  \zeta^{j}_n,
\end{equation} 
which together with
\eqref{equi2}, \eqref{equi3}, ensures  that
  \begin{equation} \label{estbakenbis2}
 \begin{aligned}
 |P_\R(\tilde {\rm
r}_n)|&\lesssim
 \max_{1 \leq  j  \leq I}\,  
  \im \zeta^{j}_n + \frac 1{\mu_n},\\
|E_\R(\tilde {\rm
r}_n)|
&\lesssim \alpha(n)+ \frac 1{\mu_n}.
\end{aligned}
\end{equation}
Combining these bounds with
 Proposition  \ref{keygen}, one gets \eqref{estbakenbis3}.   \end{proof}

  \bigbreak
 To study the behavior of the monodromy matrix $M_{U_n} (\lam)$ as $n\rightarrow\infty $, we first  apply  the connection formula \eqref{JB-1} to express
 $M_{U_n} (\lam)$ in  terms of 
    the Jost solutions $\psi_1^-(x, \lambda; \tilde {\rm
r}_n)$ and  $\psi_2^+(x, \lambda; \tilde {\rm
r}_n)$ corresponding to the potential $\tilde{\rm r}_n$. This gives:
\begin{equation} \label{relation} M_{U_n} (\lam)= \cA^{-1}_n(\mu_n, \lam) (\psi^- \psi^+) (\mu_n, \lam; \tilde{\rm
r}_n) (\psi^- \psi^+)^{-1} (0, \lam; \tilde{\rm
r}_n)\cA_n(0, \lam), \quad \forall\, \lam\in \C,\end{equation} where $\cA_n(x, \lam)=A_{n, {I}}(x, \lam)\cdots A_{n, 1}(x, \lam)$ with  $A_{n, j}(x, \lam) =A(\lam;\psi^{j}_n(x), \lam_{n}^{j})$.

 \medskip  
 The next lemma describes
 the asymptotics  of $A_{n,j}(0, \lam)$ and $A_{n,j}(\mu_n, \lam)$ for large~$n$. 
   \begin{lemma} \label{specifyA}
  {\sl Let $$D^j_n(\lam)=\begin{pmatrix}\frac {\overline\lam^j_{n}}{\lam^j_{n}}(\lam^2-\zeta^j_{n})&0\\ 0& \frac {\lam^j_{n}}{\overline\lam^j_{n}}(\overline{\zeta}^j_{n}-\lam^2)\end{pmatrix}.$$
Then 
$$\Big |A_{n, j}(0, \lam)-D^j_n(\lam)\Big|+\Big | A_{n, j}(\mu_n, \lam)-\overline{D^j_n(\overline\lam)}\Big|\lesssim  \frac{\im \zeta^j_{n}}{\mu_n |\zeta^j_{n}| } + \frac { \im \zeta^j_{n}}{\sqrt{\mu_n |\zeta^j_{n}|}}  ,$$
uniformly with respect to $\lam$ in bounded subsets of  $\C$. }
\end{lemma}    
 \begin{proof} 
 The proof follows immediately  from Lemma\refer{firstbound} which ensures   that,  for all $0 \leq j \leq I$,  the Jost solutions~$\psi_1^-(x, \lam;V_{n, j})$ and $\psi_2^+(x, \lam;V_{n, j})$ satisfy the bounds \eqref{jost-ends}, uniformly with respect to $\lam$ in bounded sets of $\C$, and therefore, one has
\begin{equation}\label{Gbeh}   \Big |G_{\lam^j_{n}} (\psi^j_n) (0)-\frac {\overline \lambda^j_{n}} {\lambda^j_{n}} \Big |\lesssim \frac{\im \zeta^j_{n}}{\mu_n |\zeta^j_{n}|}\, , \quad
  \Big |G_{\lam^j_{n}} (\psi^j_n)  (\mu_n)-\frac {\lambda^j_{n}} {\overline \lambda^j_{n}}\Big |\lesssim \frac {\im \zeta^j_{n}} {\mu_n |\zeta^j_{n}|}\, ,\end{equation}
and 
 \begin{equation}\label{Sbeh}  \Big |\cS_{\lam^j_{n}} (\psi^j_n)  (0)\Big|\lesssim \frac { \im \zeta^j_{n} }{\sqrt{\mu_n|\zeta^j_{n}|}},  \quad  \Big|\cS_{\lam^j_{n}} (\psi^j_n)  (\mu_n)\Big|\lesssim \frac { \im \zeta^j_{n}}{\sqrt{\mu_n |\zeta^j_{n}|}}. \end{equation}
 In view of \eqref{defMatrix}, 
this concludes  the proof of the lemma. 
\end{proof} 
\medbreak


\subsection{Approximation of  the Jost solutions associated to the potential $\tilde {\rm r}_n$} \label{apr} 
The aim of this section is to study the Jost solutions  $\psi_1^-(x, \lambda; \tilde {\rm
r}_n)$, $\psi_2^+(x, \lambda;  \tilde {\rm
r}_n)$ of  the Kaup-Newell   system 
\begin{equation} \label{KNrn} i \sigma_3 \partial_x \psi - \lam^2 \psi-i\lam \left(
\begin{array}{ccccccccc}
0 &\tilde{\rm
r}_n \\
\overline {\tilde{\rm
r}_n} &0
\end{array}
\right)\psi=0,
\end{equation}
 taking advantage of the smallness of 
$\partial_x\tilde {\rm r}_n$ given by   \eqref{estbakenbis3}. To this end, we 
 set
 \begin{equation}\label{eqhat-00}
 \psi(x)= B_n(x,\lam) \hat \psi(x)
 \end{equation}
 with
 \begin{equation}\label{eqhat-01}\begin{split}
&B_n(x, \lam)= \frac1{\sqrt{1+|\tilde {\rm r}_n(x)|^2/4\lam^2}} 
\left(
\begin{array}{ccccccccc}
1 & \frac {-i  \tilde {\rm r}_n(x) }{ 2\lam}\\
\frac {-i \overline{\tilde {\rm r}_n(x)}} { 2\lam} &1 
\end{array}
\right)e^{-i\sigma_3 \Phi_n(x,\lam)},\\
& \Phi_n(x, \lam)=\int_{-M}^x\frac{2\lam^2|\tilde {\rm r}_n(y)|^2+\im (\tilde {\rm r}_n(y)\overline{\partial_y\tilde{\rm r}_n(y))}}{4\lam^2+|\tilde {\rm r}_n(y)|^2}dy.
\end{split}
 \end{equation}
 Then the system \eqref{KNrn} takes the following form
\begin{equation}\label{eqhat1}
(i\s_3\partial_x-\lam^2-Q_n(\lam))\hat \psi=0,
\end{equation}
where 
\begin{equation}\label{eqhat2}
Q_n(\lam)=\begin{pmatrix}0&q_n(\lam)\\-\overline{q_n(\overline \lam)}&0\end{pmatrix}, \,
q_n(x, \lam)=-\frac{\lam e^{2i\Phi_n(x, \lam)}(2\partial_x\tilde{\rm r}_n(x)+ i |\tilde {\rm r}_n(x)|^2\tilde {\rm r}_n(x))}{4\lam^2+|\tilde {\rm r}_n(x)|^2}.
\end{equation}
 Writing 
\begin{equation}\label{B}
\psi_1^-(x, \lambda;  \tilde {\rm r}_n)=e^{- i\lam^2 x}B_n(x,\lam)\hat \eta^-_n(x,\lam), \, \psi_2^+(x, \lambda;  \tilde {\rm r}_n)=e^{-i\Phi_n(\mu_n+M, \lam)+ i\lam^2 x}B_n(x,\lam)\hat \eta^+_n(x,\lam),
\end{equation}
with $ \hat \eta^\mp_n(x,\lam)=\begin{pmatrix}\hat \eta^\mp_{n,1}(x,\lam)\\ \hat \eta^\mp_{n,2}(x,\lam)\end{pmatrix}$,
and 
invoking \eqref{mass20bis}, \eqref{estbakenbis3}, \eqref{eqhat1} and \eqref{eqhat2},  we obtain  the following bounds.

\medskip

\begin{lemma} \label{hats}
{\sl  There exists a positive constant~$C$  such that, for all~$n$ sufficiently large and~$\lam\in \Omega_+$ with
$\frac {\|\tilde {\rm r}_n\|_{L^{2}(\R)}\|\tilde {\rm r}_n\|_{\dot H^{1}(\R)}}{ |\lam|^2} \leq 1/2$,    the following estimates hold
\begin{eqnarray} \label{est-hat1}
\|{\hat \eta_{n,1}^- }(\lam)-1\|_{L^{\infty}(\R)}+\|{\hat \eta_{n,2}^+}( \lam)-1\|_{L^{\infty}(\R)}&\leq& Ce^{C\frac{(\alpha(n)+ \mu_n^{-1})}{|\lam|^2\im \lam^2}}\frac{(\alpha(n)+ \mu_n^{-1})}{|\lam|^2 \im \lam^2},\\ \label{est-hat2}
\|{\hat \eta_{n,2}^- }(\lam)\|_{L^{\infty}(\R)}+\|{\hat \eta_{n,1}^+}( \lam)\|_{L^{\infty}(\R)}&\leq& Ce^{C\frac{(\alpha(n)+ \mu_n^{-1})}{|\lam|^2\im \lam^2}}\sqrt{\frac{(\alpha(n)+ \mu_n^{-1})}{|\lam|^2\im \lam^2}}.
\end{eqnarray}
 Furthermore, setting $\lam_n=\lam_n^{I_1}$  in the first regime  \eqref{case1}, and $\lam_n=\lam_n^{I}$ in the second regime\refeq{case2}, we have
 \begin{equation} \begin{aligned}\label{est-hat-0}
 \|{\hat \eta_{n}^\pm }(\lam)-\hat\eta^{\pm,0}_n(\lam)\|_{L^{\infty}(\R)}&\leq Ce^{C\frac{\alpha(n)+ \mu_n^{-1}}{\im \lam^2} \big(\frac  1 {|\lam|^2}+\frac  1{(\im \lam_{n})^2}\big)}\sqrt{
 \frac{(\alpha(n)+ \mu_n^{-1})}{|\lam|^2\im \lam^2}}\\ & \qquad \qquad \qquad \qquad \times \Big(\frac {|\lam-i\im\lam_{n}|}{ \im \lam_{n}}+\sqrt{
 \frac{(\alpha(n)+ \mu_n^{-1})}{|\lam|^4}}\Big), \end{aligned}
 \end{equation}
 where $e^{-i\lam^2x}\hat\eta^{-,0}_n(x,\lam)$ and $e^{i\lam^2x}\hat\eta^{+,0}_n(x,\lam)$ are the Jost solutions of the system 
 \begin{equation}\label{eqhat1-0}\begin{split}
&(i\s_3\partial_x-\lam^2-Q_n^0)\psi=0,\,\,\,
Q_n^0=\begin{pmatrix}0&q_n^0\\\overline{q^0_n}&0\end{pmatrix},\\ &q_n^0(x)=\frac  i{4\im \lam_{n}}e^{i\int_{-M}^x|\tilde {\rm r}_n(y)|^2dy}\big(2\partial_x\tilde{\rm r}_n(x)+ i |\tilde {\rm r}_n(x)|^2\tilde {\rm r}_n(x)\big).
\end{split}\end{equation}}\end{lemma}     
  \begin{proof} 
The assumption $\frac {\|\tilde {\rm r}_n\|_{L^{2}(\R)}\|\tilde {\rm r}_n\|_{\dot H^{1}(\R)}}{ |\lam|^2} \leq 1/2$  together with \eqref{mass20bis} and \eqref{estbakenbis3},
yields
  \begin{equation}\label{est-qn}
  \|q_n(\lam)\|_{L^2(\R)}\lesssim \frac {\|\tilde {\rm r}_n\|_{\dot H^{1}(\R)}}{ |\lam|} \lesssim  \frac {1}{ |\lam|}  \big(\alpha_{n}+  {\mu_n}^{-1}\big)^{\frac12},
  \end{equation}
  and
   \begin{equation}\label{est-bet}\begin{split}
  \|q_n(\lam)-q_n^0\|_{L^2(\R)}+\|\overline{q_n(\overline\lam)}+\overline {q_n^0}\|_{L^2(\R)}&\lesssim 
|\lam-i\im \lam_{n}|\,\frac {\|\tilde  {\rm r}_n\|_{\dot H^{1}(\R)}}{|\lam| \im \lam_{n}}+
\frac {\|\tilde  {\rm r}_n\|^2_{\dot H^{1}(\R)}}{|\lam|^3}\\
 & \lesssim  |\lam-i\im\lam_{n}|  \frac {\big(\alpha(n)+  {\mu_n}^{-1}\big)^{\frac12}}{|\lam|  \im \lam_{n}}+\frac { (\alpha(n)+  {\mu_n}^{-1})}{|\lam|^3},
 \end{split}
\end{equation} 
which in view of 
   Lemma \ref{jost}, leads to the bounds  \eqref{est-hat1}-\eqref{est-hat-0}.
\end{proof}

 \medbreak
 
\begin{remark}
\label{est-hat-00}
{\sl Since the spectral problem \eqref{eqhat1-0} is self-adjoint, one has  the following low bound for
$\hat\eta^{\mp,0}_n=\begin{pmatrix}\hat\eta^{\mp,0}_{n,1}\\ \hat\eta^{\mp,0}_{n,2}\end{pmatrix}$:
\begin{equation}\label{low-bound-hat}
|\hat\eta^{-,0}_{n,1}(x,\lam)|^2- |\hat\eta^{-,0}_{n,2}(x,\lam)|^2\geq 1, \, \, 
 |\hat\eta^{+,0}_{n,2}(x,\lam)|^2- |\hat\eta^{+,0}_{n,1}(x,\lam)|^2\geq 1, \, \,  \forall \, (x, \lam)\in \R\times \overline{\Omega}_+.
\end{equation}
Indeed, $\hat\eta^{-,0}_n$ solves
\begin{equation}\label{eqhat0-0-}\begin{cases}
i\partial_x \eta^{-,0}_{n,1}- q_n^0 \eta^{-,0}_{n,2}=0\\
i\partial_x \eta^{-,0}_{n,2}+ 2\lam^2 \eta^{-,0}_{n,2} + \overline{q_n^0}  \eta^{-,0}_{n,1}=0\\
 \eta^{-,0}_{n,1}\big|_{x\leq-M}=1, \,\,\, \eta^{-,0}_{n,2}\big|_{x\leq-M}=0.\end{cases}
\end{equation}
As a consequence, 
$$\partial_x(|\eta^{-,0}_{n,1}|^2-|\eta^{-,0}_{n,2}|^2)=4\im \lam^2|\eta^{-,0}_{n,2}|^2\geq 0, \quad  \forall \, (x, \lam)\in \R\times \overline{\Omega}_+,$$
and $|\eta^{-,0}_{n,1}(x,\lam)|^2-|\eta^{-,0}_{n,2}(x,\lam)|^2=1$ for $x\leq -M$,  which gives the first inequality in \eqref{low-bound-hat}. 
The second one follows in a similar way.} 
\end{remark}

\medbreak


 \section{Arriving at  a contradiction in the first regime} \label{End}
 
  To complete the proof of Theorem\refer{Mainth}, we analyze separately   the regimes\refeq{case1} and\refeq{case2},
starting   with  the  regime\refeq{case1} where  $\mu_n \min_{K}  \alpha(K,n) \stackrel{n\to\infty}\longrightarrow \infty$. Recall that we fixed~$K=K_n$ so that 
\begin{equation}\label{assume}
  K_n \stackrel{n\to\infty}\longrightarrow \infty \andf   \frac {K_n}{\ds \mu_n \min_{K}  \alpha(K, n) }\stackrel{n\to\infty}\longrightarrow 0 \,.
 \end{equation}
 \subsection{Scheme of the contradiction argument}  As explained in Section \ref{genstrategy}, in the case  of \eqref{case1}  the arguments will rely on  the conservation of the Floquet discriminant. More precisely, as in the proof of Lemma \ref{lemargument},   writing $\lam^2=\re \zeta_n^{I_1}+p \im\zeta_n^{I_1}$, we shall show that
 \begin{equation}\label{key-r1}
 \lim\limits_{p\to i}\lim\limits_{n\to \infty}  e^{ i \lambda^2 \mu_n}  \Delta_{U_n}(\lam)\,\Big|_{\lam^2=\re \zeta_n^{I_1}+p \im\zeta_n^{I_1}}=0,
 \end{equation}
 which will give a contradiction in view of \eqref{behaves}.
 
\medskip  
 To establish \eqref{key-r1}, we  start by observing that   by Lemma \ref{hats},   
for all~$n$ sufficiently large and~$\lam^2= \re \zeta_n^{I_1}+p \im\zeta_n^{I_1}$, we have
\begin{equation}\label{est-hat-2}
\|{\hat \eta_n^\pm}(\lam)\|_{L^{\infty}(\R)} \lesssim 1 \,\,\andf \, \,
 \|{\hat \eta_{n}^\pm }(\lam)-\hat\eta^{\pm,0}_n(\lam)\|_{L^{\infty}(\R)}\lesssim  \sqrt {\frac {\im \zeta^{I_1}_{n}}{|\zeta^{I_1}_{n}|}} \, , 
\end{equation} 
uniformly with respect to $p$ in compact subsets of $ \C_{+}$.
 Combining these estimates with\refeq{relation}  and Lemma  \ref{specifyA}, we get   the following  result.\begin{proposition}
\label{key}
 {\sl  Let
 $$ \begin{aligned} \Sigma_n(\lam) & =e^{-\frac i  2 \|u_n\|^2_{L^2(\TT)} } \prod\limits_{j=1}^{I} \frac {(\lam^2 -\zeta^j_{n})}  {(\lam^2 -\overline{\zeta}^j_{n})} \, \hat \eta_{n,1}^{-}(\mu_n, \lam) \\ & \qquad \qquad \qquad \qquad -   e^{\frac i  2 \|u_n\|^2_{L^2(\TT)} } \prod\limits_{j=1}^{I} \frac {(\lam^2 -\overline{\zeta}^j_{n})}  {(\lam^2 -{\zeta}^j_{n})} \frac {\hat \eta_{n,2}^{-}(\mu_n, \lam)\, \, \hat \eta_{n,1}^{+}(0, \lam)} {\hat \eta_{n,2}^{+}(0, \lam)}\, .
\end{aligned} $$ 
 Then, 
 for  all $n$ sufficiently large and $\lam \in \C_{++}$ with  $\lam^2= \re \zeta_n^{I_1} + p \im \zeta_n^{I_1}$,  there holds
 \begin{equation}\begin{aligned}\label{prop-key-00}
\Big |&e^{i\mu_n\lam^2}\Delta_{U_n}(\lam)-\Sigma_n(\lam)\Big |  \lesssim \biggl(\frac {\im \zeta^{I_1}_{n}}{|\zeta^{I_1}_{n}|}\biggr)^{\frac 1 4}+ \frac  1 {\sqrt {\mu_n \im \zeta^{I_1}_{n}}},
\end{aligned}\end{equation} 
uniformly with respect to $p$ in compact subsets of $\C_+\setminus\{i+z^j, j=I_3+1,\dots, I_1\}$.}
 \end{proposition}
  \begin{proof}
  Recall that, in view of\refeq{relation}, 
  $$M_{U_n} (\lam)= \cA^{-1}_n(\mu_n, \lam) (\psi^- \psi^+) (\mu_n, \lam; \tilde{\rm
r}_n) (\psi^- \psi^+)^{-1} (0, \lam; \tilde{\rm
r}_n)\cA_n(0, \lam)\, ,$$ 
where  $\cA_n(x, \lam)=A_{n, {I}}(x, \lam)\cdots A_{n, 1}(x, \lam)$ with $A_{n, j}(x, \lam) =A(\lam;\psi^{j}_n, \lam_{n}^{j})$,~$\psi^{j}_n$ being a non-zero $H^1$-solution of $L_{V_{n, j-1}}(\lam^j_{n})\psi=0$.
 By virtue of\refeq{B}, this implies  that  $$
  \Delta_{U_n} (\lam)= \tr[ B_n^{-1}(0, \lam)\cA_n(0, \lam)\cA^{-1}_n(\mu_n, \lam)B_n(\mu_n, \lam) (\hat \eta^-_n \hat\eta^+_n) (\mu_n, \lam)e^{-i\mu_n\lam^2\sigma_3}(\hat \eta^-_n \hat\eta^+_n)^{-1} (0, \lam)].$$
  Observe that for $\lam^2= \re \zeta_n^{I_1} + p \im \zeta_n^{I_1}$ with $p$ in a compact subset $K$ of $  \C_{+}\setminus\{i+z^j, j=I_3+1, \dots, I_1\}$ and provided $n$ is sufficiently large (depending on $K$),  we have for all~$j=1, \dots, I$, 
  $$1\lesssim_K\Big |\frac {\lam^2-\zeta_n^j}{\lam^2-\bar \zeta_n^j}\Big|\leq 1, $$
  and
  $$\frac{\im \zeta^j_{n}}{\mu_n |\zeta^j_{n}| } + \frac { \im \zeta^j_{n}}{\sqrt{\mu_n |\zeta^j_{n}|}} \lesssim_K \frac{|\lam^2-\zeta_n^j|}{\sqrt{\mu_n\im \zeta_n^{I_1}}},$$
  which by
  Lemma  \ref{specifyA}, 
   leads to the bound
 \begin{equation}\label{prop-key-0}
 \Big|\cA_n(0, \lam)\cA^{-1}_n(\mu_n, \lam)-\Lambda_n(\lam^2)\Big|\lesssim_K \frac1{\sqrt{\mu_n\im\zeta_n^{I_1}}},
 \end{equation}
where
$$
\Lambda_n(\zeta)=
  \begin{pmatrix} \prod\limits_{j=1}^{I}\frac {\overline\zeta^j_{n}}{\zeta^j_{n}}\frac {(\zeta- {\zeta}^j_{n})}{(\zeta - \overline \zeta^j_{ n})}&0\\ 0&\prod\limits_{j=1}^{I} \frac {\zeta^j_{n}}{\overline\zeta^j_{n}}\frac {(\zeta-\overline{\zeta}^j_{n})}{(\zeta -  \zeta^j_{ n})}\end{pmatrix}.$$ 
 
 We next consider $B_n(0, \lam)$    and   $B_n(\mu_n, \lam)$.  In view of  \eqref{case1} and \eqref{estbakenbis3}, for $\lam^2= \re \zeta_n^{I_1} + p \im \zeta_n^{I_1}$ 
  and $n$ sufficiently large, we have
 \begin{equation}\label{prop-key-2}
  \frac {\|\tilde  {\rm r}_n\|_{L^{\infty}(\R)} }{|\lam|}\lesssim \left(\frac {\im \zeta^{I_1}_{n}}  {|\zeta^{I_1}_{n}|}\right)^{\frac14}, 
  \end{equation}
  \begin{equation}\label{prop-key-3}
  |\Phi_n(x, \lam)-\frac12\int_{-\infty} ^x|\tilde  {\rm r}_n(y)|^2dy|\lesssim \left(\frac {\im \zeta^{I_1}_{n}}  {|\zeta^{I_1}_{n}|}\right)^{\frac12},  \quad \forall \, x\in \R,
  \end{equation}
  which gives
  \begin{equation}\label{prop-key-4}
  |B_n(0, \lam)- {\rm Id}|+ |B_n(\mu_n, \lam)-e^{-\frac i 2 \|\tilde {\rm r}_n\|^2_{L^{2}(\R)}\sigma_3 }|\lesssim \left(\frac {\im \zeta^{I_1}_{n}}  {|\zeta^{I_1}_{n}|}\right)^{\frac14},
  \end{equation}
  uniformly with respect to $p$ in compact subsets of $\C_+$.
  Finally, the bound\refeq{est-qn} ensures that
$$\Big| \hat \eta^{-}_n(0, \lam)-\begin{pmatrix}1\\0\end{pmatrix}\Big|+ \Big| \hat \eta^{+}_n(\mu_n, \lam)-\begin{pmatrix}0\\1\end{pmatrix}\Big|\lesssim   (\im \zeta^{I_1}_{n})^{\frac12},$$
 which together with    \eqref{est-hat-2} and \eqref{low-bound-hat}  implies that, as  $n\rightarrow \infty$, 
\begin{equation}\label{prop-key-1}\begin{split}
(\hat \eta^-_n \hat\eta^+_n) (\mu_n, \lam)\begin{pmatrix}1&0\\0&e^{2i\mu_n\lam^2}\end{pmatrix}
(\hat \eta^-_n \hat\eta^+_n)^{-1} (0, \lam)
&=\\ \begin{pmatrix}
\hat \eta_{n,1}^{-}(\mu_n, \lam) &  - \frac {\hat \eta_{n,1}^{+}(0, \lam)} {\hat \eta_{n,2}^{+}(0, \lam)}\hat \eta_{n,1}^{-}(\mu_n, \lam)\\
\hat \eta_{n,2}^{-}(\mu_n, \lam) &  - \frac {\hat \eta_{n,1}^{+}(0, \lam)} {\hat \eta_{n,2}^{+}(0, \lam)}\hat \eta_{n,2}^{-}(\mu_n, \lam)
\end{pmatrix}&+  O\Big(  e^{-2\mu_n \im \lam^2}+ ( \im \zeta^{I_1}_{n})^{\frac 1 2}\Big),
\end{split}\end{equation}
again uniformly with respect to  $p$ in compact subsets of $ \C_{+}$.
Combining \eqref{prop-key-0}-\eqref{prop-key-1}, one gets \eqref{prop-key-00}.
\end{proof}

\medbreak 
 \begin{remark}\label{rk}
{\sl Fixing $0<r_1<1$ sufficiently small\footnote{Explicitly, one needs $r_1<\min\limits_{I_3+1\leq j\leq I_2}|z^j |$.} and setting 
 \begin{eqnarray}\label{fi-}
   \varphi_n ^-(p) &= &\hat \eta_{n,2}^{-}(\mu_n, \sqrt{\re\zeta^{I_1}_{n}+p\im \zeta^{I_1}_{n}}),\\  \label{fi+}\varphi_n ^+(p)  &= &\hat \eta_{n,1}^{+}(0, \sqrt{\re\zeta^{I_1}_{n}+p\im \zeta^{I_1}_{n}}), \end{eqnarray}
one deduces from the bounds \eqref{low-bound-hat}, \eqref{est-hat-2} and Proposition \ref{key} that, for all $p$ with $0<|p-i|\leq r_1$,
\begin{equation} \label{prep-1}
\big |e^{ i \lambda^2 \mu_n}  \Delta_{U_n}(\lam)\big|\,\Big|_{\lam^2=\re \zeta_n^{I_1}+p \im\zeta_n^{I_1}}\lesssim |p-i|+   \frac {|\varphi_n ^-(p)\varphi_n^+(p)|} { \prod\limits_{j=I_2+1}^{I_1}|p-i-z^j_n|} +o(1), \quad n\rightarrow \infty.
 \end{equation}
(Recall  that $z^j _n\stackrel{n\to\infty}\rightarrow 0$, for all $I_2+1 \leq j \leq I_1$.)
The analyticity of $\varphi_n^\pm$ together with the bound \eqref{est-hat-2} ensures that for $n$ sufficiently large
\begin{equation} \label{prep-2}
|\partial_p^k\varphi_n^\pm(p)|\lesssim_k 1, \quad \forall\, k\in \N,
\end{equation}
uniformly with respect to $p$ in compact subsets of  $ \C_+$.
Therefore,  to obtain \eqref{key-r1} it suffices to show 
that the following proposition  holds. }
   \begin{proposition}
\label{vfkey} {\sl  
For all $0 \leq k \leq I_1-I_2-1$,  
\begin{equation}
\label{contfin} \partial_p^k\varphi^\pm_n(i) \stackrel{n\to\infty}\rightarrow 0. \end{equation}  }\end{proposition} 
\end{remark}
  \medbreak 
  The remainder of this section is devoted to the proof of Proposition \ref{vfkey}.  Since $\varphi_n^-$ and~$\varphi_n^+$ can be treated similarly,  we  give the arguments only for~$\varphi_n^-$.


\subsection{Notations} For $n$ sufficiently large and $j=1, \dots, I$, we set
   \begin{equation}  \label{inversecompk} {\check \psi}_{n}^{j} =\begin{pmatrix}  \frac {\overline \psi^j_{n, 2}} {d_{\lam^j_n} (\psi^j_n)}\\ \frac {\overline \psi^j_{n, 1}} {d_{\overline \lam^j_n} (\psi^j_n)}\end{pmatrix} , \end{equation}
where as before $\psi_n^j \in H^1\setminus\{0\}$ solves $L_{V_{n, j-1}}(\lam_n^j)\psi_n^j=0$. Accordingly to Section \ref{defBacklund transformation}, $\check \psi_n^j$ is a solution of 
$L_{V_{n,j}}(\lam_n^j)\psi=0$, and the B\"acklund transformation $\cB_{\lam_{n}^{j}} (\check\psi_{n}^{j} )$ is a left inverse of $\cB_{\lam_{n}^{j}} (\psi_{n}^{j} )$ so that
 \begin{equation}\label{inversek}V_{n, j-1}= \cB_{\lam_{n}^{j}} ({\check \psi_{n}^{j} }) V_{n, j}, \quad j=1, \dots, I.
 \end{equation}
Clearly, 
 $$\check \psi_{n}^{j}(x)= \begin{cases}\ds c_n^- (j)e^{i\lam^2 x}{0\choose 1}\quad {\rm if}\,\,\, x\leq-\frac M2\\\ds c_n^+ (j)e^{-i\lam^2 x}{1\choose 0}\quad {\rm if}\,\,\, x\geq\mu_n+\frac M2,
\end{cases}$$
for  some non-zero constants $c_n^- (j)$ and $ c_n^+ (j)$, which  ensures  that 
\begin{equation}\label{5.2.0}
\check \psi_{n}^{j}(x)=c_n^-(j)\psi^-_2(x, \lam_n^j; V_{n, j})=c_n^+(j)\psi^+_1(x, \lam_n^j; V_{n, j}).
\end{equation}
Note also that in view of \eqref{delta}, the functions $\check \psi_n^j={\check\psi_{n,1}^j\choose \check\psi_{n,2}^j}$, $j=1, \dots, I_1$, satisfy
 \begin{eqnarray}
 |\check\psi^{j}_{n, 1}(x)|^2-|\check\psi^{j}_{n, 2}(x)|^2 &<0&, \quad \forall\,  x \leq \frac{\mu_n}{4K_n}, \label{delta-}\\
|\check\psi^{j}_{n, 1}(x)|^2-|\check\psi^{j}_{n, 2}(x)|^2 &>0&, \quad \forall \, x \geq  \mu_n-\frac{\mu_n}{4K_n}\label{delta+}.
 \end{eqnarray}
This property  will be at  the heart of  the proof of Proposition \ref{vfkey}. In order to link it to the functions $\varphi_n^\pm$,  we need to express the Jost solutions $\psi_1^+(x, \lam, V_{n,j})$,
$\psi_2^-(x, \lam, V_{n,j})$ in terms of~$\hat \eta_n^+(x, \lam)$, $\hat \eta_n^-(x, \lam)$. To   deal with the case of $\varphi_n^-$,  it will be enough to consider~$\psi_1^+(x, \lam, V_{n,j})$.
For\footnote{Recall that $V_{n,I}=\tilde {\rm r}_n$.} $\psi_1^+(x, \lam, \tilde {\rm r}_n)$ one  has
\begin{equation}\label{vfkey-0}\begin{aligned}
 \psi_1^+(x, \lam; \tilde {\rm r}_n)&=\frac{e^{-i\lam^2(\mu_n+M)}}{\psi^-_{1,1}(\mu_n+M, \lam; \tilde {\rm r}_n)}\psi^-_1(x, \lam;\tilde {\rm r}_n)\\& \qquad \qquad -e^{-2i\lam^2(\mu_n+M)}\,
 \frac{\psi^-_{1,2}(\mu_n+M, \lam;\tilde {\rm r}_n)}{\psi^-_{1,1}(\mu_n+M, \lam; \tilde {\rm r}_n)}
 \psi^+_2(x, \lam; \tilde {\rm r}_n),
\end{aligned} \end{equation}
 which together
with \eqref{B}, leads to the representation
\begin{equation}\label{5.2.1}
\psi_1^+(x, \lam; \tilde {\rm r}_n)=\gamma_n(x, \lam)\Psi_n^I(x,\lam)
\end{equation}
with 
\begin{equation}\label{5.2.2}\begin{split}
&\gamma_n(x,\lam)= \frac {e^ {i\Phi_n(\mu_n+M, \lam)+i\lam^2(x-2(\mu_n+M))}  }{\hat \eta_{n, 1}^- (\mu_n+M, \lam)},\\
 &\Psi_n^I(x, \lam)=
 B_n(x,\lam) \Big(e^{2i\lambda^2(\mu_n+M-x)} \hat \eta^-_n (x,\lam) - \hat \eta^-_{n, 2} (\mu_n+M,\lam)\hat \eta^+_n (x,\lam) \Big).
 \end{split}
 \end{equation}
 Using the relation 
 $$ 
 \psi^+_1(x, \lambda; V_{n, j-1}) =  \frac {\lam^j_n} {\overline \lam^j_n}  \frac 1 {\lam^2 -  \zeta^j_n}  A(\lam; \check \psi_{n}^j(x), \lam^j_n)\psi^+_1(x, \lambda; V_{n, j}), $$
 we deduce from  \eqref{5.2.1}  that
 $$\psi_1^+(x, \lam; V_{n,j})=\gamma_n(x, \lam)\Psi_n^j(x,\lam),  \quad j=0, \dots, I-1,$$
 with
 \begin{equation}\label{5.2.30}
 \Psi_n^j(x, \lam)= \left(\prod\limits_{k=j+1}^{I}\frac{\lam_n^k}{\bar\lam_n^k(\lam^2-\zeta_n^k)}\right)A(\lam; \check \psi_n^{j+1}(x),\lam_n^{j+1})\cdots A(\lam; \check \psi_n^I(x),\lam_n^I)
 \Psi_n^I(x, \lam).
  \end{equation} 
 Observe that by  \eqref{SLinfty} and \eqref{eq:norm3}, we have
 \begin{equation}\label{5.2.3}
 \Big|\frac{\lam  \cS_{\lam^j_n} (\check \psi_{n}^{j})} {\zeta -  \zeta^j_n}\Big|\lesssim  \sqrt {\frac{|\zeta_n^j|} {|\zeta_n^{I_1}|}} \lesssim \Big(\frac {\im \zeta^{I_1}_{n}}  {|\zeta^{I_1}_{n}|}\Big)^ {\frac 1 4},\quad
 j=I_1+1, \dots, I.
 \end{equation}
Furthermore, using that
 \begin{equation*}\begin{split}
 &\frac{\lam^2G_z(\eta)-|z|^2}{\lam^2-z^2}=G_z(\eta)+\frac{z(z^2-\bar z^2)|\eta_2|^2}{(\lam^2-z^2)d_z(\eta)},\\
 &\frac{\lam^2\overline{G_z(\eta)}-|z|^2}{\lam^2-z^2}=\overline{G_z(\eta)}+\frac{z(z^2-\bar z^2)|\eta_1|^2}{(\lam^2-z^2)d_{\bar z}(\eta)}, \quad \eta={\eta_1\choose \eta_2}\in \C^{2}\setminus\{0\},
 \end{split}
 \end{equation*}
  and invoking \eqref{eq:norm3}, we obtain 
 \begin{equation}\label{5.2.4}\begin{split}
 &\Big|\frac{\lam^2G_{\lam_n^j}(\check \psi_n^j)-|\lam_n^j|^2}{\lam^2-\zeta_n^j}-G_{\lam_n^j}(\check\psi_n^j)\Big|\lesssim\Big(\frac {\im \zeta^{I_1}_{n}}  {|\zeta^{I_1}_{n}|}\Big)^ {\frac12},\\
 &\Big|\frac{\lam^2\overline{G_{\lam_n^j}(\check \psi_n^j)}-|\lam_n^j|^2}{\lam^2-\zeta_n^j}-\overline{G_{\lam_n^j}(\check\psi_n^j)}\Big|\lesssim\Big(\frac {\im \zeta^{I_1}_{n}}  {|\zeta^{I_1}_{n}|}\Big)^ {\frac12}, \quad j=I_1+1, \dots, I,
 \end{split}
 \end{equation}
 which together with \eqref{defMatrix}  and  \eqref{5.2.3}, implies that,  for all $n$ sufficiently large and $j=I_1+1,\dots, I$, 
 \begin{equation}  \label{eq:est5}  \left|\frac1{\lam^2-\zeta_n^j}A(\lam; \hat \psi_n^j,\lam_n^j)
 -\begin{pmatrix}G_{\lam_n^j}(\check\psi_n^j)&0\\0&-\overline{G_{\lam_n^j}(\check\psi_n^j)}\end{pmatrix}\right|\lesssim   \Big(\frac {\im \zeta^{I_1}_{n}}  {|\zeta^{I_1}_{n}|}\Big)^ {\frac 1 4}.
 \end{equation}
 In \eqref{5.2.3}, \eqref{5.2.4}, \eqref{eq:est5} and all along this section, it is assumed that $\lam^2=\re \zeta_n^{I_1}+p \im \zeta_n^{I_1}$, and all the bounds we give are uniform with respect 
  to $p$ in compact subsets of $\C_+$.
  

\subsection{Step 0} We start by the case $k=0$. To show that $\varphi_n^-(i)\stackrel{n\to\infty}\rightarrow 0$, we consider~$\Psi_n^{I_1}={\Psi_{n,1}^{I_1}\choose \Psi_{n,1}^{I_1}}$. It follows from   \eqref{est-hat-2}, \eqref{prop-key-2}, \eqref{prop-key-3} and \eqref{5.2.2}, \eqref{5.2.30}, \eqref{eq:est5}  that  
\begin{equation}\label{5.3.0}
 \begin{split} \biggl|\Psi_n^{I_1}(x,\lam)&+e^{-\frac i2\sigma_3\int_0^x|\tilde{\rm r}_n(y)|^2dy+2i\lam^2M} \hat \eta^{-}_{n,2}(\mu_n,\lam) \prod\limits_{k=j+1}^{I}\frac{\lam_n^k}{\bar\lam_n^k}\begin{pmatrix}G_{\lam_n^k}(\check\psi_n^k)&0\\0&-\overline{G_{\lam_n^k}(\check\psi_n^k)}\end{pmatrix}\, \hat\eta^{+,0}_n (x,\lam)\biggr|\\ &   \qquad   \qquad \qquad  \qquad \lesssim e^{-2\im \lam^2(\mu_n-x)}+\Big(\frac {\im \zeta^{I_1}_{n}}  {|\zeta^{I_1}_{n}|}\Big)^ {\frac 1 4}, \quad x\in \R.  \end{split} \end{equation}
Here we also used that, according to \eqref{est-qn}, one has 
$$|\hat\eta_{n,2}^-(\mu_n+M, \lam)-e^{2i\lam^2M}\hat\eta_{n,2}^-(\mu_n, \lam)|\lesssim \|q_n (\bar\lam)\|_{L^2(\R)}\lesssim (\im \zeta_n^{I_1})^\frac12.
$$
Invoking   \eqref{low-bound-hat}, \eqref{5.2.0},  \eqref{delta+}   and taking $\lam=\lam_n^{I_1}$ and $x=\mu_n^*$ with $\mu_n^*=\mu_n-\frac{\mu_n}{8K_n}$, we deduce from \eqref{5.3.0} that
$$|\varphi_n^-(i)|\lesssim e^{-\frac{\im \zeta_n^{I_1}\mu_n}{ 4K_n} }+\Big(\frac {\im \zeta^{I_1}_{n}}  {|\zeta^{I_1}_{n}|}\Big)^ {\frac 1 4}\stackrel{n\to\infty}\longrightarrow 0.$$
This concludes the proof of  Proposition \ref{vfkey} in the case   $I_2+1=I_1$. 


\subsection{Inductive step} Assume that $I_2+1<I_1$. Motivated by the step 0, we consider the functions $\Psi_n^j(\mu_n^*, \lam)$, $j=I_2+1, \dots , I_1$. It will be convenient for us to renormalize them introducing 
\begin{equation} \label{5.4.1}
Z_n^j(p)=\Big(\prod\limits_{k=j+1}^{I_1} \frac {\bar\lam_n^k\im \zeta^{I_1}_n}  {2i\lam\im \zeta^{k}_n} \Big)\Psi_n^j(\mu_n^*, \lam)
\left|_{\lam^2=\re \zeta_n^{I_1} +p\im\zeta_n^{I_1}}\right., \quad j=I_2+1, \dots I_1.
\end{equation}
Observe that because of the analyticity of $\Psi_n^{I_1}$, $\hat\eta_n$, $\hat \eta_n^{+,0}$ with respect to $\lam$, the bounds~\eqref{est-hat-2} and  \eqref{5.3.0} imply that 
as $n\to \infty$, 
 we have, for all $l\in \N$, 
\begin{equation}
\label{5.3.0-b}
|\partial_p^lZ_n^{I_1}(p)|\lesssim_l \sum\limits_{l'=0}^{l}|\partial_p^{l'}\varphi_n^-(p)|+o(1),
\end{equation}
\begin{equation}
\label{5.3.0-bb}
\Big|\partial_p^lZ_n^{I_1}(p)+ (\partial_p^l\varphi_n^-)(p)\Xi_n\hat\eta^{+,0}_n (\mu_n^*,\lam)\Big|\lesssim_l
\sum\limits_{l'=0}^{l-1}|\partial_p^{l'}\varphi_n^-(p)|+
o(1), \end{equation}
uniformly with respect to $p$ in compact subsets of $\C_+$, with
$$\Xi_n=
e^{-\frac i2\sigma_3\int_0^{\mu_n^*}|\tilde{\rm r}_n(y)|^2dy+2i\re \zeta_n^{I_1}M} \prod\limits_{k=I_1+1}^{I}\frac{\lam_n^k}{\bar\lam_n^k}\begin{pmatrix}G_{\lam_n^k}\big(\check\psi_n^k(\mu_n^*)\big)&0\\0&-\overline{G_{\lam_n^k}\big(\check\psi_n^k(\mu_n^*)\big)}\end{pmatrix}.$$

\medskip

Accordingly to \eqref{5.2.0}, \eqref{delta+}  and \eqref{5.2.30}, $ Z_n^j ={Z_{n,1}^j\choose Z_{n,2}^j}$, $j=I_2+1, \dots, I_1,$ satisfy  
\begin{equation}\label{delta+bis}
|Z_{n,1}^j(i+z_n^j)|^2-|Z_{n,2}^j(i+z_n^j)|^2>0,
\end{equation}
and, for $I_2+1\leq j\leq I_1-1$,   can be written as
\begin{equation} \label{newkey}
Z_n^j(p)=\frac {\Upsilon_{n}^j(p)}{\prod\limits_{k=j+1}^{I_1}(p-i-z_n^k)}
\end{equation}
where
\begin{eqnarray*} \Upsilon_{n}^j(p) &=&{\bf A}_n^{j+1}(p) \cdots {\bf A}_n^{I_1}(p)Z_n^{I_1}(p), \\
 {\bf A}_n^j(p)&= &\frac{\lam_n^j}{2i\lam \im \zeta_n^j}A(\lam; \check \psi_n^j(\mu_n^*), \lam_n^j)\Big|_{\lam^2= \re\zeta_n^{I_1} +p\im \zeta_n^{I_1}}.\end{eqnarray*}
Observe that in view of \eqref{defMatrix}, we have
\begin{equation}\label{defMatrix1}
A(\lam; \eta,z)=\frac\lam z A(z;\eta,z)+\lam(\lam-z)\begin{pmatrix}G_z(\eta)&0\\0&-\overline{G_z(\eta)}\end{pmatrix}-\bar z(z-\lam)\sigma_3,
\end{equation}
  with $A(z;\eta,z)$ given by \eqref{AB-0}.
Introducing the vectors $a_n^j\in \C^2$ so that  $a_n^j$ is collinear to~$\check \psi_n^j(\mu_n^*)$ and has the form\footnote{The existence of $a_n^j$ for all $n$ sufficiently large and $j\leq I_1$ is guaranteed by \eqref{5.2.0} and \eqref{delta+}.}
\begin{equation} 
  \label{colinear}a^{j}_n=\begin{pmatrix}1\\ \xi_n^j\end{pmatrix} \with |\xi_n^j|  < 1,
   \end{equation} 
   and using \eqref{defMatrix1}, we get
$${\bf A}_n^j(p)=\BT_n^j+ (p-i-z_n^j)\BD_n^j(p),$$
where 
\begin{equation} 
  \label{Tnot} 
  \BT_n^j= \BT(\lam_n^j, a_n^j) \with  \BT(z, \eta)=z\begin{pmatrix} \frac {\overline \eta_2}{d_z (\eta)} &0 \\ 0& \frac {\overline \eta_1}{ {d_{\bar z} (\eta)}}\end{pmatrix} \begin{pmatrix} \eta_2&- \eta_1\\ \eta_2&- \eta_1\end{pmatrix} , \end{equation} 
\begin{equation} 
  \label{Dnot} 
\BD_n^j(p)= \BD_n^{j,0}+ \BD_n^{j,1}(p), \quad \BD_n^{j,1}(p)=\frac{\lam_n^j-\lam}{\lam_n^j+\lam}\BD_n^{j,0}+ \frac{\lam_n^j(\bar\lam_n^j+\lam)\im\zeta_n^{I_1}}{2i\lam(\lam_n^j+\lam)\im\zeta_n^j}\sigma_3,
\end{equation}
with
\begin{equation}
  \label{Dnot1} 
\BD_n^{j,0}= \frac{\im\zeta_n^{I_1}}{4i\im\zeta_n^j}(1-\bar g_n^j)\begin{pmatrix}g_n^j&0\\ 0&1\end{pmatrix}, \quad g_n^j=G_{\lam_n^j}(a_n^j).
\end{equation}
Clearly, for all $n$ sufficiently  large  and $j=I_2+1, \dots, I_1$,
\begin{equation}\label{D}\begin{split}
|\BD_n^{j,0}|&\lesssim 1 \\
|\partial_p^m\BD_n^{j,1}(p)|&\lesssim_m \frac{\im \zeta_n^{I_1}}{|\zeta_n^{I_1}|}, \quad \forall \, m\in \N.
\end{split}\end{equation}
Observe also that $\BT_n^j$ can be written under  the form:
 \begin{equation}\label{T-1}
\BT_n^j=b_n^j\langle\sigma_3 \sigma_1a^j_n, \cdot\rangle_{\R^2} \with   b_n^j= \frac{\lam_n^j}{d_{\bar \lam_n^j}(a_n^j)}\begin{pmatrix} 
g_n^j\overline\xi_n^j\\1
\end{pmatrix}. \end{equation} 
In view of \eqref{imptsuite}, as $n\to \infty$, one has, for all $j=I_2+1,\dots, I_1$,
\begin{equation}\label{imptsuite-1}
\frac{\lam_n^j}{d_{\bar \lam_n^j}(a_n^j)}=
- \frac 1{\delta_n^j}(1+o(1)) \, \, \, \,  \mbox{with} \, \,  \, \, \delta_n^j= 1-|\xi_n^j|^2>0, \, \, \frac {\re \lam_n^j}{\im \lam_n^j \delta_n^j}\stackrel{n\to\infty}\rightarrow 0.
\end{equation}

\medskip 
Since  for $n$ sufficiently large, the functions $Z^j_n$ and $\BA_n^j$ are holomorphic functions  of $p$ on the ball $|p-i|<1$, the identity  \eqref{newkey} ensures
that
\begin{equation} 
  \label{Taylorkey}  Z^j_n  (p)=   \int^1_0 \cdots \int^1_0 dt_1 \cdots dt_L  t^{L-1}_1 \cdots t_{L-1}  \partial^{L}_{p'} \Upsilon_n^j(p') ,
    \end{equation} 
with $L=I_1-j$, and 
$$p'=(z_n^{I_1}+i)(1-t_1)+(z_n^{I_1-1}+i) t_1 (1-t_2)+\cdots + (z_n^{j+1}+i) t_1 \cdots t_{L-1} (1-t_L) +p t_1 \cdots t_L.$$   
 
  \medskip
  To estimate the derivatives of $ \Upsilon_n^j(p)$,  we start by observing that setting
  $$x_n^j(p)= \langle \sigma_3\sigma_1 a_n^j, \Upsilon_n^j(p)\rangle_{\R^2}, \quad I_2+1 \leq  j\leq I_1,$$
  with $\Upsilon_n^{I_1}(p)= Z_n^{I_1}(p)$, we can write $ \Upsilon_n^j(p)$ in the form
  \begin{equation}\label{Xi}
   \Upsilon_n^j(p)=\sum\limits_{k=j+1} ^{I_1}x_n^k(p)\left(\prod\limits_{m=j+1}^{k-1}(p-i-z_n^m){\bf D}_n^m(p)\right)b_n^{k} + \left(\prod\limits_{m=j+1}^{I_1}(p-i-z_n^m){\bf D}_n^m(p)\right)Z_n^{I_1}(p).
   \end{equation}
Introducing 
$$\pi_n(\ell_1, \ell_2)=\prod\limits_{j=\ell_1}^{\ell_2-1}\langle  \sigma_3\sigma_1 a_n^{j}, b_n^{j+1}\rangle_{\R^2}, \quad I_2+1\leq \ell_1\leq \ell_2\leq I_1,$$
and
$$x_n^{j,0}(p)=\pi_n(j,I_1)\langle  \sigma_3\sigma_1 a_n^{I_1}, Z_n^{I_1}(p)\rangle_{\R^2}, \quad j=I_2+1, \dots, I_1,$$
we claim that the following bounds hold.
 \begin{lemma}\label{1st}
{\sl    For all $j=I_2+1, \dots, I_1-1$, one has, for all $l\in \N$, 
    \begin{equation}\label{b-1st}\begin{aligned}
    |\partial_p^lx_n^j(p)&-\partial_p^lx_n^{j,0}(p)|\lesssim_l |\pi_n(j, I_1)| \\&\biggl(\sum\limits_{l'=0}^{l-1} |\partial_p^{l'}Z_n^{I_1}(p)|+ (|p-i|+ \max\limits_{I_2+1\leq j'\leq I_1}|z_n^{j'}|)|\partial_p^lZ_n^{I_1}(p)|\biggr).
 \end{aligned}\end{equation}}

     \end{lemma}

 \begin{remark}
\label{part-T1} {\sl Below, we will use systematically the following properties of 
 the products~$\pi_n(\ell_1, \ell_2)$ :
  \begin{equation}
  \label{relI} \pi_n(\ell_1, \ell_2) \pi_n(\ell_2, \ell_3)= \pi_n(\ell_1, \ell_3), \quad \forall\, I_2+1\leq\ell_1 \leq \ell_2 \leq \ell_3\leq I_1, \end{equation}
  \begin{equation} 
  \label{estI}\big|  \pi_n(\ell_1, \ell_2) \big|  \gtrsim \prod\limits_{j=\ell_1}^{\ell_2-1} \frac {\delta^{j}_{n}+\delta^{j+1}_{n}} {\delta^{j+1}_{n}} \gtrsim 1,\quad \forall\, I_2+1\leq \ell_1<\ell_2\leq I_1,\end{equation} 
 \begin{equation} 
  \label{relIbis}\frac {|\pi_n(\ell_1, I_1)|} {\delta^{\ell_1}_{n}} \gtrsim \frac {|\pi_n(\ell_2, I_1)} {\delta^{\ell_2}_{n}},\quad  \forall\, I_2+1\leq \ell_1<\ell_2\leq I_1.
   \end{equation}  

  The first property is obvious and  the second one is  a direct consequence of  \eqref{imptsuite-1}.
  Indeed, by \eqref{T-1}  and  \eqref{imptsuite-1}, 
    we have
$$\Big|\langle  \sigma_3\sigma_1 a_n^{j},b_n^{j+1}\rangle_{\R^2}\Big|\geq \left|\frac{\lam_n^{j+1}}{d_{\lam_n^{j+1}}(a_n^{j+1})}\right|(1-|\xi_n^j\xi_n^{j+1}|)\gtrsim 
\frac1{\delta_n^{j+1}}\left(\frac{1-|\xi_n^j|^2}{1+|\xi_n^j|}+\frac{1-|\xi_n^{j+1}|^2}{1+|\xi_n^{j+1}|}\right)\gtrsim 
\frac{\delta_n^j+\delta_n^{j+1}}{\delta_n^{j+1}}.$$
Finally, to establish  \eqref{relIbis} it is enough to observe that accordingly to   \eqref{relI} and \eqref{estI},
$$\frac {|\pi_n(\ell_1, I_1)|} {\delta^{\ell_1}_{n}}=\frac 1{\delta_n^{\ell_1}}|\pi_n(\ell_1, \ell_2)||\pi_n(\ell_2, I_1)|\gtrsim\frac1{\delta_n^{\ell_1}}\Big(\prod\limits_{j=\ell_1}^{\ell_2-1}\frac{\delta_n^j}{\delta_n^{j+1}}\Big)|\pi_n(\ell_2, I_1)|=\frac {|\pi_n(\ell_2, I_1)|} {\delta^{\ell_2}_{n}}.$$}
\end{remark}

\begin{proof}[Proof of Lemma \ref{1st}] We have $x_n^{I_1}(p)=x_n^{I_1,0}(p)$ and by \eqref{Xi},
$$x_n^{I_1-1}(p)=x_n^{I_1-1, 0}(p)+(p-i)\langle  \sigma_3\sigma_1 a_n^{I_1-1}, {\bf D}_n^{I_1}(p)Z_n^{I_1}(p)\rangle_{\R^2}, $$
which together with \eqref{Dnot}, \eqref{D} and \eqref{estI} leads to \eqref{b-1st} for $j=I_1-1$.  Next assume that \eqref{b-1st} holds for $I_2+1< j_0\leq I_1-1$ and consider $x_n^{j_0-1}(p)$.
From \eqref{Xi},  we get
\begin{equation*}\begin{split}
x_n^{j_0-1}(p)=&\sum\limits_{k=j_0} ^{I_1}x_n^k(p)\Big\langle\sigma_3\sigma_1 a_n^{j_0-1}, \Big(\prod\limits_{m=j_0}^{k-1}(p-i-z_n^m){\bf D}_n^m(p)\Big)b_n^{k}\Big\rangle_{\R^2}\\
&+ 
\Big\langle\sigma_3\sigma_1 a_n^{j_0-1},\Big(\prod\limits_{m=j_0}^{I_1}(p-i-z_n^m){\bf D}_n^m(p)\Big)Z_n^{I_1}(p)\Big\rangle_{\R^2},\end{split} \end{equation*}
which allows us to decompose $x_n^{j_0-1}(p)-x_n^{j_0-1, 0}(p)$ as follows:
\begin{equation}\label{proof-1st-1}
x_n^{j_0-1}(p)-x_n^{j_0-1, 0}(p)= (I)+ (II) +(III) +(IV)
\end{equation} with
 \begin{eqnarray*}
(I)&=&\langle \sigma_3\sigma_1 a_n^{j_0-1}, b_n^{j_0}\rangle_{\R^2}( x_n^{j_0}(p)-x_n^{j_0, 0}(p)),\\
(II)&=&\sum\limits_{k=j_0+1} ^{I_1}x_n^k(p)\Big\langle\sigma_3\sigma_1 a_n^{j_0-1}, \Big(\prod\limits_{m=j_0}^{k-1}(p-i-z_n^m){\bf D}_{n}^{m, 0}\Big)b_n^{k}\Big \rangle_{\R^2},\\
(III)&=&\sum\limits_{k=j_0+1} ^{I_1}x_n^k(p)\Big\langle\sigma_3\sigma_1 a_n^{j_0-1}, \Big(\prod\limits_{m=j_0}^{k-1}(p-i-z_n^m)\Big)\Big (\prod\limits_{m=j_0}^{k-1}{\bf D}_{n}^{m}(p)-
 \prod\limits_{m=j_0}^{k-1}{\bf D}_{n}^{m,0}
\Big)b_n^{k}\Big\rangle_{\R^2},\\
(IV)&=&\Big\langle\sigma_3\sigma_1 a_n^{j_0-1},\Big(\prod\limits_{m=j_0}^{I_1}(p-i-z_n^m){\bf D}_n^m(p)\Big)Z_n^{I_1}(p)\Big\rangle_{\R^2}.\end{eqnarray*}
In view of the properties \eqref{relI} and \eqref{estI}, the required bound for the terms $(I)$, $(IV)$ follows directly from the induction hypothesis and the estimate   \eqref{D} respectively.
Furthermore, since by the induction hypothesis, we have, for all $j=j_0, \dots, I_1$, 
\begin{equation}\label{proof-1st-20}
|\partial_p^lx_n^j(p)|\lesssim_l|\pi_n(j, I_1)|\sum\limits_{l'=0}^l|\partial_p^{l'}Z_n^{I_1}(p)|, \quad \forall \, l\in \N,
\end{equation}
the bound \eqref{D}  together with  \eqref{relI} and \eqref{estI}, ensures that  
$$ |\partial_p^l(III))|\lesssim_l |\pi_n(j_0, I_1)| \left(\sum\limits_{l'=0}^{l-1} |\partial_p^{l'}Z_n^{I_1}(p)|+ (|p-i|+  |z_n^{j_0}| )|\partial_p^lZ_n^{I_1}(p)|\right), \quad \forall \,l\in \N.$$
To complete the proof of Lemma \ref{1st}, it remains to consider the term $(II)$.  It results straightforwardly from the  definition of $a_n^j, \BD_n^{j, 0}, b_n^j$ (see \eqref{colinear}, \eqref{Dnot1}, \eqref{T-1}) 
and the bound \eqref{imptsuite-1} that 
\begin{equation}\label{proof-1st-2}
\biggl|\Big\langle\sigma_3\sigma_1 a_n^{j_0-1}, \Big(\prod\limits_{m=j_0}^{k-1}{\bf D}_{n}^{m, 0}\Big)b_n^{k}\Big \rangle_{\R^2}\biggr| \lesssim  \frac1{\delta_n^{k}}\biggl|1-\xi_n^{j_0-1}\overline\xi_n^k\prod\limits_{m=j_0}^{k}g_n^m\biggr|, \quad k=j_0+1, \dots,I_1.
\end{equation}
The expression $1-\xi_n^{j_0-1}\overline\xi_n^k\prod\limits_{m=j_0}^{k-1}g_n^m$ can be decomposed as
follows
\begin{equation}\label{proof-1st-3}\begin{split}
1-\xi_n^{j_0-1}\overline\xi_n^k\prod\limits_{m=j_0}^{k}g_n^m=&1-\prod\limits_{m=j_0}^{k}\xi_n^{m-1}\overline{\xi}_n^{m}g_n^m+
\xi_n^{j_0-1}\overline\xi_n^k\left(\prod\limits_{m=j_0}^{k}g_n^m\right)\left(\prod\limits_{m'=j_0}^{k-1}|\xi_n^{m'}|^2-1\right)\\
=&\sum\limits_{m=j_0}^{k}\left( \prod\limits_{\ell=j_0}^{m-1}\xi_n^{\ell-1}\overline{\xi}_n^{\ell}g_n^{\ell}\right)(1-\xi_n^{m-1}\overline{\xi}_n^{m}g_n^{m})\\
&+\xi_n^{j_0-1}\overline\xi_n^k\left(\prod\limits_{m=j_0}^{k}g_n^m\right)\sum\limits_{m=j_0}^{k-1}\left(\prod\limits_{\ell=j_0}^{m-1}|\xi_n^\ell|^2\right)(|\xi_n^m|^2-1).
\end{split}
\end{equation}
Recalling that $|g_n^j|=1$, $|\xi_n^j|<1$ and taking into account \eqref{imptsuite-1}, \eqref{relI} and  \eqref{estI},  we deduce from \eqref{proof-1st-2}, \eqref{proof-1st-3} that
\begin{equation*}\begin{split}
\biggl|\Big\langle\sigma_3\sigma_1 a_n^{j_0-1}, \Big(\prod\limits_{m=j_0}^{k}{\bf D}_{n}^{m, 0}\Big)b_n^{k}\Big\rangle_{\R^2}\biggr|\lesssim\frac1{\delta_n^{k}}
\sum\limits_{m=j_0}^{k}\delta_n^m\big|\langle \sigma_3\sigma_1 a_n^{m-1}, b_n^m\rangle_{\R^2}\big|\\
\lesssim 
\sum\limits_{m=j_0}^{k} |\pi_n(m,k) | \big|\langle \sigma_3\sigma_1 a_n^{m-1}, b_n^m\rangle_{\R^2}\big|\lesssim \sum\limits_{m=j_0}^{k}\big|\pi_n(m-1,k)\big|\lesssim \big|\pi_n(j_0-1, k)\big|,\end{split}
\end{equation*}
which in view of \eqref{proof-1st-20} implies
$$ |\partial_p^l(II))|\lesssim_l |\pi_n(j_0-1, I_1)| \left(\sum\limits_{l'=0}^{l-1} |\partial_p^{l'}Z_n^{I_1}(p)|+ (|p-i|+ |z_n^{j_0}|)|\partial_p^lZ_n^{I_1}(p)|\right), \quad \forall \,l\in \N,$$
thereby concluding the proof of Lemma \ref{1st}.
\end{proof}

\medbreak

We are now in position to finish the proof of Proposition \ref{vfkey}. We proceed by induction. Assume that for some $0\leq j\leq I_1-I_2-2$ we have
\begin{equation}\label{vkey-proof-1}
 \partial_p^k\varphi^-_n(i) \stackrel{n\to\infty}\rightarrow 0, \quad k=0, \dots, j.\end{equation}
 To establish  that $\partial_p^{j+1}\varphi^{-}_n(i) \stackrel{n\to\infty}\rightarrow 0$, we consider $Z_n^{I_1-j-1}(i+z_n^{I_1-j-1})$.  It follows from\refeq{Taylorkey}  that it can be written in the form:
  \begin{equation} 
  \label{Taylorkey-1}  
  Z_n^{I_1-j-1}(i+z_n^{I_1-j-1})
  =   \int^1_0 \cdots \int^1_0 dt_1 \cdots dt_{j+1} t^j_1 t^{j-1}_2 \cdots t_{j}  \partial^{j+1}_{p} \Upsilon_n^{I_1-j-1}(p) ,
    \end{equation} 
with 
\begin{equation}
  \label{Taylorkey-10}  
p=i+ z_n^{I_1} (1-t_1)+z_n^{I_1-1}t_1 (1-t_2)+\cdots + z_n^{I_1-j}t_1 \cdots t_{j} (1-t_{j+1}) +z_n^{I_1-j-1} t_1 \cdots t_{j+1}.
\end{equation}
Recalling \eqref{Xi}, we decompose   $\Upsilon_n^{I_1-j-1}$ as follows: 
\begin{equation} \label{ups-0}
\Upsilon_n^{I_1-j-1}(p)=\sum\limits_{k=I_1-j}^{I_1} x_n^k(p)\left(\prod\limits_{m=I_1-j}^{k-1}(p-i-z_n^m)\right) \tilde b_n^k +\Upsilon_{n,1}^{I_1-j-1}(p)+\Upsilon_{n,2}^{I_1-j-1}(p),
\end{equation}
where
 \begin{equation}\label{b-tilde}
 \tilde { b}_n^k= \Big(\prod\limits_{m=I_1-j}^{k-1}{\bf D}_{n}^{m, 0}\Big)b_n^{k}, \quad k=I_1-j, \dots, I_1,
 \end{equation}
\begin{eqnarray*} \Upsilon_{n,1}^{I_1-j-1}(p)&=&\sum\limits_{ k=I_1-j+1}^{I_1}x_n^k(p)\Big(\prod\limits_{m=I_1-j}^{k-1}(p-i-z_n^m) \Big)\Big (\prod\limits_{m=I_1-j}^{k-1}{\bf D}_{n}^{m}(p)-
 \prod\limits_{m=I_1-j}^{k-1}{\bf D}_{n}^{m,0}
\Big)
 b_n^{k}\\ \Upsilon_{n,2}^{I_1-j-1}(p)&=&\left(\prod\limits_{m=I_1-j}^{I_1}(p-i-z_n^m){\bf D}_n^m(p)\right)Z_n^{I_1}(p).\end{eqnarray*}
Accordingly to    \eqref{prep-2}, \eqref{5.3.0-b}, \eqref{D}, \eqref{imptsuite-1} and \eqref{b-1st}, we have 
\begin{equation} \label{ups-1}
| \partial^{j+1}_{p} \Upsilon_{n,1}^{I_1-j-1}(p)|\lesssim \frac{\im \zeta_n^{I_1}}{|\zeta_n^{I_1}|}\sum\limits_{k=I_1-j+1}^{I_1}\frac{|\pi_n(k, I_1)|}{\delta_n^k}=\pi_n(I_1-j+1, I_1)o(1), \,\,\, {\rm as} \,\,\,n\to \infty. \end{equation}
To estimate $ \partial^{j+1}_{p} \Upsilon_{n,2}^{I_1-j-1}(p)$, we use \eqref{5.3.0-b} and \eqref{D}
 to get
\begin{equation}
\label{ups-2}
| \partial^{j+1}_{p} \Upsilon_{n,2}^{I_1-j-1}(p)|\lesssim { |p-i|} |\partial^{j+1}_{p}\varphi_n^-(p)|+\sum\limits_{l=0}^j |\partial^{l}_{p}\varphi_n^-(p)| +o(1), \,\,\,{\rm as}\,\,\, n\to \infty,
\end{equation}
uniformly with respect to $p$ in compact subsets of $\C_+$. 

\medskip
Combining \eqref{Taylorkey-1} with \eqref{ups-0}, \eqref{ups-1}, \eqref{ups-2} and taking into account \eqref{prep-2},\refeq{estI} and\refeq{vkey-proof-1},
we obtain   
\begin{equation}
 \label{decom-Z0}
 Z_n^{I_1-j-1}(i+z_n^{I_1-j-1})= \sum\limits_{k=I_1-j}^{I_1} \kappa_n^k \tilde b_n^k + {\rm Rem}_n
 \end{equation}
 where 
 $$\kappa_n^{k}= \int^1_0 \cdots \int^1_0 dt_1 \cdots dt_{j+1} t^j_1 t^{j-1}_2 \cdots t_{j}  \partial^{j+1}_{p}\Big(x_n^k(p)\prod\limits_{m=I_1-j}^{k-1}(p-i-z_n^m)\Big),$$
 with  $p$ given by   \eqref{Taylorkey-10},
 and where the remainder term ${\rm Rem}_n$  satisfies  
 \begin{equation}
 \label{decom-Z1}
 {\rm Rem}_n= \pi_n(I_1-j+1, I_1)o(1), \,\,\, {\rm as}\, \,\,n\to \infty.
 \end{equation}
 Furthermore, using \eqref{prep-2},  \eqref{5.3.0-b}, \eqref{5.3.0-bb}, \eqref{b-1st} and \eqref{vkey-proof-1},   we get,  as $n\to \infty$,
\begin{equation}
 \label{decom-Z2}
\begin{split}
\kappa_n^k=\pi_n(k, I_1)o(1),\quad k=I_1-j+1, \dots, I_1,\\
\kappa_n^{I_1-j}=\pi_n(I_1-j, I_1)\big(\hat \kappa_n\partial^{j+1}_{p}\varphi_n^-(i)+o(1)\big),
 \end{split}
 \end{equation}
with
$$\hat \kappa_n=  -\frac1{(j+1)!}\langle\sigma_3\sigma_1a_n^{I_1}, \Xi_n \hat\eta^{+,0}_n (\mu_n^*,\lam_n^{I_1})\rangle_{\R^2}. $$
Note that in view of \eqref{low-bound-hat} and \eqref{est-hat-2}, we have 
\begin{equation}\label{bound-kappa}
1\lesssim|\hat \kappa_n|\lesssim1.
\end{equation}

\bigbreak 

To finish the proof of Proposition \ref{vfkey},   we 
first  recall that, according to  \eqref{delta+bis}, one has 
\begin{equation} \label{eqZprop}    \big\langle Z^{I_1-j-1}_n (i+z_n^{I_1-j-1}), \sigma_3  Z^{I_1-j-1}_n (i+z_n^{I_1-j-1})   \big\rangle_{\C^2} >0.\end{equation} 
On the other hand, using  \eqref{imptsuite-1}, \eqref{decom-Z0}, \eqref{decom-Z1}, \eqref{decom-Z2} and Remark \ref{part-T1}, we  obtain
\begin{equation}\label{5-end0}
\begin{split}
\big\langle Z^{I_1-j-1}_n (i+z_n^{I_1-j-1}), \sigma_3  Z^{I_1-j-1}_n (i+z_n^{I_1-j-1})   \big\rangle_{\C^2} =\sum\limits_{k=I_1-j}^{I_1} |\kappa_n^k|^2
\big\langle \tilde b_n^k, \sigma_3 \tilde b_n^k  \big\rangle_{\C^2}\\+ \sum\limits_{{k, k'=I_1-j\atop k\neq k'}}^{I_1}\kappa_n^k\overline\kappa_n^{k'}\big\langle \tilde b_n^k, \sigma_3 \tilde b_n^{k'} \big\rangle_{\C^2}
+\frac{\pi_n^2(I_1-j, I_1)}{\delta_n^{I_1-j}}o(1), \qquad {\rm as }\,\, n\to \infty.
\end{split}
\end{equation}
It follows directly from    \eqref{Dnot1}, \eqref{T-1},  \eqref{b-tilde} that, for all $I_1-j\leq k'<k\leq I_1$, 
$$
\Big|\big\langle \tilde b_n^k, \sigma_3 \tilde b_n^{k'} \big\rangle_{\C^2}\Big|\lesssim\frac{|1-\xi_n^{k'}\overline\xi_n^k\prod\limits_{m=k'+1}^kg_n^m|}{\delta_n^k\delta_n^{k'}},$$
and therefore proceeding as in the proof of Lemma \ref{1st}, we get:
$$
\biggl| \sum\limits_{{k, k'=I_1-j\atop k'\neq k}}^{I_1}\kappa_n^k\overline\kappa_n^{k'}\big\langle \tilde b_n^k, \sigma_3 \tilde b_n^{k'} \big\rangle_{\C^2}\biggr|
\lesssim  \sum\limits_{{k, k'=I_1-j\atop k'< k}}^{I_1}|\kappa_n^k||\kappa_n^{k'}|\frac {|\pi_n(k', k)|}{\delta_n^{k'}}.
$$
This bound together with \eqref{decom-Z2}, \eqref{bound-kappa} and Remark \ref{part-T1}, ensures that
\begin{equation}\label{5-end1}
 \sum\limits_{{k, k'=I_1-j\atop k'\neq k}}^{I_1}\kappa_n^k\overline\kappa_n^{k'}\big\langle \tilde b_n^k, \sigma_3 \tilde b_n^{k'} \big\rangle_{\C^2}=\frac{\pi_n^2(I_1-j, I_1)}{\delta_n^{I_1-j}}o(1),
  \quad {\rm as }\,\, n\to \infty.
\end{equation}
 Finally, since for all $j=I_2+1, \dots, I_1$, 
$$g_n^j=-1+ o(1), \quad n\to \infty,$$
we have
$$\big\langle \tilde b_n^k, \sigma_3 \tilde b_n^k  \big\rangle_{\C^2}=4^{-k+I_1-j}(1+o(1))\big\langle  b_n^k, \sigma_3  b_n^k  \big\rangle_{\C^2}=-\frac1{4^{k-I_1+j}\delta_n^k}(1+o(1)), \quad n\to \infty,$$
which implies that, as $n\to\infty$,
\begin{equation}\label{5-end2}\begin{split}
\sum\limits_{k=I_1-j}^{I_1}  |\kappa_n^k|^2  
\big\langle \tilde b_n^k, \sigma_3 \tilde b_n^k  \big\rangle_{\C^2}\leq -\frac1{\delta_n^{I_1-j}}|\pi_n(I_1-j, I_1)|^2|\hat \kappa_n|^2(|\partial_p^{j+1}\varphi_n^-(i)|^2+o(1)).\end{split}\end{equation}
Gathering \eqref{5-end0}, \eqref{5-end1} and \eqref{5-end2}, we end up with the bound:
 \begin{equation}\label{5-end}\begin{split}
 \big\langle Z^{I_1-j-1}_n (i&+z_n^{I_1-j-1}), \sigma_3  Z^{I_1-j-1}_n (i+z_n^{I_1-j-1})\rangle_{\C^2} \\&\leq -\frac1{\delta_n^{I_1-j}}|\pi_n(I_1-j, I_1)|^2|\hat \kappa_n|^2(|\partial_p^{j+1}\varphi_n^-(i)|^2+o(1)), \quad n\to\infty,
 \end{split}
 \end{equation}
 which in view of \eqref{eqZprop}, implies that
 $$\partial_p^{j+1}\varphi_n^-(i)\stackrel{n\to\infty}\rightarrow 0.$$
 This concludes the proof of Proposition \ref{vfkey} and hence the proof of    Theorem \ref{Mainth} in the first regime.
\medskip


 \section{Arriving at  a contradiction in the second regime} \label{End2}
We now consider the second regime where   $$\ds \mu_n \alpha(n) \stackrel{n\to\infty}\longrightarrow \beta \in \R_+\, ,$$ 
with  $\alpha(n)=\max_{1 \leq  j  \leq I}\,  |\zeta^{j}_n| \im \zeta^{j}_n$. 
 In view of\refeq{estbakenbis3}, we have  
 \begin{equation}\label{case2-r}
 \|\tilde  {\rm
r}_n\|^2_{\dot H^{1}(\R)}\lesssim \frac 1 {\mu_n}.
\end{equation}
In this regime the analysis follows closely the arguments of
 \cite{BaPe2}.
Recall that  the zeros~$\zeta_n^j$, $j=1, \dots, I$,  are ordered so that~$\zeta_n^{I}\stackrel{n\to\infty}\longrightarrow\zeta^{I}\neq 0$, and  consequently,
the sequence~$(\mu_n \im \zeta^{I}_n)_{n \in \N}$ is bounded. Therefore, there exists $\tilde \beta \in \R_+$ such that, 
up to passing to a subsequence, 
\begin{equation}\label{limnew} \mu_n \im \zeta^{I}_n \stackrel{n\to\infty}\longrightarrow \tilde \beta \, .\end{equation}
Furthermore, writing
\begin{equation}\label{dec} \mu_n  \re \zeta^{I}_n = 2\pi  N_n+2\pi \rho_n\,, \end{equation} with $N_n=\big [\frac {\mu_n  \re \zeta^{I}_n}{2\pi}\big] \stackrel{n \to  + \infty} \longrightarrow - \infty$,
$\rho_n=\big \{\frac {\mu_n  \re \zeta^{I}_{n}}{2\pi}\big\}\in [0,1[$, again after eventually extracting a subsequence,  one can assume that
\begin{equation}\label{dec1}
\rho_n
 \stackrel{n \to  + \infty} \longrightarrow \rho^* \in [0, 1].
 \end{equation}

\medskip  We first exclude the case of 
$\tilde \beta = 0$.  To this end,  we consider  the Jost solution~$\psi^+_1(x, \lam_n^{I}; \tilde{\rm r}_n)$.
Recall that it satisfies
\begin{equation}\label{initcond}\psi^+_1(x, \lam_n^{I}; \tilde{\rm r}_n)= \begin{cases}e^{-i\zeta^{I}_{n} x}\begin{pmatrix}1\\0\end{pmatrix} \,\,\,{\rm if } \,\,\, x\geq \mu_n+M, \\   c_ne^{i\zeta^{I}_{n} x}\begin{pmatrix}0\\1\end{pmatrix} \,\,\,{\rm if } \,\,\, x\leq -M,\end{cases}\end{equation}
 for some constant $c_n$.
 Invoking\refeq{eqhat-00}, we  write $\psi^+_1(x, \lam_n^{I}; \tilde{\rm r}_n)$ in the form:
 $$ \psi^+_1(x, \lam_n^{I}; \tilde{\rm r}_n)=e^{i\Phi_n(\mu_n+M, \lam_{n}^{I})}B_n(x, \lam^{I}_{n}) \varrho_n(x).$$
 Then $\varrho_n={\varrho_{n,1}\choose \varrho_{n,2}} $ solves:
\begin{equation}\label{case2-0}
(i\s_3\partial_x-\zeta^{I}_{n}-Q_n(\lam^{I}_{n}))\varrho_{n}=0,
\end{equation}
and  
\begin{equation}\label{case2-1}
\varrho_n= \begin{cases}e^{-i\zeta^{I}_{n} x}\begin{pmatrix}1\\0\end{pmatrix}\qquad \qquad\qquad\,\,\,\,\,{\rm if } \,\,\, x\geq \mu_n+M,\\
c_ne^{i(\zeta^{I}_{n} x-\Phi_n(\mu_n+M, \lam_{n}^{I}))}\begin{pmatrix}0\\1\end{pmatrix} \,\,\,{\rm if } \,\,\, x\leq -M.
\end{cases}
\end{equation}

Since by virtue of  \eqref{est-qn} and \eqref{case2-r}, one has  $\|Q_n(\lam^{I}_{n})\|_{L^1(\R)}\lesssim  1$, we infer that 
\begin{equation}\label{case2-2}
\|\varrho_{n}\|_{L^{\infty}([- M, \mu_n+M])}\lesssim 1.
\end{equation}
  Taking advantage of \eqref{case2-0}, we deduce  that
\begin{equation}\label{case2-30}\begin{aligned}
\partial_x\big(|\varrho_{n,1}|^2-|\varrho_{n,2}|^2\big)&=2 \im \zeta^{I}_{n} \big(|\varrho_{n,1}|^2+|\varrho_{n,2}|^2\big)\\&+ 2\im \big(q_n(\lam^{I}_{n}) \varrho_{n,2}\overline {\varrho}_{n,1}-
\overline{q_n(\overline {\lam^{I}_{n}}}) \overline {\varrho}_{n,2}\varrho_{n,1}\big).\end{aligned}
\end{equation}
Thanks to  \eqref{eqhat2} and  \eqref{case2-2}, there holds
$$\begin{aligned}\big|\im\big(q_n(\lam^{I}_{n}) \varrho_{n,2}\overline {\varrho}_{n,1}-
\overline{q_n(\overline\lam^{I}_{n})}\overline{\varrho}_{n,2}\varrho_{n,1}\big)\big| &=\big|\im\big(\overline {q_n(\lam^{I}_{n})}+\overline{q_n(\overline\lam^{I}_{n}})\big){\overline {\varrho}_{n,2}\varrho_{n,1}}\big)\big|\\ &\lesssim \re \lam^{I}_{n}\big(\big|\partial_x\tilde{\rm r}_n\big| +\big|\tilde{\rm r}_n\big|^3\big)\\&\lesssim  \im \zeta^{I}_{n}\big(\big|\partial_x\tilde{\rm r}_n\big| +\mu_n^{-\frac 3 4}\big),\end{aligned}$$
which, according to\refeq{case2-30}, implies   that 
\begin{equation}\label{case2-3}\begin{split}
\partial_x\big(|\varrho_{n,1}|^2-|\varrho_{n,2}|^2\big)&\lesssim \im \zeta^{I}_{n}(1+|\partial_x\tilde{\rm r}_n|).
\end{split}
\end{equation} 
Since by \eqref{case2-1}, we have
\begin{eqnarray*}
|\varrho_{n,1}(\mu_n+M)|^2-|\varrho_{n,2}(\mu_n+M)|^2=|\varrho_{n,1}(\mu_n+M)|^2\gtrsim 1,\\|\varrho_{n,1}(-M)|^2-|\varrho_{n,2}(-M)|^2=-|\varrho_{n,2}(\mu_n+M)|^2<0,
\end{eqnarray*}
the bound \eqref{case2-3} implies that 
$$ \mu_n  \im \zeta^{I}_{ n}\gtrsim1.$$

\bigskip

To deal with the case  $\tilde  \beta>0$,  we consider the monodromy matrix~$M_{U_n}(\lam)$ for~$\lam=\sqrt{\re \zeta^{I}_{n}+  \mu_n^{-1}p}\in \C_+$ with $p\in \C$, and as in \cite{BaPe2}, prove   the following  convergence result.  \begin{proposition}  
\label{limfond} {\sl  Let  $\cM_n  (p)=M_{U_n}(\sqrt{\re \zeta^{I}_n+ \mu_n^{-1} p})$.
There exists $q\in L^2([0,1])$ such that, up to a subsequence extraction,  \begin{equation}
\label{limsyst}\cM_{n} (p)\stackrel{n \to  + \infty} \longrightarrow \cM  (p)=\begin{pmatrix}\cM_{11}(p) &\cM_{12}(p)\\
\cM_{21}(p) &\cM_{22}(p)\end{pmatrix}, \end{equation} 
uniformly with respect to $p$ in compact subsets  of $\C$, where
\begin{equation}
\label{limsystval}  \cM  (p)=  e^{i\theta^* \sigma_3}\begin{pmatrix}\prod\limits_{j=\tilde I_1+1}^{I}\frac {p- (z^j+i) \tilde\beta} { p- (\overline{z^j} - i )\tilde\beta} &0\\
0&(-1)^{I}e^{-i\omega^*}\end{pmatrix}
\cM^* (p)\begin{pmatrix}1&0\\0&(-1)^{I} e^{i\omega^*}
 \prod\limits_{j=\tilde I_1+1}^{I}\frac {p- (\overline{z^j}-i) \tilde\beta} { p- ({z^j} + i )\tilde\beta}\end{pmatrix},
\end{equation}
$\theta^*=-2\pi\rho^*-\frac m 2$, $\omega^*=2\sum\limits_{j=1}^{\tilde I_1}\arg\varphi^j $ and  $\cM^* (p) $ is the monodromy matrix of the system
$$
i\sigma_3\partial_y \psi -p\psi +\begin{pmatrix}
0&   q  \\   
 \overline q&0
\end{pmatrix}\psi=0$$
on $[0,1]$. Explicitly, $\cM^*(p)=E(1,p)$ with $E(y,p)$ solving
$$\begin{cases}
i\sigma_3\partial_y E -pE +\begin{pmatrix}
0&   q  \\   
 \overline q&0
\end{pmatrix}E=0,\\
E(0,p)= \rm Id.
\end{cases}
$$
 }\end{proposition}
 \begin{proof} 
 According to\refeq{relation} and  \eqref{eqhat-00}, one has
 \begin{equation} \label{prop6.1-0}
M_{U_n} (\lam)= \cA^{-1}_n(\mu_n, \lam) B_n(\mu_n, \lam) E_{n}(\mu_n, \lam)B_n^{-1}(0, \lam)\cA_n(0, \lam)\, ,\end{equation}
where~$E_n(x, \lam)$ is the fundamental solution of the system \eqref{eqhat1}:
$$
(i\s_3\partial_x-\lam^2-Q_n(\lam))E_n=0, \quad E_n(0, \lam)=\rm Id.$$
Using Lemma\refer{specifyA} and  \eqref{eqhat-01}, \eqref{case2-r}, one can easily check that, for $\lam^2=\re \zeta^{I}_{n}+  \mu_n^{-1}p$ and~$n$ sufficiently large, the following bounds hold
\begin{equation}\label{prop6.1-1}
  |B_n(0, \lam)- {\rm Id}|+ |B_n(\mu_n, \lam)-e^{-\frac i 2 \|\tilde {\rm r}_n\|^2_{L^{2}(\R)}\sigma_3 }|\lesssim \mu_n^{-\frac14},
  \end{equation}
and 
\begin{equation}\label{prop6.1-2}\begin{split}
\left|\Big(\prod\limits_{j=1}^{I}\frac{\lam_n^j}{\overline\lam_n^j(\lam^2-\zeta_n^j)}\Big)\cA_n(0, \lam)-\begin{pmatrix}1&0\\0& (-1)^{I}
 \prod\limits_{j=1}^{I}\frac {\zeta_n^j\big(p-\mu_n\im\zeta_n^{I}(\overline z_n^j-i)\big)}{\overline\zeta_n^j\big(p-\mu_n\im\zeta_n^{I}(z_n^j+i)\big)}\end{pmatrix}\right|\lesssim \mu_n^{-\frac12},\\
 \left|\Big(\prod\limits_{j=1}^{I}\frac{\overline\lam_n^j(\lam^2-\zeta_n^j)}{\lam_n^j}\Big)\cA_n^{-1}(\mu_n, \lam)-\begin{pmatrix}
  \prod\limits_{j=1}^{I}\frac {\overline\zeta_n^j\big(p-\mu_n\im\zeta_n^{I}(z_n^j+i)\Big)}{\zeta_n^j\big(p-\mu_n\im\zeta_n^{I}(\overline z_n^j-i)\big)}&0
 \\0& (-1)^{I}
\end{pmatrix}\right|\lesssim \mu_n^{-\frac12},\end{split}
 \end{equation}
 uniformly with respect to $p$ in compact subsets of $\C\setminus\cup_{j=I_1+1}^{I}\{\tilde\beta(z^j+i), \tilde\beta(\overline{z^j}-i)\}$.
 
 \bigskip
 
 We next address $E_n(x, \lam)$. Combining  \eqref{case2} with \eqref{est-bet} and  \eqref{case2-r},  we readily deduce that:
$$\|Q_n(\lam)\big|_{\lam=\sqrt{\re \zeta^I_n+ \mu_n^{-1} p}}-Q_n^0\|_{L^1(\R)}\lesssim \frac1{\sqrt{\mu_n}}, \quad \|Q_n^0\|_{L^1(\R)}\lesssim 1,$$
uniformly with respect to $p$ in bounded sets of $\C$. As a consequence, we obtain
\begin{equation}\label{prop6.1-3}
\|E_n(\sqrt{\re \zeta^I_n+\mu_n^{-1} p})\|_{L^\infty([0, \mu_n])}\lesssim 1, \, \,   \|E_n(\sqrt{\re \zeta^I_n+\mu_n^{-1} p})-E_n^0(p)\|_{L^\infty([0, \mu_n])}\lesssim  \frac1{\sqrt{\mu_n}},
\end{equation}
with  $E_n^0(x,p)$ solving:
$$
(i\s_3\partial_x-\re\zeta_n^{I}-\mu_n^{-1}p-Q_n^0)E_n^0=0, \quad E_n^0(0, p)=\rm Id.$$
Setting 
\begin{equation}\label{limfond-2}
E_n^0(x,p)=e^{-i\re \zeta^I_{n}x\sigma_3}\tilde E_{n}^0(\mu_n^{-1}x, p).\end{equation}
we get that $\tilde E_n^0(y,p)$ solves
$$( i\sigma_3\partial_y-p-\tilde Q_n^0)\tilde E_n^0=0, \quad \tilde E_n^0(0,p)= \rm Id,$$
where 
 $$\tilde Q_n^0=\begin{pmatrix}0&\tilde q_n^0\\\overline{\tilde q^0_n}&0\end{pmatrix},\quad \tilde q_n^0(y)=\mu_ne^{2i\re\zeta^I_{n}\mu_n y}q_n^0(\mu_ny),$$
 with $q_n^0$ given by 
 \eqref{eqhat1-0}.
By virtue of \eqref{case2-r}, the sequence $(\tilde q_n^0)$ is bounded in $L^2([0,1])$ and therefore there exists $q\in L^2([0,1])$ such that, up to a subsequence extraction,
$$\tilde q_n^0\stackrel{n \to  + \infty} \rightharpoondown q \, \,\, {\rm in} \, \,L^2([0, 1]),$$
which implies  that  
\begin{equation}
\label{cvhatvn}\tilde E_n^0(y,p)  \stackrel{n \to  + \infty} \longrightarrow E(y,p),\end{equation}   
uniformly  with respect to $y\in [0, 1]$ and $p$ in bounded subsets of $\C$.

Combining \eqref{prop6.1-1} - \eqref{cvhatvn}, and using that $\frac{z_n^j}{\overline z_n^j} \stackrel{n \to  + \infty} \longrightarrow 1$ for $j=1, \dots, \tilde I_1$, we conclude that
\begin{equation}\label{limfond-3}
 \cM_{n} (p)\stackrel{n \to  + \infty} \longrightarrow \cM  (p),
 \end{equation}
 uniformly with respect to $p$ in compact subsets  of $\C\setminus \cup^{I} _{j=1}  \{\tilde \beta(z^j+i), \tilde\beta\overline {(z^j+i)}\}$,
 where
 $$\cM  (p)=
 e^{i\theta^* \sigma_3}\begin{pmatrix}\prod\limits_{j=\tilde I_1+1}^{I}\frac {p- \tilde\beta(z^j+i)} { p- \tilde \beta(\overline{z^j} - i )} &0\\
0&(-1)^{I}e^{-i\omega^*}\end{pmatrix}
\cM^* (p)\begin{pmatrix}1&0\\0&(-1)^{I} e^{i\omega^*}
 \prod\limits_{j=\tilde I_1+1}^{I}\frac {p- (\overline{z^j}-i) \tilde\beta} { p- ({z^j} + i )\tilde \beta}\end{pmatrix}.
$$

By analyticity of $M_{U_n}(\lam)$, the limiting matrix $\cM$ is an entire function of $p$ and the convergence in \eqref{limfond-3} is in fact uniform with respect to $p$ in bounded subsets of $\C$.

\end{proof}

\medbreak

\begin{remark}
{\sl The analyticity of the matrix $\cM$ implies that  the matrix $\cM^*=\begin{pmatrix}\cM_{11}^* &\cM_{12}^*\\
\cM_{21}^*&\cM_{22}^*\end{pmatrix}$ has to satisfy:
$$\cM_{11}^*(-i\tilde\beta)=\cM_{22}^*(i\tilde\beta)=0.$$
Since $\det \cM^*=1$, this implies that
$\cM_{12}^*(\pm i\tilde\beta)\neq 0$ and $\cM_{21}^*(\pm i\tilde\beta)\neq 0$. }
\end{remark}

\medbreak

\begin{cor}
\label{corimp} {\sl  The matrix $\cM  (p)$  has the following properties:\begin{itemize}

\item[(i)] $\tr \cM  (p) =2 \cos\Big(p-\theta^*)$;

\item[(ii)] $\cM  (p)= e^{i (\theta^*-p) \sigma_3} + o (e^{|\im p|})$, as $|p|\rightarrow \infty$.

\item[(iii)] $\cM_{12}(\pm i \tilde\beta)\neq 0 $ and $ \cM_{21}(\pm i\tilde\beta)\neq 0$.

\end{itemize}}\end{cor}
\begin{proof} 
Recalling that  $\tr \cM_n(p)=\Delta_{u_{n}}(\sqrt{\mu_n\re\zeta^I_{n}+p})$ and applying \eqref{p1-1}, we obtain
$$\tr \cM_n(p)=2\cos\big(\mu_n\re\zeta^I_{n}+\frac12\|u_{n}\|_{L^2(\TT)}^2+p\big)+o(1), \quad {\rm as}\,\,\, n\to \infty,$$
which after passing to the limit $n \rightarrow \infty$, gives (i).
The asymptotics in   (ii)  follow from\refeq{limsystval} and the fact that    $\cM^*(p)=e^{-i  p  \sigma_3} + o(e^{|\im p|}) $, as  $|p| \to \infty$. Finally, 
(iii) is immediately given by
  $\cM_{12}(p)=(-1)^{I}e^{i(\theta^*+\omega^*)}\cM_{12}^*(p)$ and
$\cM_{21}(p)=(-1)^{I}e^{-i(\theta^*+\omega^*)}\cM_{21}^*(p)$.
\end{proof}

\medbreak

Since $\det \cM=1$, item (i) of Corollary \ref{corimp} implies that for $p=p_k$ with $p_k=\pi k+\theta^*$,~$k\in \Z$, 
either
\begin{equation}\label{a}
\cM(p_k)=(-1)^k I,
\end{equation}
 or
 \begin{equation}\label{b33}
 |\cA^D(p_k)|^2 + |\cA^N(p_k)|^2>0,
 \end{equation}
 where the functions $\cA^D$, $\cA^N$ are defined by:
\begin{eqnarray*}    \cA^D(p)   &= & \frac i 2\big( \cM_{11} (p)+ \cM_{12} (p)-\cM_{22} (p)-\cM_{21} (p) \big), \\    \cA^N(p)   &= &\frac i 2\big(  \cM_{11} (p)- \cM_{12} (p)+\cM_{21} (p) -\cM_{22} (p) \big).\end{eqnarray*} 
Thanks to  Lemma\refer{DN}, one necessarily 
has \eqref{a}, for all $k\in \Z$. Indeed, if not then there exists~$k_0\in \Z$  such that, for example\footnote{The case $|\cA^N(p_{k_0})|> 0$ can be treated in the same way.}, $|\cA^D(p_{k_0})|=c_0> 0$.  Since $\cM$ is  continuous, there exists $\eta>0$ such that, for all $p$ in  $[p_{k_0}-\eta,  p_{k_0}+ \eta]$, we have  
$$|\cA^D(p)| \geq \frac {c_0} 2.$$  
 Since  $\cM_{n}(p) \stackrel{n\to \infty} \longrightarrow \cM(p)$,  uniformly with respect to $p$ in compact subsets  of~$\C$, this implies that  there exists an integer~$n_0$ so that
 \begin{equation}
\label{diricprop}|A_{U_n}^D(\lam) |\geq \frac {c_0} 4\,\virgp \end{equation} 
  for all~$n \geq n_0$, and 
 all $\lam \in \C_+$ such that $\lam^2= \re \zeta^{I}_n+ p\, \mu^{-1}_n$ with  $p \in [p_{k_0}-\eta,  p_{k_0} + \eta]$.
Recalling that $A^D_{U_n}(\lam)=A^D_{\tilde u_n}(\sqrt{\mu_n}\lam)$ with  $\tilde u_n(x)=u_n(t_n, x+x_n)$ for some $x_n\in \R$, we obtain
that  \begin{equation}
\label{diricpropbis}|A^D_{\tilde u_n}(\lam) |\geq \frac {c_0} 4\,\virgp \end{equation}
  for all~$n \geq n_0$, and all $\lam\in \C_+ $ with $\lam^2  \in [\mu_n \re \zeta^{I}_n+ p _{k_0}-\eta,  \mu_n \re\zeta^{I}_n+ p_{k_0} + \eta]$.  

\smallskip Taking into account that
$$\mu_n  \re \zeta^{I}_n+p_{k_0}= \pi  (2N_n+k_0) -\frac12\|u_{n}\|^2_{L^2(\TT)} + o(1), \quad {\rm as}\,\, n\rightarrow \infty,$$
we deduce that there exists $n_1\in \N$ such that for all $n\geq n_1$, $A^D_{\tilde u_n}(\lam)$ does not vanish for~$\lam\in \C_+$ with  $\lam^2\in \Big[\pi  (2N_n+k_0) -\frac12\|u_n\|^2_{L^2(\TT)}-  \frac\eta2,  \pi  (2N_n+k_0) -\frac12\|u_n\|^2_{L^2(\TT)} +\frac\eta2\Big]$,
which in view of the boundedness of $(u_0^{(n)})$ in $H^1(\TT)$,   contradicts Lemma \ref{DN}.
Thus,
\begin{equation}\label{aa}
\cM(p_k)=(-1)^k I,\quad \forall\,\, k\in \Z,
\end{equation}
 which means that
\begin{equation}
\label{consclaim1} \cM_{12} (p_k)=\cM_{21} (p_k)=0, \forall \,\, k\in \Z.  \end{equation} 
 We deduce   that   the function
 $$ \varphi(p)= \frac {\cM_{12} (p +\theta^*)} {\sin p}\,  $$
  is an entire function of $p\in\C$, which in view of  Corollary \ref{corimp} (ii),  satisfies 
 \begin{equation}
\label{consbisclaim1}  \varphi(p)\stackrel{|p| \to \infty} \longrightarrow 0 \, .\end{equation} 
  So $\cM_{12} \equiv 0$ which  contradicts (iii) of Corollary \ref{corimp}, and    completes the proof of the theorem.
\bigskip 


\appendix

\section{Regularized Determinants}\label{basicdeterminants}
In this appendix, we recall  the basic properties of  the regularized determinants  ${\rm det}_n({\rm I}-A)$ for~$A$ in~$\mathscr{C}_n$,   the set of bounded operators~$A$ on
a separable Hilbert space\footnote{In our case $\cH$ is  $L^2(\R, \C^2)$.} $\cH$ such that~$|A|^n$ is of trace-class, endowed with the norm $\|A\|_n = \big[{ \rm Tr}  \big(|A|^n\big)\big] ^{\frac 1 n}$. For further details, we refer to the monograph  of Simon\ccite{Simon} and the references therein. 

\medskip To introduce the  regularized determinants, let us  start by defining,  for any bounded operator~$A$ on $\cH$,
$$R_n (A)={\rm I}- ({\rm I}-A) \exp \Big(\sum^{n-1}_{k=1} \frac {A^k} k\Big)  \cdotp$$ 
Clearly, 
\begin{equation}
\label{RA0}
R_n(A)=A^nh_n(A),
\end{equation}
where $h_n$ is an entire function on $\C$.
This shows  that $R_n (A)$ belongs to~$\mathscr{C}_1$ if $A$ is in~$\mathscr{C}_n$,
which justifies the following definition:
\begin{definition}  
\label{defdety}
{\sl For any  operator  $A$ in~$\mathscr{C}_n$, $n \geq 2$, we define 
\begin{equation} 
\label{defdetn}{\rm det}_n ({\rm I}-A) = {\rm det}({\rm I}-R_n (A)) = {\rm det}\Big(({\rm I}-A) \exp \Big(\sum^{n-1}_{k=1} \frac {A^k} k\Big)\Big)\cdotp\end{equation}
}
\end{definition}

 Note  that, for all $A$ in $\mathscr{C}_{n}$ such that $\|A\|<1$ (or more generally $\|A^p\|<1$, for some~$p$), one has: \begin{equation} 
\label{dety2}{\rm det}_n ({\rm I}-A)=\exp \Big( - {\rm Tr} \sum^{\infty}_{k=n} \frac {A^k} k\Big) \cdotp\end{equation}

In the following proposition, we  summarize some useful properties of the regularized determinants:
\begin{proposition}
\label{defdetyrk}
{\sl For any integer $n\geq 1$,  there exists a positive constant~$C_n$ such that the following estimates hold.
\begin{enumerate} 
\item  For all~$A \in \mathscr{C}_n$,
\begin{equation} 
\label{boundedest}\big|{\rm det}_n ({\rm I}-A)\big| \leq \exp \big(C_n \|A\|_n^n\big), \end{equation}   
\begin{equation} 
\label{detyest3}|{\rm det}_n ({\rm I}-A)- 1| \leq C_n\|A^n\|_1 \exp \big(C_n \|A\|_n ^n\big). \end{equation}
 \item For all~$A, B$ in $\mathscr{C}_n$, 
 \begin{equation} 
\label{detyest4}|{\rm det}_n ({\rm I}-A)- {\rm det}_n ({\rm I}-B))| \leq \|A-B\|_n \exp \big(  C_n \big(\|A\|_n+ \|B\|_n+  1\big)^n\big) .
\end{equation}
\item Let ~$A \in \mathscr{C}_n$. Then~${\rm I}-A$ is invertible if and only if ~${\rm det}_n ({\rm I}-A)\neq 0$, and furthermore, one has 
\begin{equation} 
\label{detyest5}\| ({\rm I}-A)^{-1}\| \leq \frac {C_n} {|{\rm det}_n ({\rm I}-A)|} \exp\big(  C_n \ \|A\|^n_n\big). \end{equation}
 \end{enumerate} }
\end{proposition}

\medbreak


\section{Zakharov-Shabat spectral problem}
 The aim of this section is  to  recall some  classical  properties of    the solutions of the Zakharov-Shabat spectral problem: 
 \begin{equation}\label{ap-ZS}
 (\cL_{\bf q}-\zeta)\psi=0, \,\,\cL_{\bf q}=i\sigma_3\partial_x-Q, \,\,\,Q=\begin{pmatrix}0&q_1\\q_2&0\end{pmatrix}, \quad {\bf q}=(q_1, q_2).
 \end{equation}
   For  further details, the reader can consult\ccite{BaPe2, bealscoifman, GKap, GKap2} and the references therein.

 \medskip
Given $ {\bf q}=(q_1, q_2)\in L^2_{loc}(\R, \C^2)$, we denote $E_{\bf q}(x, \zeta)$   the   fundamental solution of \eqref{ap-ZS} satisfying~$E_{\bf q}(0, \zeta)=\rm Id$.
Using that
 
\begin{equation}
\label{ap-ZH-1} 
 E_{\bf q} (x, \zeta)=e^{-i\zeta x\sigma_3}-i\int_0^xe^{-i\zeta(x-y)\sigma_3}\sigma_3Q(y) E_{\bf q} (y, \zeta) dy,
 \end{equation}
  one readily deduces (see for exemple \cite{GKap, GKap2}) that for all $x\in \R_+$,  $E_{\bf q} (x, \zeta)$ is an analytic function of $\zeta$ and ${\bf q}$ admitting the following estimate
  \begin{equation}\label{ap-ZH-2}
|E_{\bf q} (x, \zeta)|  \leq  \exp\big(|\im \zeta| x+ \|{\bf q}\|_{L^2([0,x], \C^2)}\sqrt x \big), \quad \forall (x,\zeta)\in \R_+\times \C.\end{equation}
Furthermore, as $|\zeta|\rightarrow \infty$, one has
\begin{equation}\label{gk}
E_{\bf q}(x, \zeta)=e^{- ix \zeta \sigma_3}+ o(e^{|\im \zeta  x|}),\end{equation}
locally uniformly with respect to $x\in \R$.
We also recall the following continuity property
of  $E_{\bf q}$: if  $({\bf q}_n)_{n \in \N}$  converges weakly to  ${\bf q}$ in $L_{\rm loc}^2$, as~$n\to \infty$, then 
\begin{equation}  \label{cv} E_{{\bf q}_n} (x, \zeta)  \stackrel{n\to\infty}\longrightarrow 
E_{\bf q} (x, \zeta),
\end{equation}uniformly on bounded subsets of $\R\times \C$. 

\medskip

 If $\im \zeta>0$, then  denoting the columns of $E_{\bf q}(x, \zeta)$ by $e_1(x, \zeta;{\bf q})$ and $e_2(x, \zeta;{\bf q})$, one has the following bounds.
\begin{lemma}[\cite{BaPe2}]\label{jost} 
{\sl There exists $C>0$ such that the following estimates hold
\begin{equation}\label{jost-1}
\Big|e^{i\zeta x}{\bf p}e_1(x, \zeta; {\bf q^{(1)}})-\begin{pmatrix}1\\0\end{pmatrix}\Big| \leq  \exp\biggl(\frac{\|{\bf q}^{(1)}\|^2_{L^2([0, x])}}{4\im \zeta}\biggr)\frac{\|{\bf q}^{(1)}\|^2_{L^2([0, x])}}{4\im \zeta},\quad {\bf p}=\begin{pmatrix}1&0\\ 0&0\end{pmatrix},
\end{equation}
and
\begin{equation}
\big|e^{i\zeta x}\big(e_1(x, \zeta; {\bf q^{(1)}})-e_1(x, \zeta; {\bf q}^{(2)})\big)\big|\leq Ce^{\frac C{\im \zeta}\big(\|{\bf q}^{(1)}\|^2_{L^2([0, x])}+\|{\bf q}^{(2)}\|^2_{L^2([0, x])}\big)}\frac {\|{\bf q}^{(1)}-
{\bf q}^{(2)}\|_{L^2([0, x])}}{\sqrt{\im \zeta}},
\end{equation}
for all $(x, \zeta)\in \R_+\times \C_+$ and all
$ {\bf q}^{(1)}=(q_1^{(1)}, q_2^{(1)}) , \, {\bf q}^{(2)}=(q_1^{(2)}, q_2^{(2)}) \in L^2_{loc}(\R, \C^2) $.}
\end{lemma}

\medskip

We next assume that ${\bf q} \in L^1(\R, \C^2)$, and consider the Jost solutions $\psi_{1}^\mp(x, \zeta; {\bf q})$ and~$\psi_{2}^\mp(x, \zeta; {\bf q})$  of the system\refeq{ap-ZS}: 
\begin{equation}\label{Jost1-ZS}\begin{split}
\psi_1^-(x, \zeta;{\bf q})&= e^{ -i \zeta x} \left[ \left(
\begin{array}{ccccccccc}
1  \\
0 
\end{array}
\right) + o(1)\right], \quad  x \to - \infty, \quad \im \zeta\geq 0,\\
\psi_2^-(x, \zeta; {\bf q}) &=   e^{ i \zeta x} \, \, \,
 \left[ \left(
\begin{array}{ccccccccc}
0  \\
1 
\end{array}
\right) + o (1)\right], \quad x \to - \infty, \quad \im \zeta\leq 0,\\
\psi_1^+(x, \zeta;{\bf q})&= e^{ -i \zeta x} \left[ \left(
\begin{array}{ccccccccc}
1  \\
0 
\end{array}
\right) + o(1)\right], \quad  x \to + \infty, \quad \im \zeta\leq 0,\\
\psi_2^+(x, \zeta; {\bf q}) &=   e^{ i \zeta x} \, \, \, \left[ \left(
\begin{array}{ccccccccc}
0  \\
1 
\end{array}
\right) + o (1)\right], \quad x \to - \infty, \quad \im \zeta\geq 0.
\end{split}\end{equation}

\begin{lemma}
\label{J-ZS}
{\sl For  all $x\in \R$,  $\psi_1^- (x, \zeta; {\bf q})$,   $\psi_2^+ (x, \zeta; {\bf q})$ are continuous functions of $\zeta$ in $\overline\C_+$, holomorphic functions of $\zeta\in \C_+$, 
and analytic in ${\bf q} \in L^1(\R, \C^2)$. Similarly,  $\psi_1^+ (x, \zeta; {\bf q})$ and~$\psi_2^- (x, \zeta; {\bf q})$ are continuous with respect to  $\zeta\in \overline \C_-$,
holomorphic with respect to~$\zeta\in \C_-$, and analytic in  ${\bf q} \in L^1(\R, \C^2)$.} Furthermore,
\begin{enumerate}
\item[(i)]
 if ${\bf q} \in W^{1,1}(\R, \C^2)$, then 
\begin{equation}\label{J-ZS-1}\begin{split}
\sup\limits_{x\in \R, \, \zeta\in \overline\C_\pm} \langle \zeta\rangle\Big |e^{i\zeta x}\psi_1^\mp(x, \zeta; {\bf q})-{1\choose 0}\Big|\lesssim 1,\\
\sup\limits_{x\in \R, \, \zeta\in \overline\C_\pm} \langle \zeta\rangle \Big|e^{-i\zeta x}\psi_2^\pm(x, \zeta; {\bf q})-{0\choose 1}\Big|\lesssim 1,
\end{split}
 \end{equation}
uniformly with respect to ${\bf q} $ in bounded subsets of  $W^{1,1}(\R, \C^2)$;
\item[(ii)] if $\langle x\rangle {\bf q}\in L^1(\R, \C^2)$, then $\psi_1^- (x, \zeta; {\bf q})$,   $\psi_2^+ (x, \zeta; {\bf q})$ $($resp. 
$\psi_1^+(x, \zeta; {\bf q})$,   $\psi_2^- (x, \zeta; {\bf q})$$)$ belong to $C^1(\overline \C_+)$ $($resp. to $C^1(\overline \C_-)$$)$, and for all $R\in \R$ one has
\begin{equation}\label{J-ZS-1-1}\begin{split}
\sup\limits_{x\lessgtr R, \, \zeta\in \overline\C_\pm}  \Big|\partial_\zeta\big(e^{i\zeta x}\psi_1^\mp(x, \zeta; {\bf q})\big)\Big|\lesssim_R 1,\\
\sup\limits_{x\gtrless R, \, \zeta\in \overline\C_\pm} \Big|\partial_\zeta\big(e^{-i\zeta x}\psi_2^\pm(x, \zeta; {\bf q})\big)\Big|\lesssim_R 1,
\end{split}
\end{equation}
uniformly with respect to $\langle x\rangle{\bf q} $ in bounded subsets of  $L^1(\R, \C^2)$.
\end{enumerate}
\end{lemma}

\begin{proof} We will only prove the statement for $\psi_1^-$, the proof for the  other Jost functions is similar.
Writing $\psi_1^-(x, \zeta; {\bf q})= e^{-i\zeta x}\chi(x, \zeta; {\bf q}) $ and viewing $\chi$ as  the solution of the following Volterra equation:
\begin{equation}
\label{ap-ZH-11} 
 \chi(x, \zeta; {\bf q})={1\choose 0}-i\int_{-\infty}^x K(x,y, \zeta; {\bf q)}\chi(y, \zeta;{\bf q}), \quad K(x,y;\zeta;{\bf q})= \begin{pmatrix}1&0\\0&-e^{2i\zeta(x-y)}\end{pmatrix} Q(y),
 \end{equation}
 one concludes immediately that for all  $x\in \R$, $ \chi(x, \zeta; {\bf q})$ is  continuous with respect to $\zeta$ in~$\overline \C_+$, holomorphic in $\zeta\in \C_+$, 
 analytic in ${\bf q}$ as a map from  $L^1(\R, \C^2)$ to $H^\infty(\C_+)\cap C(\overline \C_+)$, and admits the bound
\begin{equation}\label{J-ZS-1-2}
\sup\limits_{x\in \R, \, \zeta\in \overline \C_+} |\chi(x, \zeta; {\bf q})|\leq e^{\|{\bf q}\|_{L^1(\R, \C^2)}}.
 \end{equation}
 Furthermore, since 
 $$\big|\partial_\zeta K(x,y, \zeta; {\bf q})\big|\leq 2(x-y)|{\bf q}(y)|, \quad \forall \,y\leq x, \, \, \im \zeta\geq 0,$$
 we deduce from \eqref{ap-ZH-11} that $\chi$ is a $C^1$ function of $\zeta\in \overline \C_+$ admitting the bound \eqref{J-ZS-1-1}, provided $\langle x\rangle {\bf q}\in L^1(\R, \C^2)$.
 
 \smallskip Finally, to prove \eqref{J-ZS-1}, we write
 $$ \chi(x, \zeta; {\bf q})=\begin{pmatrix}1\\0\end{pmatrix}+\tilde\chi(x, \zeta; {\bf q}).
 $$ 
 Then,   $\tilde\chi$ satisfies
 \begin{equation}\label{chi-ZS}
 \tilde\chi(x, \zeta; {\bf q})=\chi_0(x, \zeta; {\bf q})\begin{pmatrix}0\\1\end{pmatrix}-i\int_{-\infty} ^xK(x,y, \zeta; q)\tilde\chi(y, \zeta; {\bf q}) dy,
 \end{equation}
with
$$\chi_0(x, \zeta; {\bf q})=i\int_{-\infty}^xe^{2i\zeta (y-x)}q_2(y)dy.$$
Assuming that $q_2\in W^{1,1}(\R)$, and integrating by parts, we get
$$\sup\limits_{x\in \R}|\chi_0(x, \zeta;{\bf q})|\leq\frac{\|\partial_xq_2\|_{L^1(\R))}}{|\zeta|},\quad \forall\, \zeta\in \overline\C_+\setminus\{0\},$$
which in view of \eqref{chi-ZS}, leads to the bound
$$\sup\limits_{x\in \R}|\chi(x, \zeta; {\bf q})|\leq e^{\|{\bf q}\|_{L^1(\R, \C^2)}}\frac{\|\partial_xq_2\|_{L^1(\R))}}{|\zeta|}, \quad \forall\, \zeta\in \overline\C_+\setminus\{0\}.  $$
Combining this inequality with \eqref{J-ZS-1-2}, one gets the result.
\end{proof}

\medbreak

Introducing
\begin{equation}\label{ab-ZS}
a_{\bf q}(\zeta)=\det(\psi_1^-(x,\zeta;{\bf q}), \psi_2^+(x,\zeta; {\bf q})), \quad b_{{\bf q}}(\zeta)=\det(\psi_1^+(x,\zeta;{\bf q}), \psi_1^-(x,\zeta; {\bf q})),
\end{equation}
we deduce from Lemma \ref{J-ZS}  the following result.
\begin{lemma} \label{ab-ZH} {\sl  The map ${\bf q}\to a_{\bf q}$ $($resp.  the map ${\bf q}\to b_{\bf q}$$)$ is analytic  from $L^1(\R, \C^2)$ to~$H^\infty(\C_+)\cap C(\overline \C_+)$ $($resp. to  $L^\infty(\R)\cap C(\R)$$)$.
  In addition, one has :
\begin{enumerate}
\item[(i)] if ${\bf q}\in W^{1,1}(\R, \C^2)$, then
\begin{equation}\label{est-ab-ZS-0}
\sup\limits_{\zeta\in \overline\C_+} \langle \zeta\rangle |a_{\bf q}(\zeta)-1|+\sup\limits_{\zeta\in \R} \langle \zeta\rangle |b_{\bf q}(\zeta)|\lesssim 1,
\end{equation}
uniformly with respect to ${\bf q} $ in bounded subsets of  $W^{1,1}(\R, \C^2)$;
\item[(ii)]
if $\langle x\rangle {\bf q}\in L^1(\R, \C^2)$, then $a_{\bf q}(\zeta)$ $($resp.  $b_{\bf q}(\zeta)$$)$ is a $C^1$ function of $\zeta\in \overline \C_+$ $($resp. of~$\zeta\in \R$$)$, 
and 
\begin{equation}\label{est-ab-ZS-00}
\sup\limits_{\zeta\in \overline\C_+}  |\partial_\zeta a_{\bf q}(\zeta)-1|+\sup\limits_{\zeta\in \R}| \partial_\zeta b_{\bf q}(\zeta)|\lesssim 1,
\end{equation}
uniformly with respect to $\langle x\rangle{\bf q} $ in bounded subsets of  $L^1(\R, \C^2)$.
\end{enumerate}}
\end{lemma}

To prove Proposition\refer{car-sabis}, we will  also need the following improvement of \eqref{est-ab-ZS-0} in the case of $q\in W^{N,1}(\R)$, $N\geq 2$.
\begin{lemma}\label{b-ZS}
Let $N\in \N^*$, and ${\bf q}\in W^{N, 1}(\R, \C^2)$. Then,
\begin{equation} \label{est-b-ZS}
\sup\limits_{\zeta\in \R} \langle \zeta\rangle^N|b_{\bf q}(\zeta)|\lesssim 1,
\end{equation}
uniformly with respect to ${\bf q} $ in bounded subsets of  $W^{N,1}(\R, \C^2)$.
\end{lemma}

Lemma \ref{b-ZS} can be readily deduced from the proof of Theorems E and 6.1 in \cite{bealscoifman}.
\section{Proof of Proposition\refer{keygen}}\label{app-gw}

In order to prove Proposition\refer{keygen}, we proceed by contradiction, assuming that there exists a sequence $(u_n)_{n \in \N} \subset \cS(\R)$  such that, for all $n$, $M(u_n) \leq M$, $\tilde a_{u_n}$ has no zeros  in the angle~$\big\{ \zeta\in \C_+:  \frac \pi 2 + \frac 1 n <  \arg  \zeta< \pi \big\}$, and 
\begin{equation}\label{C0}
\|u_n\|^2_{\dot H^{1}(\R)} \geq n (P^2_\R(u_n)+ |E_\R(u_n)|).
\end{equation}
 By scale invariance, one can always assume that  $\|u_n\|_{\dot H^{1}(\R)}=1$. Then  according to\refeq{asympek},  
 \eqref{mass1-0} and \eqref{C0}, one gets:
 \begin{eqnarray}
\label{eq1} P_\R(u_n)&=&- 8\sum\limits_{j} \im\zeta^n_j-\frac{4}{\pi}\int^\infty_{-\infty}   s   \,  d\mu_{u_n} (s) \stackrel{n \to  + \infty} \longrightarrow0, \\
\label{eq2} E_\R(u_n)&=&16 \sum\limits_{j} \re\zeta^n_j\im\zeta^n_j+\frac{8}{\pi}\int^\infty_{-\infty}  s^{2}  \,  d\mu_{u_n}  (s) \stackrel{n \to  + \infty} \longrightarrow0 , \end{eqnarray} 
 and 
 \begin{equation}
 \label{mass}\|u_n\|^2_{L^{2}(\R)}= 4 \sum_{j}  \arg \zeta^n_j + \frac 2 {\pi}  \int^\infty_{-\infty}  d\mu_{u_n} (s)\leq M, \end{equation} 
where $\zeta^n_j$ are  the zeros  of~$\tilde  a_{u_n}$ in $\C_+$, counted with their multiplicity,   and  $\mu_{u_n} $  is the positive measure  on $\R$ with finite moments of all orders, that arises in the representation
\begin{equation} \label{gena-n}
\tilde a_{u_n} (\zeta) =e^{-\frac{i}2\|u_n\|^2_{L^2(\R)}}\prod_{j}  \Big( \frac {\zeta -\zeta^n_j}  {\zeta - \overline \zeta^n_j} \Big) \exp\Big(\frac 1  {i\pi} \int^\infty_{-\infty} \frac {s} {\zeta-s} \, d\mu_{u_n} (s)\Big).
\end{equation}
 We claim that
\begin{equation}
\label{negrealz}\int^\infty_{-\infty}  s^{2}  \,  d\mu_{u_n}  (s) 
\stackrel{n \to  + \infty} \longrightarrow0, 
\end{equation}
which together  with   \eqref{eq1} and \eqref{mass}   ensures that
\begin{equation}
\label{resume}
\int^\infty_{-\infty}   |s| d\mu_{u_n} (s)\stackrel{n \to  + \infty} \longrightarrow0, \quad\sum\limits_{j}  \im\zeta^n_j\stackrel{n \to  + \infty} \longrightarrow0.\end{equation}

Indeed, it follows from \eqref{eq2}  that, as $n$ goes to infinity,  
\begin{equation}\label{C1}
 \sum_{ \re \zeta^{n}_j>0}  \re\zeta^n_j\im\zeta^n_j+\frac{1}{2\pi}\int^\infty_{-\infty}  s^{2}  \,  d\mu_{u_n}  (s) = -
\sum_{ \re \zeta^{n}_j<0} \re\zeta^n_j\im\zeta^n_j + o (1).
\end{equation}
Since the zeros $\zeta^n_j$ are located in the angle~$\big\{\zeta\in \C_+:  0<  \arg  \zeta< \frac \pi 2+\frac 1 n \big\}$, we have 
$$-\sum_{ \re\zeta^{n}_j<0} \re\zeta^n_j\im\zeta^n_j \lesssim \frac1n \sum_{ \re\zeta^{n}_j<0} (\im\zeta^n_j)^2\lesssim 
\frac1n \big(\sum_j\im\zeta^n_j\big )^2,$$
which together  with \eqref{eq1} and  \eqref{mass}  leads to the bound
$$-\sum_{ \re\zeta^{n}_j<0} \re\zeta^n_j\im\zeta^n_j \lesssim \frac1n\Big(P^2_\R(u_n)+\int^\infty_{-\infty}  s^{2}  \,  d\mu_{u_n}  (s)\Big),
$$
and therefore, in view of \eqref{eq1} and \eqref{C1}, 
gives \eqref{negrealz}.

\medskip To arrive at a contradiction, 
 we shall analyze  the  profile decomposition  of the  sequence~$(u_n)_{n \in \N}$  with respect to
the Sobolev embedding~$H^{1}(\R) \hookrightarrow L^{p}(\R)$,~$p>2$, and prove the following lemma.

\begin{lemma}
\label{endkey}
{\sl  There exist  an integer~$1 \leq L_1 \leq   \frac M {4 \pi}$,  a family of  non-zero profiles~$(V^{(\ell)})_{1 \leq \ell \leq L_1}$ in~$H^1(\R)$ satisfying, for all~$1 \leq \ell\leq L_1$,~$\tilde a_{V^{(\ell)}} \equiv 1$, and a family of orthogonal cores~$(\underline{x}^{(\ell)})_{1 \leq \ell \leq L_1}$
 such that,
up to a subsequence extraction,   
\begin{equation}\label{decH1bis} u_n(x)= \sum_{\ell=1}^{L_1}  V^{(\ell)}(x-x^{(\ell)}_n)  + {\rm
r}_n(x)\quad {\rm with} \,\,\,\|{\rm
r}_n\|_{\dot H^{1}(\R)}\stackrel{n\to\infty}\longrightarrow  0.\end{equation}
}\end{lemma} 

The decomposition \eqref{decH1bis} leads immediately to a contradiction since it ensures that for all~$n$ sufficiently large 
$\tilde a_{u_n}(\zeta)$ has at least$ \frac1{4\pi}\sum_{\ell=1}^{L_1}M_\R(V^{(\ell)})$ zeros in the angle   $\{\zeta\in \C_+:   \, \,  \frac{3\pi}4\leq\arg \zeta<\pi\big\}$,
see the proof of  Lemma\refer{lemzerosH1I0}.

 \begin{proof}[Proof of Lemma \ref{endkey}] In order to prove  the lemma, we  argue as in the proof of Theorem~3 in\ccite{BaPe}.  The boundedness of the sequence 
  $(u_n)_{n \in \N}$   in $H^1(\R)$ together with the fact that~$E_\R(u_n) \stackrel{n \to  + \infty} \longrightarrow 0$, ensures that
   there exist a sequence of  profiles $( {V^{(\ell)} })$ in $H^{1}(\R)$ which are not all zero,
 and a sequence of orthogonal cores  $(\underline{x}^{(\ell)})$   
  such that, up to a subsequence extraction,  we have for any integer 
$L\geq 1$ and all $p>2$, \begin{equation} \label{decompN}
u_n(x)=\sum_{\ell=1}^{L} V^{(\ell)}(x-x^{(\ell)}_n) +{\rm
r}_n^{L}(x), \quad  \limsup_{n\to\infty}\;\|{\rm
r}_n^{L}\|_{L^p(\R)}\stackrel{L\to\infty}\longrightarrow 0,\end{equation}
and 
 \begin{eqnarray}\label{ortogonalth2H1} 
 \|u_n\|_{L^2(\R)}^2&=&\sum_{\ell=1}^{L} \|  V^{(\ell)}\|_{L^2(\R)}^2+\|
 {\rm
r}_n^L\|_{L^2(\R)}^2+o(1),\\
\label{ortogonalth2H2} \|u_n\|_{\dot H^1(\R)}^2&=&\sum_{\ell=1}^{L} \| V^{(\ell)}\|_{\dot H^1(\R)}^2+\|
 {\rm
r}_n^L\|_{\dot H^1(\R)}^2+o(1),\\
\label{ortogonalth2H3}E_\R(u_n)&=&\sum_{\ell=1}^{L} E_\R( V^{(\ell)})+E_\R({\rm r}_n^L)+o(1),
\end{eqnarray}
as $n\rightarrow \infty$.

\medskip
The next step will be to show that 
 for any  profile~$V^{(\ell)}$ involved in the  decomposition\refeq{decompN}, the function $\tilde a_{V^{(\ell)}}$ does not vanish in $\C_+$.
For that purpose, we  proceed by contradiction, assuming that there exist~$\ell_0 \in \N$,~$\lam_0\in \C_{++}$ and~$\psi_{0} \in H^1(\R)$ such that~$\|\psi_{0} \|_{L^2( \R)}=1$ and $L_{V^{(\ell)}}(\lambda_0)\psi_0=0$.
Then,  considering the system 
\begin{equation} \label{syststudy} \left(i \sigma_3 \partial_x - \lam_0^2   -i   \lam_0 \left(
\begin{array}{ccccccccc}
0 &u_n(x) \\
 \overline {u_n} (x)&0
\end{array}
\right) \right)\psi_{0}(\cdot -x^{(\ell_0)}_n) = \cR_n(x)\, ,\end{equation}  and taking advantage of the orthogonality condition between the cores~$(\underline{x}^{(\ell)})$, one can easily check that
\begin{equation} \label{l2normsyststudy} \|\cR_n\|_{L^2( \R)} \stackrel{n \to  + \infty} \longrightarrow 0\, .\end{equation} 
By virtue of  \eqref{gena-n} and \eqref{resume}, we have 
 \begin{equation} \label{usefulest}  
 |a_{u_n} (\lam_0)|\stackrel{n \to  + \infty} \longrightarrow1.
 \end{equation}
 Recalling \eqref{awithdet},  \eqref{T-HS}, and taking into account the bound \eqref{detyest5} and the representation~$L_{u_n} ^{-1}(\lam_0)=(I-T_{u_n}(\lam_0))^{-1}(\mathcal{L}-\lam_0^2)^{-1}$, we deduce from \eqref{usefulest} that
  \begin{equation}
\label{inverseopbis}  \|L_{u_n} ^{-1}(\lam_0)\|\lesssim 1,
\end{equation}
which  leads to a contradiction since $\|\psi_{0} \|_{L^2( \R)} =1$ and $\|\cR_n \|_{L^2( \R)} \stackrel{n \to  \infty}\longrightarrow 0$. This concludes the  proof of the fact that, for any profile  $V^{(\ell)}$ involved in\refeq{decompN},  the function $\tilde a_{V^{(\ell)}}$ has no zeros  in $\C_+$.

\medskip
Since the functions $\tilde a_{V^{(\ell)}}$ don't vanish on $\C_+$,  Lemma \ref{studyfirstpartimaginary} ensures    that   
\begin{equation}\label{E-pos}
E_\R(V^{(\ell)}) \geq 0, \quad \forall\, \ell.
\end{equation}
Taking into account that $E_\R(u_n)\stackrel{n\to\infty}\longrightarrow 0$ and $\limsup_{n\to\infty}\;\|{\rm
r}_n^{L}\|_{L^p(\R)}\stackrel{L\to\infty}\longrightarrow 0$, which in view of  \eqref{ortogonalth2H3} implies 
\begin{equation}\label{intermed1}
 \limsup\limits_{n\rightarrow \infty}\left|\sum_{\ell=1}^{L} E_\R(V^{(\ell)})+\|{\rm r}_n^L\|_{\dot{H}^1(\R)}^2\right|\stackrel{L\to\infty}\longrightarrow 0, \end{equation} 
 we infer from 
  \eqref{E-pos} that
  $$ E_\R(V^{(\ell)})= 0, \quad \forall\,\,\ell\geq 1,$$
  and
 $$\limsup\limits_{n\rightarrow \infty}\|{\rm r}_n^L\|_{\dot{H}^1(\R)}^2\stackrel{L\to\infty}\longrightarrow 0.$$
 By virtue of   Lemma \ref{studyfirstpartimaginary}, this ensures  that  for any $\ell\geq 1$,
 $\tilde a_{V^{(\ell)}}\equiv 1$,  and then ~$M_\R(V^{(\ell)}) \in 4\pi \N$.
  Consequently, there is a finite number $L_1 \leq   \frac M {4 \pi}$  of non-zero profiles in the decomposition\refeq{decompN}, and setting 
${\rm r}_n={\rm r}_n^{L_1} $, we have
$$\lim_{n\to\infty}\;\|{\rm
r}_n\|_{\dot{H}^1(\R)}= 0,$$
 which concludes  the proof of Lemma \ref{endkey}, and therefore also the proof of Proposition\refer{keygen} 
 \end{proof}
 \medbreak
 
 
\section{Proof of Proposition\refer{car-sabis}}\label{gen-multi}

\subsection{Structure of the conservation laws} We will need the following result from \cite{BaPeL} describing 
the structure of the conserved quantities  $E_k(u)$ introduced in \eqref{as}.

\begin{proposition}[\cite{BaPeL}]
\label{Ek}
{\sl For any~$k\in \N^*$, the conserved quantity $E_k(u)$ admits the following decomposition:
\begin{equation}\label{genmu} E_k(u)=  \sum^{k+1}_{j=1}\mu_{k, j} (u),\end{equation}
with
\begin{equation}\label{mu-0}
\mu_{k,1}(u)= \frac{i}{(-2)^{k+1}} \int_\R \xi^k |\widehat{u}(\xi)|^2 d\xi,
\end{equation}
and 
\begin{equation}\label{mu-1}
\mu_{k,j}=\int_{\R}  dx \, P_{k,j} (u, \overline u, \partial_x u, \partial_x\bar u\cdots, \partial_x^p u, \partial_x^p\overline u), \quad p=p(k,j),\,\,k\in \N^*, \,\, 2\leq j\leq k+1,
\end{equation}
where $p(k,j)= \left[\frac{k-j}2\right]+1$
and  $P_{k,j}$ is a homogeneous polynomial of degree~$2j$, with a total of $k-j+1$ derivatives in each term\footnote{ $\mu_{k,j}$ are homogeneous  with respect to the natural scaling:
\begin{equation}\label{mu-2}
 \mu_{k,j}(u_\mu)=\mu^{-k}\mu_{k,j}(u),\quad \forall \,k\geq 1, \,\,1\leq j\leq k+1.
 \end{equation}}.}
  \end{proposition}
\begin{remark}
{\sl Note that Proposition \ref{Ek} ensures that for all $\ell\geq 2$, there  exists a continuous function $F_\ell:\R_+\to \R_+$ vanishing at the origin such that 
\begin{equation}
\label{energycont}|E_{2\ell}(u)+\frac{i}{2^{2\ell+1}}\|u\|_{H^{\ell}(\R)}^2|\leq F_\ell(\|u\|_{H^{\ell-1}(\R)})\|u\|_{H^{\ell}(\R)}.
 \end{equation}}
\end{remark}
Note also that, in view of Proposition \ref{Ek}, to establish Proposition \ref{car-sabis}, it is enough to show that the following result holds.
\begin{proposition}\label{car-as-2}
 {\sl   Let $u\in H^1(\R)$ such that  $\tilde a_u\equiv 1$ and let $\ds N= \frac {M_\R(u)}{4 \pi}\geq 1$. Then~$u \in \ds H^\infty(\R)$ and there exists 
$\gamma_0, \dots \gamma_{2N-1}\in \R$ so that
 $u$ is a critical point of the functional\footnote{Note that by virtue of \eqref{relcons} and \eqref{asympek}, $E_k(u)$, $k\geq 0$, are purely imaginary.}~${E}_{2N}+ \sum\limits_{j=0}^{2N-1}\gamma_j E_j$.}
   \end{proposition}
   
   To prove Proposition\refer{car-as-2}, we first establish the result for generic  reflectionless Schwartz potentials\footnote{By generic reflectionless Schwartz potentials, we mean potentials $u\in \cS(\R)$ with $b_u\equiv0$ and with $\tilde a_u$ having only a  finite number of simple zeros in $\C_+$.}(see Proposition \ref{multis} below), and then recover the case of algebraic solitons by showing that they arise as  $H^\infty_{loc}$-limits 
   of such potentials.
   
\subsection{The case of  bright multi-solitons}\label{bright}  
This subsection  is devoted to the proof of the following proposition.
 \begin{proposition}
\label{multis}
{\sl  Let~$u_0\in\cS(\R)$ be such that 
\begin{equation}
\label{a-u0}
 \tilde a_{u_0}(\zeta)= e^{-\frac{i}2\|u_0\|^2_{L^2(\R)}}\prod^N_{j=1}  \Big( \frac {\zeta -\zeta^0_j}  {\zeta - \overline \zeta^0_j} \Big),
 \end{equation}
where~$\zeta^0_1, \dots, \zeta^0_N\in  \C_+$,  with $\zeta^0_j \neq \zeta^0_i$, for all $i \neq j$. Then, there exists a  family of real numbers $(\gamma_k)_{0 \leq k \leq 2N}$   
with $\gamma_{2N}=1$
so that setting $I(u)=\sum^{2N}_{k=0} \gamma_k E_k(u)$, we have
\begin{equation}\label{extremum}\begin{cases}
\ds \frac {\delta I(u)} {\delta u}|_{u= u_0} =0\\
\ds \frac {\delta I(u)} {\delta \overline u}|_{u= u_0} =0,
\end{cases}
\end{equation}
where  $\ds \frac {\delta I(u)} {\delta u}$, $\ds \frac {\delta I(u)} {\delta \overline u}$ stand  respectively for the functional derivatives of $I(u)$ with respect to $u, \overline u$.}\end{proposition} 
 \begin{proof} 
 This proposition is a direct consequence of the trace formulae  \eqref{asympek} and\refeq{mass1-0}. 
First, the equivalence between  the systems \eqref{sp} and  \eqref{sp1}, \eqref{sp1-1}, \eqref{sp1-2} implies that
\begin{equation}\label{bright1}
\tilde a_u(\zeta)=e^{-\frac i 2\|u\|^2_{L^2(\R)}}a_{\bf q}(\zeta), \quad   \tilde b_u(\zeta)=e^{\frac i2\big(\int_0^{+\infty}|u(y)|^2dy-\int^0_{-\infty}|u(y)|^2dy\big)} \,\frac{ b_{\bf q}(\zeta)}{2i\zeta},
\end{equation}
where $a_{\bf q}, b_{\bf q}$ are the spectral coefficients of the Zakharov-Shabat system \eqref{ap-ZS} associated to  ${\bf q}=(q_1, q_2)$ given by \eqref{sp1-1}.
This ensures that the map $u\mapsto \tilde a_u$ is real analytic from~$H^{1,1}(\R)$ to $H^\infty(\C_+)\cap C(\overline \C_+)$.
 Consequently, 
since by \eqref{a-u0}, $|\tilde a_{u_0}(\xi)|=1$ for all~$\xi \in \R$, 
 there exists a neighborhood  $\cW$ of $u_0$ in $H^{1,1}(\R)$ so that 
the following properties hold:
\begin{itemize}
\item[(i)]  for all $u\in \cW$,
the function 
$\tilde a_u(\zeta)$ has precisely $N$ simple zeros $\zeta_1(u), \dots, \zeta_N(u)$ in~$\C_+$, and these zeros are real analytic\footnote{In fact, $\zeta_j $, $j=1, \dots, N$,  are real analytic functions of $u$ in a $L^2$ neighborhood of $u_0$.} with respect to  $u\in \cW$;
\item[(ii)]  for all $u\in \cW$, 
\begin{equation}\label{a-bounds}
\frac12\leq |\tilde a_u(\xi)|\leq \frac32, \quad \forall \, \xi \in \R.
\end{equation}
\end{itemize}
In view of these properties, for $u\in \cS(\R)\cap \cW$,  the trace formulae   \eqref{asympek} and \eqref{mass1-0}  take the following form:
\begin{equation} \label{bright2}
E_k(u)=E_k^d(u)+E^c_k(u), \quad k\in \N,
\end{equation}
with
\begin{equation}\begin{split}
\label{bright3}
&E_0^d(u)=-2i\sum^{N}_{j=1} \arg \zeta_j(u), \quad E_k^d(u)=- \frac{2i}k\sum\limits_{j=1} ^N\im(\zeta_j(u))^k, \quad k\in \N^*,\\
&E_k^c(u)=\frac{i}{2\pi}\int^\infty_{-\infty}   \xi^{k-1} \ln |\tilde a_u(\xi)|^2 d\xi, \quad k\in \N. \end{split} \end{equation} 
We claim that 
\begin{equation}\label{dEkc}
\frac {\delta E_k^c}{\delta u}(u_0)=\frac {\delta E_k^c}{\delta \bar u}(u_0)=0, \quad \forall\, k\in \N.
\end{equation}
Indeed, in view of \eqref{bright1} and the   corresponding properties of $b_{\bf q}$, for all $\xi \in \R^*$, $\tilde b_u(\xi)$ is a real analytic function of 
$u\in H^{1,2}(\R)=\{v \in H^1(\R): \langle x\rangle^2 v, \langle x\rangle^2\partial_xv\in L^2(\R)\}$, 
 with\footnote{We denote by $d_u \tilde b_{u}(\xi)$ the differential of $\tilde b_{u}(\xi)$ with respect to $u$.}~$d_u \tilde b_{u}(\xi)$ bounded on $H^{1, 2}(\R)$
 uniformly with respect to $\xi\in \R^*$ and $u$ in bounded subsets of $H^{1,2}(\R)$. 
 In addition, for all $m\in \N^*$,
$$\sup\limits_{\xi\in \R}\langle \xi \rangle^{m} |\tilde b_{u}(\xi)|\lesssim_m 1,$$
uniformly with respect to $u$ in bounded subsets of $H^{1,2}(\R)\cap H^{m,1}(\R)$
where $H^{m,1}(\R)
=\{v\in H^m(\R):  \langle x \rangle \partial_x^kv\in L^2(\R), \, k=0, \dots, m\}$.
Together with \eqref{a-bounds} and \eqref{identab1}, this implies that for all  $u\in \cS(\R)\cap \cW$, $h\in \cS(\R)$ and $k\in \N$,
$$\frac {d}{dt} E_k^c(u+th)\Big|_{t=0}=-\frac i{\pi}\int_\R \frac{\xi^{k}}{|\tilde a_{u}(\xi)|^2}\re\Big (\overline{\tilde b_{u}(\xi)}\partial_t\tilde b_{u+th}(\xi)\Big|_{t=0}\Big)d\xi.$$
Since $\tilde b_{u_0}\equiv 0$, we obtain that
$$\frac {d}{dt} E_k^c(u_0+th)\Big|_{t=0}=0, \quad \forall\, h\in \cS(\R), \,\, k\in \N.$$

In view of \eqref{dEkc}, to complete the proof  of the proposition, it suffices to show  that one can   find~$(\gamma_k)_{0 \leq k \leq 2N}$   with $\gamma_{2N}=1$ so that, for all $1 \leq j \leq N$,
\begin{equation}\label{extremum1}\begin{cases}
\ds \frac {\partial \mathfrak{i}(\zeta)} {\partial {\zeta_j}}|_{\zeta= \zeta^0} =0,\\
\ds \frac {\partial \mathfrak{i}(\zeta)} {\partial {\overline \zeta_j}}|_{\zeta= \zeta^0} =0,
\end{cases}
\end{equation}
where 
\begin{equation}\label{defi} \mathfrak{i} (\zeta)= \sum^{2N}_{k=1} \frac {\gamma_k} {k} \sum^{N}_{j=1}  (\zeta^k_j - \overline\zeta^k_j) +  {2i} \gamma_0 \sum^{N}_{j=1} \arg \zeta_j , \, \, 
\zeta=(\zeta_1, \cdots, \zeta_N),\,\,\zeta^0=(\zeta^0_1, \cdots, \zeta_N^0)\in \C_+^N.
\end{equation}
Since, by assumption $\zeta_j^0 \neq 0$,  for all $1 \leq j \leq N$,  solving \eqref{extremum1} amounts to solve the following system
\begin{equation}\label{zero}\begin{cases}
\ds \sum^{2N-1}_{k=0} \gamma_k (\zeta_j^0)^{k}+  (\zeta_j^0)^{2N}=0\\
\ds \sum^{2N-1}_{k=0} \gamma_k (\overline \zeta_j^0)^{k}+  (\overline \zeta_j^0)^{2N}=0.
\end{cases}
\end{equation}
For this purpose, consider  the polynomial $Q$ of degree $2N$  defined  by 
$$ Q(X)= X^{2N}+\sum^{2N-1}_{k=0} \gamma_k X^k.$$ 
It follows from \eqref{zero} that,  \begin{equation}\label{zerofact}
Q(X)= \prod^N_{j=1} (X- \zeta_j^0) (X- \overline \zeta_j^0)\, .
\end{equation}
This uniquely determines  the coefficients $(\gamma_k)_{0 \leq k \leq 2N-1}\in \R^{2N-1}$ by means of $\zeta_j^0$, $1\leq j \leq N$,   which  completes the proof of the proposition.  
\end{proof}
\medbreak

 \begin{remark}
 \label{rem0}
{\sl 
 Denoting by $z_1, \cdots, z_{2N}$ the family of zeros $\zeta^0_1, \cdots, \zeta^0_N, \overline \zeta^0_1, \cdots, \overline \zeta^0_{2N}$ of $Q$,   one obviously has $$\gamma_k= (-1)^k \sum_{1\leq p_1< \cdots< p_{2N-k}\leq 2N} z_{p_1} \cdots z_{p_{2N-k}}.$$}
  \end{remark}
\medbreak
\subsection{Proof of  Proposition\refer{car-sabis}}\label{technicalC} 
Let $u \in H^1(\R)$ such that $\tilde a_u\equiv 1$ on $\C_+$. According to Lemma \ref{studyfirstpartimaginary}, $P_\R(u)=E_\R(u)=0$ and there exists an integer~$N$  such that $M_\R(u)=4 \pi N$. 
As mentioned above, we shall deduce the result by approximating $u$  by  generic   reflectionless Schwartz potentials. To construct such approximations, we first approximate~$u$ by potentials from 
$$\cS_{reg}=\{ v\in \cS(\R): \,\tilde a_v\,\, {\rm does\,\, not\,\, vanish\,\, on} \,\,\R_+\,\,{\rm  and }\,\,\tilde a_v
 {\rm \,\,has\,\, only \,\,simple \,\, zeros \,\,in} \,\,\C_+\}.$$ 
 The set $\cS_{reg}(\R)$ is known to be dense in~$\cS(\R) $ (see\ccite{bealscoifman,  Sulem1,  Sulem,  Sulem0,  Lee}). Therefore, there exists a sequence 
$(u_n)_{n \in \N}\subset \cS_{reg}(\R)$ that converges to $u$ in $H^1(\R)$. 

\medskip 

The following lemma provides some useful informations about the zeros of~$\tilde a_{u_n}$. 
\begin{lemma}
\label{zeroapp}
{\sl  For all $\frac\pi 2<\theta < \pi$,  there exists an integer $n(\theta)$ such that for all $n\geq n(\theta)$ one has:
\begin{enumerate}
\item $ \tilde a_{u_n} (\zeta)$ does not vanish if $ \pi-\theta \leq \arg \zeta \leq  \theta$;
\item in the angle~$\{ \zeta\in \C_+:   \, \theta  < \arg \zeta< \pi \big\}$,
$\tilde a_{u_n}$   has exactly $N$  simple zeros~$\zeta^{(n)}_1, \dots, \zeta^{(n)}_N$   and these zeros  satisfy
\begin{equation}\label{zeroN} -c_1\leq \re \zeta^{(n)}_j \leq - c_0 ,   \,  \,   \forall \, 1 \leq j \leq N,
\end{equation} for some constants $c_1>c_0>0$ independent of $n$.
 \end{enumerate}}\end{lemma}  
\begin{remark}\label{Rem01}
{\sl
 In view of \eqref{zeroN}, after   eventually passing to a subsequence, we can   assume  that 
\begin{equation}\label{zeroNlim}   \zeta^{(n)}_j \stackrel{n\to\infty}\longrightarrow    \zeta_j <0 \,,  \, \,  \forall \, 1 \leq j \leq N \, .\end{equation}} \end{remark}

\medskip 

\begin{proof}[Proof of Lemma\refer{zeroapp}] We give a direct proof of this lemma, but of course it can also be  deduce from the proofs of Lemmas \ref{lemzerosH1I0} and \ref{lemzerosH1}. Let 
$\theta\in (\frac\pi 2, \pi)$. Since $u_n \stackrel{n\to\infty}\longrightarrow u$ in $L^2$,
  combining  the stability estimate \eqref{eqstab}  with the fact that $\tilde a_u\equiv 1$, we infer that,  for all~$n$ sufficiently large, the spectral coefficient~$\tilde  a_{u_n}(\zeta)$ does not vanish in the angle~$\{\zeta\in \C_+:   \, \pi-\theta\leq\arg \zeta\leq \theta \big\}$  and
  \begin{equation}\label{E-1}
\int\limits^{+ \infty \, e^{i\theta}}_0 \frac {\tilde a'_{u_n}(s)}  {\tilde  a_{u_n}(s)} ds\stackrel{n\to\infty}\longrightarrow 0.
\end{equation}
In view of 
  Lemma \ref{intrel} i), this implies that\footnote{Recall that $n(u_n; \theta)$ denotes  the number of zeros of $\tilde a_{u_n}$ in the angle~$\{ \zeta\in \C_+:   \, \,  \theta < \arg \zeta< \pi \big\}$.} as $n\to \infty$,
    $$  n(u_n; \theta)=N+o(1),$$
   thereby showing that
  for $n$ sufficiently large,  $\tilde  a_{u_n}$ admits  exactly~$N$   simple zeros~$\zeta^{(n)}_1, \dots, \zeta^{(n)}_N$ in the angle $\{ \zeta\in \C_+:   \, \,  \theta< \arg \zeta< \pi \big\}$.
  
  \medskip 
  It remains to prove \eqref{zeroN}. The lower bound is immediate since the sequence $(u_n)$ is bounded in $H^1(\R)$. To establish the upper bound we proceed as follows. Denote by $v_{n,j}$
  the potential obtained from $u_n$ after removing the zero $\zeta_j^{(n)}$ by the corresponding  B\"acklund transformation:
  $$v_{n,j}= \cB_{\lam_j^{(n)}}  (\psi_{n,j} ) u_{n},$$
  where  $\lam^{(n)}_j= \sqrt {\zeta^{(n)}_j} \in \C_{++}$  and   $L_{u_{n} }(\lam_j^{(n)}) \psi_{n,j}=0$, with $\psi_{n,j} \in L^2(\R, \C^2)\setminus\{0\}$.  Then in view of \eqref{masszero},
  we have
  \begin{equation}\label{mass20}
 \|v_{n,j}\|_{L^2(\R)}^2= \|u_{n}\|_{L^2(\R)}^2- 4  \arg \zeta_j^{(n)}=4\pi(N-1)+o(1), \quad {\rm as}\,\, n\to \infty.
 \end{equation}
 On the other hand, accordingly to    \eqref{SLinfty} and \eqref{backtransf}, $v_{n,j}$ is given by 
 \begin{equation}\label{vnj}
v_{n,j}= -G_{\lam_j^{(n)}}^2 (\psi_{n,j})u_n+ G_{\lam_j^{(n)}} (\psi_{n,j})\cS_{\lam_j^{( n)}} (\psi_{n,j}), \end{equation} 
with
   \begin{equation}\label{basicest2}
\big|G_{\lam_j^{(n)}} (\psi_{n,j})\big| = 1 \andf \big|\cS_{\lam_j^{( n)}} (\psi_{n,j}) |  \leq 4 \im \lam_j^{(n)}.\end{equation} 
 \medskip 
The boundedness of $(u_n)$ in $L^2(\R)$ together with \eqref{mass20}, \eqref{vnj} and \eqref{basicest2}, ensures that
there exists a positive constant~$C$ such that, for all  $n$  large enough, 
 \begin{equation}\label{lpS}
 \|\cS_{\lam_j^{(n)}}(\psi_{n,j})\|_{L^2(\R)}\leq C.\end{equation}
Therefore,
  \begin{equation}\label{mass200}\begin{split}
   \|v_{n,j}\|_{L^2(\R)}^2= &\|u_{n}\|_{L^2(\R)}^2+ \|\cS_{\lam_j^{(n)}}(\psi_{n,j}\|^2_{L^2(\R)}-2\re\Big(u_n, \overline{G_{\lam_j^{(n)}} (\psi_{n,j})}\cS_{\lam_j^{( n)}} (\psi_{n,j})\Big)_{L^2(\R)}
\\ &\geq4\pi N-2\re\Big(u, \overline{G_{\lam_j^{(n)}} (\psi_{n,j})}\cS_{\lam_j^{( n)}} (\psi_{n,j})\Big)_{L^2(\R)} +o(1), \quad  {\rm as}\,\, n\to \infty.
\end{split}\end{equation}
It follows from \eqref{mass20} and \eqref{mass200} that
 \begin{equation}\label{mass201}
\re\Big(u, \overline{G_{\lam_j^{(n)}} (\psi_{n,j})}\cS_{\lam_j^{( n)}} (\psi_{n,j})\Big)_{L^2(\R)}\geq 2\pi+o(1), \quad n\to\infty.
\end{equation}
The bounds \eqref{basicest2}, \eqref{lpS} guarantee that for all $\varepsilon >0$ there exists  $C_\varepsilon >0$ so that
$$ \Big |\Big(u, \overline{G_{\lam_j^{(n)}} (\psi_{n,j})}\cS_{\lam_j^{( n)}} (\psi_{n,j})\Big)_{L^2(\R)}\Big|\leq \varepsilon +C_\varepsilon \im \lam_j^{(n)},$$
 which thanks to \eqref{mass201} implies that
 $$ \im \lam_j^{(n)}\gtrsim 1,$$
 for all $n$ sufficiently large and all $j=1, \dots, N$.
 Since for all $j=1, \dots, N$, $\arg \zeta_j^{(n)}\stackrel{n\to\infty}\longrightarrow \pi$, this gives the desired upper bound.

 \end{proof}

  \bigbreak

With Lemma\refer{zeroapp} at  hand,
we   apply the  B\"acklund transformations to
  remove   the zeros~$\zeta^{(n)}_1, \cdots, \zeta^{(n)}_N$ one by one.
This gives us   the following family of  sequences 
\begin{equation} \label{casN}  \begin{cases} u^{(0)}_{n} &=  u_{n}\, , \\  u^{(j)}_{n}&= \cB_{\lam_j^{(n)}}  (\psi_{n}^{j} ) u^{(j-1)}_{n}  \, , \, \, j=1,\cdots, N, \end{cases} \end{equation}  
 where 
  $\psi_{n}^{j} \in L^2(\R, \C^2)\setminus\{0\}$ solves
  $L_{u^{(j-1)}_{n} }(\lam_j^{(n)}) \psi_{n}^{j}=0$.
  
 The functions   $u^{(j)}_{n}$ belong to   $\cS(\R)$ and in view of  \eqref{BLinfty},  \eqref{masszero} and  \eqref{zeroNlim}, they satisfy:
  \begin{equation} \label{casNL2}  \|u^{(j)}_{n} \|_{L^{2}(\R)} \leq C \andf \|u^{(j)}_{n} \|_{L^{\infty}(\R)} \leq C, \end{equation}
  for all $n$ sufficiently large and all $j=1, \dots, N$.
 
 \smallskip  Furthermore, since
  \begin{equation}\label{en-unj}
  E_\R(u_n^{(j)})=E_\R(u_n)-16\sum\limits_{k=1}^j  \im \zeta^{(n)}_k\re \zeta_k^{(n)}\stackrel{n\to\infty}\longrightarrow 0,
  \quad \forall \,    j=1,\cdots, N,
  \end{equation}
  we obtain that
  \begin{equation} \label{casNH1}   \|u^{(j)}_{n} \|_{H^{1}(\R)} \leq C,
  \end{equation}
  for all $n$ sufficiently large and all $j=1, \dots, N$.
  
    \smallskip  Note also that
  $$\|u_n^{(N)}\|_{L^2(\R)}^2=4\pi N-4\sum\limits_{j=1}^N \arg \zeta_j^{(n)}\stackrel{n\to\infty}\longrightarrow 0,  $$
  which together with \eqref{en-unj}  implies that
  \begin{equation}\label{unN}
  \|u_n^{(N)}\|_{H^1(\R)}\stackrel{n\to\infty}\longrightarrow 0.
  \end{equation}
 \bigbreak 
 
 We next  fix $A_0$ so that 
 \begin{equation}\label{E0} 
 \int_{|y| \geq A_0} |u(y)|^2  dy \leq  \frac 1 4. 
 \end{equation}
 Then for all $n$ sufficiently large and all $j=1, \dots N$, we have
 \begin{equation}\label{E2}
 |1+{G_{\lam^{(n)}_j} (\psi^j_n)}(x)| \leq \frac 1 4, \quad \forall\, |x|\geq A_0.
 \end{equation}
 Indeed, since $$\psi^j_n (x)=c^{j, -}_n \psi_1^-(x, \lam^{(n)}_j;u^{(j-1)}_n) = c^{j, +}_n \psi_2^+(x, \lam^{(n)}_j;u^{(j-1)}_n)$$
 for some non zero constants $c^{j, \pm}_n$,
 setting $\epsilon_n= 2\sum^N_{j=1} (\pi - \arg \zeta^{(n)}_j)$ and invoking Lemma\refer{gen-propagation}, we get, for all $x\in \R$, and $j=1, \dots, N$,
 \begin{equation}\label{E3}\begin{split}
 \int^x_{- \infty}|u^{(j)}_n(y)|^2  dy +2 (\pi - \arg \overline {G_{\lam^{(n)}_j}(\psi^j_n)(x)})&\leq \int^x_{- \infty}|u_n(y)|^2  dy +2 \sum\limits_{k=1}^j(\pi - \arg \zeta^{(n)}_k)\\
 \leq& \int^x_{- \infty}|u_n(y)|^2  dy+\epsilon_n,\end{split}\end{equation}
 and 
 \begin{equation}\label{E4}\begin{split}
 \int_x^{ +\infty}|u^{(j)}_n(y)|^2  dy +2 (\pi - \arg {G_{\lam^{(n)}_j} (\psi^j_n)}(x))&\leq \int_x^{ +\infty}|u_n(y)|^2  dy +2 \sum\limits_{k=1}^j(\pi - \arg \zeta^{(n)}_k)\\
 \leq& \int_x^{+\infty}|u_n(y)|^2  dy+\epsilon_n.\end{split}\end{equation}
 Since $\epsilon_n \stackrel{n\to\infty}\rightarrow 0$, and by construction, $\|u_n-u\|_{L^{2}(\R)} \stackrel{n\to\infty}\longrightarrow 0$, we deduce from \eqref{E0}, \eqref{E3} and\refeq{E4}  that, for
 all $n$ sufficiently large and all $j=1, \dots N$,
 $$\pi- \arg \overline {G_{\lam^{(n)}_j} (\psi^j_n)(x)}\leq \frac 14, \quad \forall\, x\leq -A_0,$$
 $$\pi-\arg G_{\lam^{(n)}_j} (\psi^j_n)(x)\leq \frac 14, \quad \forall\, x\geq A_0,$$
 which leads immediately to \eqref{E2}.
 
 \medbreak
 
 As a direct consequence of \eqref{E2}, we obtain that $\psi_n^j={\psi_{n,1}^j\choose \psi_{n,2}^j}$, $j=1, \dots , N$ , satisfy
 \begin{equation}\label{corpsi}\begin{aligned}
 |\psi_{n,1}^{j} (x) |^2- |\psi_{n,2}^{j} (x) |^2 &>0, \, \, \forall \, x \leq -A_0\\
 |\psi_{n,2}^{j} (x) |^2- |\psi_{n,1}^{j} (x) |^2 &>0, \, \, \forall \, x \geq A_0,
  \end{aligned}\end{equation}
 provided $n$ is sufficiently large.
 \medskip

 We next consider, for all  $n$ sufficiently large  and $j=1, \dots, N$, $ \check \psi_{n}^{j} ={\check\psi_{n,1}^j\choose\check \psi_{n,2}^j}$ defined by
   \begin{equation}\label{def-check}
    \check \psi_{n}^{j} =c_n \begin{pmatrix}  \frac {\overline \psi_{n,2}^{j} } {d_{\lam_j^{(n)}} (\psi_{n}^{j} )}  \\ \frac {\overline \psi_{n,1}^{j} } {d_{\overline \lam_j^{(n)}} (\psi_{n}^{j} )}\end{pmatrix}.
    \end{equation}
   In view of \eqref{corpsi}, the normalizing constant $c_n$ can be chosen so that
   \begin{equation}\label{value} {\check  \psi^{j}}_{n} (-A_0) =\begin{pmatrix}  z^{(n)}_j\\ 1 \end{pmatrix} \with  |z^{(n)}_j| < 1 \,. \end{equation} 
  Note  that thanks to \eqref{corpsi}, one also has \begin{equation}\label{eq1chek}\begin{aligned}
 |\check \psi_{n,1}^{j} (x) |^2- |\check \psi_{n,2}^{j} (x) |^2 &< 0, \, \, \forall \, x \leq -A_0\\
  |\check \psi_{n,2}^{j} (x) |^2- |\check \psi_{n,1}^{j} (x) |^2 &< 0, \, \, \forall \, x \geq A_0.
  \end{aligned}\end{equation}

    \medbreak 
   The next lemma provides  some useful  bounds for $ \check \psi_{n}^{j}$.
     \begin{lemma}\label{check-est} {\sl $ $ 
   \begin{enumerate}
   \item[(i)] There exist   positive constants~$c_2 <  c_1$ and  $C$ such that, for all $x \in \R$, $n$ sufficiently large and $1\leq j \leq N$,
   \begin{equation}\label{croissance}   c_2 e^{-C|x|} \leq  |{\check  \psi^{j}}_{n} (x)| \leq c_1 e^{C|x|}. \end{equation} 
\item[(ii)]  There exists $C>0$ such that, for all $n$ sufficiently large and
~$1\leq j \leq N$, one has 
\begin{equation}\label{eq2chek} 1 -|z^{(n)}_j|^2 \leq C\re \lam_j^{(n)}.\end{equation} 

  \end{enumerate}}
  \end{lemma}
  
  \begin{remark} The bound \eqref{croissance} allows us to improve \eqref{eq1chek} to:
 \begin{equation}\label{eq1chekR}\begin{aligned}
 |\check \psi_{n,2}^{j} (x) |^2- |\check \psi_{n,1}^{j} (x) |^2 &\gtrsim\re \lam_j^{(n)}e^{Cx}, \quad  \forall\,x\leq -A_0\\
  |\check \psi_{n,1}^{j} (x) |^2- |\check \psi_{n,2}^{j} (x) |^2 &\gtrsim \re \lam_j^{(n)}e^{-Cx}, \quad \forall\,x\geq A_0,
  \end{aligned}
  \end{equation}
 for all $n$ sufficiently large and $1\leq j \leq N$. Indeed, 
  using the identity
 \begin{equation}\label{iden-g}
 1+G_z(\eta)=\frac{2\re z |\eta|^2}{d_z(\eta)},
 \end{equation}  together with
  \eqref{E2}, \eqref{corpsi} and \eqref{def-check}, one readily obtains that
   $$\begin{aligned}
 |\check \psi_{n,2}^{j} (x) |^2- |\check \psi_{n,1}^{j} (x) |^2 &\gtrsim\re \lam_j^{(n)}|\check \psi_n^j(x)|^2, \quad \forall\, x\leq -A_0\\
  |\check \psi_{n,1}^{j} (  x) |^2- |\check \psi_{n,2}^{j} (  x) |^2 &\gtrsim \re \lam_j^{(n)}|\check \psi_n^j(x)|^2, \quad \forall\, x\geq A_0,
  \end{aligned}
  $$
  which, in view of \eqref{croissance}, implies \eqref{eq1chekR}.
  \end{remark}

\begin{proof}[Proof of Lemma \ref{check-est}]
  Since  $ \check \psi_{n}^{j}$ solves
   \begin{equation}\label{check-2}L_{u^{(j)}_{n} }(\lam_j^{(n)}) \check \psi_{n}^{j}=0,
  \end{equation}
  applying  Gronwall's lemma, and taking into account  \eqref{zeroNlim}, \eqref{casNL2} and \eqref{value}, we get\refeq{croissance}. 
  To prove \eqref{eq2chek},  we compute $ \frac d {dx} \big(|\check\psi_{n,1}^j|^2- |\check\psi_{n,2}^j|^2\big)$ using the equation
   \eqref{check-2}:
\begin{equation}\label{components-check} \frac d {dx} \big(|\check\psi_{n,1}^j|^2- |\check\psi_{n,2}^j|^2\big)= 
 2 \im \zeta _j^{(n)}|\check\psi_n^j|^2 + 2\re\lam_j^{(n)}\big(u_n^{(j)} \overline {\check \psi_{n,1}^j}\check\psi_{n,2}^j+ \overline{u_n^{(j)}}  {\check \psi_{n,1}^j}\overline{\check\psi_{n,2}^j}\big).  \end{equation}
Integrating this identity    between $-A_0$ and $A_0$ and taking into account  \eqref{casNL2} and \eqref{croissance}, we readily deduce that 
$$\big||\check \psi_{n,1}^{j} (A_0)|^2-|\check \psi_{n,2}^{j} (A_0)|^2 - (|z^{(n)}_j|^2-1) \big| \lesssim\re \lam_j^{(n)},$$
which gives \eqref{eq2chek}, in light of \eqref{value}-\eqref{eq1chek}. 
 \end{proof}
\medbreak

We  are now in position to approximate $u$ by bright multi-solitons.
With the previous notations, consider the   family of    sequences~$(v^{(j)}_{n})_{n \in \N}$,~$0 \leq j \leq N$, defined by $$ \begin{cases} v^{(N)}_{n} &=  0, \\ v^{(j-1)}_{n}&= \cB_{\lam_j^{(n)}}  (\phi_{n}^{j} )v^{(j)}_{n},\end{cases} 
$$ with  $ \phi_{n}^{j} $ solving
 \begin{equation} \label{cond}\begin{cases} L_{v^{(j)}_{n}} (\lam^{(n)}_j)  \phi^{j}_{n}=  0\, , \\  \phi^{j}_{n} (-A_0)=  {\check  \psi^{j}}_{n} (-A_0) =\begin{pmatrix}  z^{(n)}_j\\ 1 \end{pmatrix} .\end{cases} \end{equation} 
 By construction, the functions   $v^{(j)}_{n}$ belong to   $\cS(\R)$ and, in view of  \eqref{BLinfty} and \eqref{zeroNlim}, we have
\begin{equation} \label{casNSLinfty}  \|v^{(j)}_{n} \|_{L^{\infty}(\R)} \leq 4 \sum^N_{\ell= j+1} \im \lam^{(n)}_\ell \leq C, \quad \forall \, j=0, \dots, N. \end{equation}  

Arguing as for $\check \psi_n^j$, we infer that there exist   positive constants~$c_2 <  c_1$ and  $C$ such that, for all $x \in \R$, all $n$ sufficiently large    and  $1\leq j \leq N$ ,
   \begin{equation}\label{croissancetild}  c_2 e^{-C|x|} \leq  |\phi_{n}^{j} (x)| \leq c_1 e^{C|x|} \,. \end{equation}  
   In addition, using \eqref{check-2} and \eqref{cond}, one can easily check that, for all~$R\geq A_0$, there exists~$C_R>0$ such that for all $n$ large enough, there holds
\begin{equation}\label{eq3tild}  \| \phi_{n}^j- \check \psi_{n}^{j}\|_{W^{1, \infty}([-R, R])} \leq C_R \|v^{(j)}_{n}- u^{(j)}_n\|_{L^{\infty}([-R, R])}. \quad\end{equation}  

\medbreak We first prove that for all $j=1, ..., N$,  $v_n^{j)}-u_n^{(j)}\stackrel{n\to\infty}\longrightarrow  $ in $L^\infty_{loc}$.
Setting $v_{n}=v^{(0)}_{n}$,  we then get
 \begin{equation}\label{cvLinfty}v_{n} \stackrel{n\to\infty}\longrightarrow  u \quad \mbox{in} \quad L_{\rm loc}^\infty(\R).
 \end{equation}
 \begin{lemma}\label{lE3}
For all $j=1, ..., N$, there holds
$$v_{n}^{(j)}-u_n^{(j)} \stackrel{n\to\infty}\longrightarrow  0 \quad \mbox{in} \quad L_{\rm loc}^\infty(\R).$$
\end{lemma}
\begin{proof} Since $v_n^{(N)}=0$ and $\|u^{(N)}_n\|_{H^1(\R)}\stackrel{n\to\infty}\longrightarrow 0$, it is sufficient to show that
for all $R \geq A_0$, all $n$ sufficiently large, and
$1\leq j \leq N$, we have
\begin{equation}\label{eq3tildbis} \|v^{(j-1)}_{n}- u^{(j-1)}_n\|_{L^{\infty}([-R, R])} \lesssim_R \|v^{(j)}_{n}- u^{(j)}_n\|_{L^{\infty}([-R, R])}\,.\end{equation}
Since  by construction, 
$$\begin{aligned}
u^{(j-1)}_n &= - G^2_{\lam_j^{(n)}}  (\check \psi_{n}^{j}) u^{(j)}_n + G_{\lam_j^{(n)}}  (\check \psi_{n}^{j}) \cS_{\lam_j^{(n)}}  (\check \psi_{n}^{j}), \\
v^{(j-1)}_{n} &= - G^2_{\lam_j^{(n)}}  ( \phi_{n}^{j})v^{(j)}_{n}+G_{\lam_j^{(n)}}  ( \phi_{n}^{j}) \cS_{\lam_j^{(n)}}  (\phi_{n}^{j}),
  \end{aligned}$$
  where
$$\begin{aligned}
|G_{\lam_j^{(n)}}  (\check \psi_{n}^{j})|  = |G_{\lam_j^{(n)}}  (\phi_{n}^{j})|  =1\andf 
|\cS_{\lam_j^{(n)}}  (\check \psi_{n}^{j}) |, |\cS_{\lam_j^{(n)}}  (\phi_{n}^{j}) |  \leq 4 \im {\lam_j^{(n)}},
  \end{aligned}$$
    it easily follows from \eqref{zeroNlim} and \eqref{casNSLinfty} that 
\begin{equation}\label{eqindif} \begin{aligned} |v^{(j-1)}_{n}- u^{(j-1)}_n| & \lesssim |v^{(j)}_{n}- u^{(j)}_n|+|G_{\lam_j^{(n)}}  ( \phi_{n}^{j}) - G_{\lam_j^{(n)}}  (\check \psi_{n}^{j})|\\ & +|\cS_{\lam_j^{(n)}}  (\phi_{n}^{j}) - \cS_{\lam_j^{(n)}}  (\check \psi_{n}^{j})|\, .\end{aligned}\end{equation}    
To estimate $G_{\lam_j^{(n)}}  ( \check \psi_{n}^{j}) - G_{\lam_j^{(n)}}  ( \phi_{n}^{j})$,   we use \eqref{iden-g}
which implies that
  $$G_{\lam_j^{(n)}}  ( \check\psi_{n}^{j}) - G_{\lam_j^{(n)}}  ( \phi_{n}^{j})=2 \re \lam_j^{(n)} \Big( \frac {|\check  \psi^{(j)}_{n} |^2} {d_{\lam_j^{(n)}} (\check \psi_{n}^{j})} - \frac {|\phi^{(j)}_{n} |^2} {d_{\lam_j^{(n)}} ( \phi_{n}^{j})}\Big) .$$
   Since by virtue of  \eqref{croissance} and \eqref{croissancetild}, for all $R>0$, there holds
  $$ \|\check \psi_{n}^{j}\|_{L^{\infty}([-R, R])} + \| \phi_{n}^{j}\|_{L^{\infty}([-R, R])}\lesssim_R 1,$$
and $$ \Big\|\frac {\re \lam_j^{(n)} } {d_{\lam_j^{(n)}} (\check \psi_{n}^{j})}\Big\|_{L^{\infty}([-R, R])} + \Big\|\frac {\re \lam_j^{(n)} } {d_{\lam_j^{(n)}} ( \phi_{n}^{j})}\Big\|_{L^{\infty}([-R, R])} \lesssim_R 1,$$
we deduce that 
  $$ \begin{aligned} \|G_{\lam_j^{(n)}}  (\check \psi_{n}^{j}) - G_{\lam_j^{(n)}}  ( \phi_{n}^{j})&\|_{L^{\infty}([-R, R])} \lesssim_R \|\check \psi_{n}^{j}-   \phi_{n}^{j}\|_{L^{\infty}([-R, R])}\\& + \frac 1 {\re \lam_j^{(n)}} \|d_{\lam_j^{(n)}} (\check \psi_{n}^{j}) -d_{\lam_j^{(n)}} (\phi_{n}^{j})\|_{L^{\infty}([-R, R])}, \end{aligned} $$
 which together with \eqref{eq3tild}  gives
$$ \begin{aligned} \|G_{\lam_j^{(n)}}  (\check \psi_{n}^{j}) - G_{\lam_j^{(n)}}  ( \phi_{n}^{j})\|_{L^{\infty}([-R, R])} \lesssim_R \|u_n^{(j)}- v_n^{(j)}\|_{L^{\infty}([-R, R])}\\ + \frac 1 {\re \lam_j^{(n)}} \|d_{\lam_j^{(n)}} (\check\psi_{n}^{j}) -d_{\lam_j^{(n)}} ( \phi_{n}^{j})\|_{L^{\infty}([-R, R])}. \end{aligned}$$
For $d_{\lam_j^{(n)}} (\check\psi_{n}^{j}) -d_{\lam_j^{(n)}} ( \phi_{n}^{j})$, we have:
 \begin{equation*}
  \begin{aligned}
  \frac1{2\im \zeta_j^{(n)}}\frac d {dx} \big(d_{\lam_j^{(n)}} ( \check\psi_{n}^{j})&-d_{\lam_j^{(n)}} (\phi_{n}^{j})\big)=\lam_j^{(n)}(|\check\psi_{n,1}^j|^2-|\phi_{n,1}^j|^2)\\&-\bar\lam_j^{(n)}
 (|\check\psi_{n,2}^j|^2-|\phi_{n,2}^j|^2)+iu_n^{(j)}\check\psi_{n,2}^j\overline{\check\psi_{n,1}^j}-iv_n^{(j)}\phi_{n,2}^j\overline{\phi_{n,1}^j}.
 \end{aligned}
 \end{equation*}
 Invoking \eqref{eq3tild}, we readily gather that,  for all $x\in [-R, R]$, $R\geq A_0$, 
$$ \big|\frac d {dx} \big(d_{\lam_j^{(n)}} (\check \psi_{n}^{j})-d_{\lam_j^{(n)}} ( \phi_{n}^{j})\big)\big| \lesssim_R  \re \lam_j^{(n)} \|u^{(j)}_{n}- v^{(j)}_n\|_{L^{\infty}([-R, R])}.$$ 
Since $d_{\lam_j^{(n)}} (\check \psi_{n}^{j}) (-A_0)= d_{\lam_j^{(n)}} ( \phi_{n}^{j}) (-A_0)$, we deduce that $$\|d_{\lam_j^{(n)}} (\check \psi_{n}^{j}) -d_{\lam_j^{(n)}} ( \phi_{n}^{j})\|_{L^{\infty}([-R, R])}\lesssim_R \re \lam_j^{(n)} \|u^{(j)}_n-v_n^{(j)}\|_{L^{\infty}([-R, R])} \,.$$
So, for all $R\geq A_0$,  
$$ \|G_{\lam_j^{(n)}}  (\check \psi_{n}^{j}) - G_{\lam_j^{(n)}}  ( \phi_{n}^{j})\|_{L^{\infty}([-R, R])}\lesssim_R \|u^{(j)}_{n}- v^{(j)}_n\|_{L^{\infty}([-R, R])}.$$ Arguing similarly, we get 
$$ \|\cS_{\lam_j^{(n)}}  (\check \psi_{n}^{j}) - \cS_{\lam_j^{(n)}}  ( \phi_{n}^{j})\|_{L^{\infty}([-R, R])}\lesssim_R \|u^{(j)}_{n}- v^{(j)}_n\|_{L^{\infty}([-R, R])},$$ which  completes the proof of
\eqref{eq3tildbis}, and hence of Lemma \ref{lE3}.
\end{proof}

\bigbreak

The next lemma shows that the potentials $v_n^{j}$ are  reflectionless and generic.
\begin{lemma}
\label{multichecktild}
{\sl For all~$0\leq j \leq N-1$ and~$n$ sufficiently large, we have the following properties.
\begin{enumerate}
\item[(i)] The spectral coefficient $\tilde a_{v_n^{(j)}}$ has the following form:
 \begin{equation}\label{prop1} \tilde a_{v^{(j)}_{n}}(\zeta)= {e^{-\frac{i}2\|v^{(j)}_{n}\|^2_{L^2(\R)}}} \prod^N_{\ell=j +1} \frac {\zeta-\zeta^{(n)}_\ell} {\zeta- \overline \zeta^{(n)}_\ell}  \,. \end{equation} 
 As a consequence, 
 \begin{equation}\label{prop2} \|v^{(j)}_{n}\|^2_{L^{2}(\R)}= 4  \sum^N_{\ell=j +1} \arg \zeta^{(n)}_\ell .
 \end{equation}
 \item[(ii)]
  The Sobolev norms of $v^{(j)}_{n}$ are uniformly bounded:
   \begin{equation}\label{prop3}\|v^{(j)}_{n}\|_{H^{L}(\R)}\leq C_L  , \quad \forall L\, . \end{equation}\end{enumerate}}
 \end{lemma}
 \begin{proof} 
 We proceed by induction.
  By construction,  $v^{(N)}_{n}=0$. Then, it  follows from\refeq{cond} that
$$\phi^N_{n}(x)=\begin{pmatrix}  z^{(n)}_N e^{-i \zeta^{(n)}_N (x+A_0)}\\ e^{i \zeta^{(n)}_N (x+A_0)} \end{pmatrix}.$$
By  \eqref{eq2chek}, we have  $z_N^{(n)}\neq 0$, provided $n$ is large enough, and hence $ \phi^{N}_{n} (x)$ grows as~$x \to \pm \infty$. Consequently, 
$\check \phi_n^N=\begin{pmatrix}  \frac {\overline \phi_{n,2}^{N} } {d_{\lam_N^{(n)}} (\phi_{n}^{N} )} \\ \frac {\overline \phi_{n,1}^{N} } {d_{\overline \lam_N^{(n)}} (\phi_{n}^{N} )}\end{pmatrix}$ is a $L^2$ solution 
of $L_{v_n^{(N-1)}}(\lam_N^{(n)})\psi=0$. Since~$v^{(N-1)}_{n}= \cB_{\lam_N^{(n)}}  (\phi_{n}^{N} ) 0$, and therefore $ \cB_{\lam_N^{(n)}}  (\check\phi_{n}^{N} ) v^{(N-1)}_{n}=0$,
 \refeq{backtransfrela} implies that
$$ \tilde a_{v^{(N-1)}_{n}}(\zeta)=  \frac {\overline \zeta^{(n)}_N}    {\zeta^{(n)}_N}  \frac {\zeta-\zeta^{(n)}_N} {\zeta- \overline \zeta^{(n)}_N}.$$
Consequently,
$$ \|v^{(N-1)}_{n}\|^2_{L^{2}(\R)}= 4    \arg \zeta^{(n)}_N,$$
and 
$$ E_{2L}(v^{(N-1)}_{n})= - \frac i L \im (\zeta^{(n)}_N)^L, \quad \forall\, L\in \N^*,$$
which in view of \eqref{energy}, \eqref{energycont} and \eqref{casNSLinfty}, ensures that
$$\|v^{(N-1)}_{n}\|_{H^{L}(\R)}\leq C_L, \quad \forall \, L\in \N^*.$$

Now assume that the assertions of Lemma \ref{multichecktild} hold for some $j \leq N-1$, and consider~$\phi_{n}^{j}$. As for~$\check \psi_n^{j}$, one has
 \begin{equation} \label{components-phi}\frac d {dx} \big(|\phi_{n,1}^j|^2- |\phi_{n,2}^j|^2\big)= 
 2 \im \zeta _j^{(n)}|\phi_n^j|^2 + 2\re\lam_j^{(n)}\Big(v_n^{(j)} \overline { \phi_{n,1}^j}\phi_{n,2}^j+\overline{v_n^{(j)}}  \phi_{n,1}^j\overline{\phi_{n,2}^j} \Big).  \end{equation}
This shows that $ \frac d {dx} \big(|\phi_{n,1}^j|^2- |\phi_{n,2}^j|^2\big)\geq \im \zeta _j^{(n)}|\phi_n^j|^2$, provided $ |v^{(j)}_{n}| \leq{ \im \lam^{(n)}_j}$.

 \medskip Recalling that
$$\begin{aligned} \|v^{(j)}_{n}\|^2_{L^{2}(\R)} =4 \sum^N_{\ell=j +1} \arg \zeta^{(n)}_\ell \andf \|u^{(j)}_{n}\|^2_{L^{2}(\R)} =4 \sum^N_{\ell=j +1} \arg \zeta^{(n)}_\ell+ \|u^{(N)}_{n}\|^2_{L^{2}(\R)}\, ,\end{aligned}$$
with $ \|u^{(N)}_{n}\|_{L^{2}(\R)}\stackrel{n\to\infty}\longrightarrow 0$,
 we infer that  \begin{equation} \label{global}  \begin{aligned} \|v^{(j)}_{n}\|^2_{L^{2}(\R)} - \|u^{(j)}_{n}\|^2_{L^{2}(\R)} &\stackrel{n\to\infty}\longrightarrow 0,\end{aligned}\end{equation} 
 which together with Lemma \ref{lE3}
  implies that  for all $R\geq 0$, 
$$\|v^{(j)}_{n}\|^2_{L^{2}(|x|\geq R)} - \|u^{(j)}_{n}\|^2_{L^{2}(|x|\geq R)} \stackrel{n\to\infty}\longrightarrow 0 \, . $$ 
Since, in view of \eqref{E3} and  \eqref{E4}, 
 $$\|u^{(j)}_{n}\|^2_{L^{2}(|x|\geq R)} \leq \|u_n\|^2_{L^{2} (|x|\geq R)}+\epsilon (n), \quad \forall R\geq 0, $$
 we obtain
  \begin{equation} \label{L2R} \|v^{(j)}_{n}\|^2_{L^{2}(|x|\geq R)} \leq \|u\|^2_{L^{2} (|x|\geq R)}+o_R(1), \quad \mbox{as} \quad n \to \infty.\end{equation} 
  Therefore, invoklng \eqref{zeroNlim} and using that 
  $$\|v^{(j)}_{n}\|^2_{L^{\infty}(|x|\geq R)} \leq 2 \|v^{(j)}_{n}\|_{L^{2}(|x|\geq R)} \|\partial_xv^{(j)}_{n}\|_{L^{2}(|x|\geq R)}\leq C \|v^{(j)}_{n}\|_{L^{2}(|x|\geq R)},$$
  we can find $R_0\geq A_0$ so that for all $n$ large enough and all $j=1, \dots, N$,
  $$ \|v^{(j)}_{n}\|_{L^{\infty}(|x|\geq R_0)} \leq{ \im \lam^{(n)}_j}.$$
  This implies that
 \begin{equation} \label{derglobalbis}\frac d {dx} \big(|\phi^{j}_{n,1}|^2- |\phi^{j}_{n,2}|^2\big)\geq \im \zeta _j^{(n)}|\phi_n^j|^2\geq \im \zeta _j^{(n)}
 \big|| \phi^{j}_{n,1}|^2- | \phi^{j}_{n,2}|^2\big|,
  \quad \forall \, |x| \geq R_0.\end{equation}  
  Integrating this inequality, we obtain:
  \begin{equation}\label{fin1}\begin{split}
  &\big(|\phi^{j}_{n,1}|^2- |\phi^{j}_{n,2}|^2\big)(x)\geq e^{\im \zeta_j^{(n)}(x-R_0)}\big(|\phi^{j}_{n,1}|^2- |\phi^{j}_{n,2}|^2\big)(R_0), \quad \forall\, x\geq R_0,\\
  &\big(|\phi^{j}_{n,2}|^2- |\phi^{j}_{n,1}|^2\big)(x)\geq e^{-\im \zeta_j^{(n)}(x+R_0)}\big(|\phi_{n,2}|^2- |\phi^{j}_{n,1}|^2\big)(-R_0), \quad \forall\, x\leq -R_0.
  \end{split}
  \end{equation}

 Furthermore, it follows from \eqref{components-check}, \eqref{components-phi} and the bounds \eqref{casNL2}, \eqref{croissance}, \eqref{casNSLinfty}, \eqref{croissancetild} and \eqref{eq3tild}
 that, for all $|x|\leq R_0$, 
\begin{equation}\label{fin-lD4}
\begin{aligned}
\Big|\frac d {dx} \big(\big(|\phi^{j}_{n,1}|^2- | \phi^{j}_{n}|^2\big)- \big(|\check \psi^{j}_{n,1}|^2- |\check \psi^{j}_{n,2}|^2\big)\big)(x)\Big|&\lesssim \re \lam^{(n)}_j\big(| \phi^{j}_{n}- \check \psi^{j}_{n}| +|v^{(j)}_{n} - u^{(j)}_n|\big)(x)\\ &\lesssim \re \lam^{(n)}_j\|v^{(j)}_{n}- u^{(j)}_n\|_{L^{\infty}([-R_0, R_0])}.\end{aligned}
\end{equation}
Since $\phi^{j}_{n} (-A_0)=  {\check  \psi^{j}}_{n} (-A_0)$, we deduce from \eqref{fin-lD4} that
$$\Big|\big(| \phi^{}_{n,1}|^2- | \phi^{j}_{n,2}|^2\big)(\pm R_0)- \big(|\check \psi^{j}_{n,1}|^2- |\check \psi^{j}_{n,2}|^2\big)(\pm R_0)\Big|\lesssim\re \lam^{(n)}_j\|v^{(j)}_{n}- u^{(j)}_n\|_{L^{\infty}([-R_0, R_0])},$$ 
which together with\refeq{eq1chekR} and Lemma \ref{lE3} implies that
 \begin{equation}\label{fin2}
 \begin{aligned} |\phi^{j}_{n,1} (R_0)|^2- | \phi^{j}_{n,2} (R_0)|^2 >0,\\
 | \phi^{j}_{n,2}(-R_0)|^2- | \phi^{j}_{n,1}(-R_0)|^2> 0,
 \end{aligned}
 \end{equation}
 provided $n$ is sufficiently large.

 \smallskip It follows from \eqref{fin1} and  \eqref{fin2} that
 $\begin{pmatrix}  \frac {\overline \phi_{n,2}^{j} } {d_{\lam_j^{(n)}} (\phi_{n}^{j} )} \\ \frac {\overline \phi_{n,1}^{j} } {d_{\overline \lam_j^{(n)}} (\phi_{n}^{j} )}\end{pmatrix}$
 is exponentially decaying as $|x|\to \infty$. Therefore, by\refeq{backtransfrela}, we have
 $$\tilde a_{v^{(j-1)}_{n}}(\zeta)=  \frac {\overline \zeta^{(n)}_j}    {\zeta^{(n)}_j}  \frac {\zeta-\zeta^{(n)}_j} {\zeta- \overline \zeta^{(n)}_j}\tilde a_{v^{(j)}_{n}}(\zeta)=
 {e^{-\frac{i}2\|v^{(j-1)}_{n}\|^2_{L^2(\R)}}} \prod^N_{\ell=j } \frac {\zeta-\zeta^{(n)}_\ell} {\zeta- \overline \zeta^{(n)}_\ell} ,$$
 which,  in view of  \eqref{energy}, \eqref{asympek}  and   \eqref{energycont},
 completes the proof of the lemma.  \end{proof} 

\bigbreak
We are now in position to finish the proof of Proposition \ref{car-as-2}. By Lemma \ref{multichecktild} and Proposition \ref{multis}, for all $n$ sufficiently large, there exist 
$\gamma_0^{(n)}, \dots, \gamma_{2N-1}^{(n)}\in \R$ such that setting $I_n(v)=E_{2N}(v)+ \sum\limits_{k=0}^{2N-1}\gamma_k^{(n)} E_k(v)$, we have
\begin{equation}\label{end}
 \frac{\delta I_n}{\delta v}(v)\Big|_{v=v_n}=0, \quad  \frac{\delta I_n}{\delta \bar v}(v)\Big|_{v=v_n}=0.
 \end{equation}

\smallskip In view of Remarks \ref{rem0} and    \ref{Rem01}, 
$$ \gamma_k^{(n)}\stackrel{n\to\infty}\longrightarrow\gamma_k \in \C, \quad \forall\, k=1, \dots, 2N-1.$$
Furthermore by Lemmas \ref{lE3} and \ref{multichecktild}, the sequence $(v_n)$ is bounded in $H^L(\R)$ for all $L$, and converges to $u$ in $L^\infty_{loc}(\R)$.
This implies that \begin{equation*} u \in H^\infty (\R) \andf v_{n} \stackrel{n\to\infty}\longrightarrow u \quad \mbox{in} \quad H_{\rm loc}^L(\R), \quad \forall L.\end{equation*} 
Passing to the the limit $n\to \infty$ in \eqref{end},  we obtain:
\begin{equation*}
 \frac{\delta I}{\delta v}(v)\Big|_{v=u}=0, \quad  \frac{\delta I}{\delta \bar v}(v)\Big|_{v=u}=0, 
 \end{equation*}
 with
  $I(v)=E_{2N}(v)+ \sum\limits_{k=0}^{2N-1}\gamma_k E_k(v)$, which concludes the proof of Proposition \ref{car-as-2}.
 

\end{document}